\documentclass[12pt]{article}
\usepackage[margin=25truemm,bottom=30truemm]{geometry}
\usepackage{amsthm,amsmath}
\usepackage{bm}
\usepackage{amssymb}
\usepackage{listings}
\usepackage{graphicx}
\usepackage{here}
\usepackage[subrefformat=parens]{subcaption}
\usepackage{afterpage,array,rotating}
\usepackage{ascmac}
\usepackage{amsfonts}
\usepackage{color}
\usepackage{xcolor}
\usepackage{colortbl}
\usepackage{booktabs}
\usepackage{bm}
\usepackage{enumitem}

\theoremstyle{definition}
\newtheorem{thm}{Theorem}[section]

\newtheorem{Lemma}{Lemma}[section]
\newtheorem{rem}{Remark}[section]
\newtheorem{assum}{A}

\numberwithin{equation}{section}

\newcommand*{\inp}{\xrightarrow{\mathrm{p}}}
\newcommand*{\ind}{\xrightarrow{\mathrm{d}}}

\newcommand{\Aref}[1]{(A\ref{#1})}

\title{Adaptive Bayes estimator
 for stochastic differential equations with jumps under small noise asymptotics}
\author{
        Shuntaro Suzuki \thanks{shuntarou@akane.waseda.jp} \\
        Graduate School of Fundamental Science and Engineering, Waseda University\\
        \vspace{1pt}\\
        Takaaki Wakamatsu \thanks{wakamatsu@moegi.waseda.jp}\\
        School of Fundamental Science and Engineering, Waseda University \\
        \vspace{1pt} \\
        Yasutaka Shimizu \thanks{shimizu@waseda.jp}\\
        Department of Applied Mathmatics, Waseda University   
    }
\date{2024 Febrary 20th}

\begin{document}
\maketitle
\begin{abstract}
In this paper, we consider parameter estimation for stochastic differential equations driven by Wiener processes and compound Poisson processes. We assume unknown parameters corresponding to coefficients of the drift term, diffusion term, and jump term, as well as the Poisson intensity and the probability density function of the underlying jump.
 We propose estimators based on adaptive Bayesian estimation from discrete observations. We demonstrate the consistency and asymptotic normality of the estimators within the framework of small noise asymptotics.
\end{abstract}

{\it{Key words:}} Jump-diffusion process; small noise asymptotics; discrete observations; Bayesian inference; adaptive estimation.\\
{\it{MSC{2020}:}} \textbf{Primary 62M20};Secondary:62G20,62F15.
\section{Introduction}
Consider a stochasitc process $(X_{t}^{\theta,\varepsilon})_{t \in [0,1]}$ following
 a one-dimensional stochastic differential equation under the usual conditions for a stochastic basis $(\Omega,\mathcal{F}, (\mathcal{F}_{t})_{t\in [0,1]},\mathbb{P})$:
\begin{align}
    \left\{
        \begin{aligned}
            dX_{t}^{\theta,\varepsilon} &\,= a(X_{t}^{\theta,\varepsilon},\mu_{0}) dt + \varepsilon b(X_{t}^{\theta,\varepsilon},\sigma_{0}) dW_{t} + \varepsilon c(X_{t-}^{\theta,\varepsilon},\alpha_{0}) dJ_{t}^{\lambda} \\      
             X_{0}^{\theta,\varepsilon}  &\, =  x
            \end{aligned}
            \right.   , 
            \tag{1.1}\label{1.1}
\end{align}
where $\varepsilon > 0$, and $\Theta_{i},(i = 1,2,3)$ are smooth convex open sets on $\mathbb{R}^{d_{i}}$, with $d_{i}(i = 1,2,3)$ being natural numbers. Additionally, let $\theta_{0} = (\sigma_{0},\mu_{0},\alpha_{0})$ belong to $\Theta_{0} = \Theta_{1} \times \Theta_{2} \times \Theta_{3} \subset \mathbb{R}^{d}$, where $d = d_{1} + d_{2} + d_{3}$ is given, and the parameter space $\Theta$ is defined as $\Theta = \overline{\Theta}_{0}$. The domains of the functions $a, b, c$ are each $\mathbb{R} \times \overline{\Theta}_{i},(i = 1,2,3)$. 
where $\varepsilon >0$ is a known constant, $a: \mathbb{R}\times \overline{\Theta}_1 \to \mathbb{R}$, $b: \mathbb{R}\times \overline{\Theta}_2 \to \mathbb{R}$ and $c: \mathbb{R}\times \overline{\Theta}_3 \to \mathbb{R}$ with $\Theta_i \subset \mathbb{R}^{d_i}\ (i=1,2,3)$ are open convex subsets, respectively.  We denote by $d = d_1+d_2 + d_3$ and $\Theta := \overline{\Theta}_0$ with $\Theta_0=\Theta_1\times \Theta_2\times \Theta_3$, and assume that the true value of the parameters $\theta = (\sigma, \mu, \alpha)$, say $\theta_0 = (\sigma_0, \mu_0, \alpha_0)$, belongs to $\Theta_0: \theta_0\in \Theta_0$. 

The process $J^{\lambda} = (J_{t}^{\lambda})_{t \in [0,1]}$ is a compound Poisson process given by
\begin{align}
J_{t}^{\lambda}&:= \sum_{k = 1}^{N_{t}^{\lambda}}V_{k}, \nonumber
\end{align}
where $N^{\lambda} = (N_{t}^{\lambda})_{t \in [0,1]}$ is a Poisson process with intensity $\lambda > 0$.
 The random variables $V_{k}'s$ are independent random variables 
with a common probability density function $f_{\alpha_{0}}$, and they are independent 
of $(N_{t}^{\lambda})_{t \in [0,1]}$ and the standard Brownian motion $W = (W_{t})_{t \in [0,1]}$.

We consider 
the estimation of $\theta_0$ using adaptive Bayesian method and examine the asymptotic properties of the estimator
 under the conditions that $\varepsilon \to 0$, and $\lambda \to \infty$, as well as $h_n \to 0$ as $n \to \infty$.
These assumptions form a similar asymptotic scheme to that used in Kobayashi and Shimizu \cite{9}. 

In Bayesian estimation research for stochastic differential equations, simultaneous Bayes estimators 
and adaptive Bayes estimators for ergodic diffusion processes were proposed by
 Yoshida \cite{17}. Various studies have been conducted by Nomura and Uchida \cite{10}, Kaino and Uchida \cite{4},
  Uchida and Yoshida \cite{16}, and others. Adaptive estimators for ergodic diffusion processes were proposed by
  
 Uchida and Yoshida \cite{15}, leading to hybrid estimators such as those proposed by Kessler \cite{7}, 
  Kamatani and Uchida \cite{5}, Kaino and Uchida \cite{4}, Nomura and Uchida \cite{10}, and others.
   In parameter estimation for ergodic jump-diffusion processes, threshold estimation methods for Poissonian
 jump-diffusion processes were proposed by Shimizu and Yoshida \cite{12}, while simultaneous Bayes estimators and adaptive Bayes estimators for 
  the same model as in Shimizu and Yoshida \cite{12} were 
  proposed by Ogihara and Yoshida 
\cite{11}. Additionally, various previous studies such as Amorino and Gloter \cite{1}, 
Gloter {\rm et al} \cite{3} Shimizu \cite{14}, and Amorino and Gloter {\rm et al} \cite{2} exist.

We focused on Bayes estimator because, unlike the maximum likelihood method, 
it does not require optimization to compute the estimator. The reason for proposing adaptive Bayes estimator is to reduce the dimension of the integral variable needed to compute the estimator. By proposing a method that computes the parameters sequentially, 
we aim to reduce the dimension of the integral variable of the Bayes estimator and to make the computation of the estimate numerically easier.

The asymptotic behavior of Bayes estimators is shown by applying precise
 moment evaluation to the quasi-log-likelihood function as demonstrated in Yoshida \cite{17}.
In this study, moment evaluation is applied to the quasi-log-likelihood proposed by Kobayashi and Shimizu \cite{9}.
based on the inequalities that Kobayashi and Shimizu \cite{14} used to prove the asymptotic property of the maximum likelihood estimator.

The structure of this paper is as follows:
In Chapter 2, we discuss the definitions of necessary symbols and assumptions required. Chapter 3 elaborates on the asymptotic normality of adaptive Bayes estimators and considers specific examples of jump distributions that meet the assumptions necessary for achieving asymptotic normality. Chapter 4 conducts a numerical verification of the asymptotic
 behavior of Bayes estimators. Specifically, we generate 1000 independent discrete observations following equation (1.1), compute 1000 estimators based on each observation, and then calculate their biases and variances. The discussion on determining the optimal threshold, which is essential in threshold estimation methods, is based on Kobayashi and  Shimizu \cite{9}   and
 Shimizu \cite{13}. Chapter 5 lists the lemmas necessary for proving the asymptotic normality of Bayes estimators, while Chapter 6 presents the proofs of the lemmas mentioned in Chapter 5 and the proof of asymptotic normality for adaptive Bayes estimators.

\section{Assumptions and notations}
\subsection{Notations}
In this section, we prepare notations and make assumptions.
\begin{itemize}
\item We define the set $I_{x}$ as image of the mapping $t \mapsto x_{t}$ on $[0,1]$ and define the set $I_{x}^{\gamma}$ as follows:
\begin{align}
I_{x}^{\gamma} &:= 
\bigg\{y \in \mathbb{R} \mid \mathrm{dist}(y,I_{x}) = \inf_{x \in I_{x}}\lvert x - y \rvert < \gamma\bigg\}. \nonumber
\end{align}
Additionally, we define the set $\tilde{I}_{x}^{\gamma}$ as:
\begin{align}
\tilde{I}_{x}^{\gamma} &:= \bigg\{X_{t}^{\theta,\varepsilon} \in I_{x}^{\gamma}\quad \text{for all}\ t \in [0,1] \bigg\}. \nonumber
\end{align}

\item Let $p$ be any integer. For any random variable X and Y defined on $(\Omega,\mathcal{F})$ to $(\mathbb{R},\mathcal{B})$,
 we denote  $\mathbb{E}_{\theta_{0}}[\lvert X \rvert^{p}] \lesssim \mathbb{E}_{\theta_{0}}[\lvert Y \lvert^{p}]$ by 
\begin{align}
    &\mathbb{E}_{\theta_{0}}[\lvert X \rvert^{p}]  \leq C_{p}\mathbb{E}_{\theta_{0}}[\lvert Y \rvert^{p}], \nonumber
\end{align}
where $C_{p} > 0$ is a constant depending only on $p$.

\item Let $f$ be a function from the set $A (\subset \mathbb{R}^{p})$ to $\mathbb{R}$. For symbols $k = 1,2,3,4$, we define $\partial_{x}^{k}f$ as follows:
\begin{align}
    \partial_{x}f &:= (\partial_{x_{i}} f)_{1 \leq i \leq p}, \nonumber\\
    \partial_{x}^{2}f &:= (\partial_{x_{i}}\partial_{x_{j}}f)_{1 \leq i,j \leq p}, \nonumber\\
    \partial_{x}^{3}f &:= (\partial_{x_{i}}\partial_{x_{j}}\partial_{x_{k}}f)_{1 \leq i,j,k \leq p}, \nonumber\\
    \partial_{x}^{4}f &:= (\partial_{x_{i}}\partial_{x_{j}}\partial_{x_{k}}\partial_{x_{l}}f)_{1 \leq i,j,k,l \leq p}. \nonumber 
\end{align}

\item Let $R : \mathbb{R} \times (0,1]  \to \mathbb{R}$, be a function
satisfying that there exist a constant value $C > 0$ such that for any sequence $\{u_{n}\}_{n = 1}^{\infty}$ ,
\begin{align}
    |R(x,u_{n})| \lesssim u_{n} (1 + |x|)^{C}. \nonumber
\end{align}

\item Let $\bm{n} = (n_{1},\dots,n_{d})$ be a multi-index with $|\bm{n}|$ as $|\bm{n}| = n_{1} + \cdots + n_{d}$,
 and denote
\begin{align}
    \partial_{\bm{\theta}}^{\bm{n}} := \partial_{\theta_{1}}^{n_{1}}\cdots \partial_{\theta_{d}}^{n_{d}}. \nonumber
\end{align}

\item Let $C^{k,l}(\mathbb{R} \times A;\mathbb{R})$ be the family of functuion $f : \mathbb{R} \times A \to \mathbb{R}$ such that $f(x,a)$
is $k,l$ times continuously differentiable with respect to $x$ and $a$ respectively. 

 \item Let $C_{\uparrow}^{k,l}(\mathbb{R} \times A;\mathbb{R})$ 
 be a family of function $f \in C^{k,l}(\mathbb{R} \times A;\mathbb{R})$ such that
\begin{align}
&\sup_{\theta \in \Theta}| \partial_{x}^{m}\partial_{\bm{\theta}}^{\bm{n}} f(x,\theta) | = R(x,1). \nonumber
\end{align}
\item Let $C^{k,l,m}(\mathbb{R} \times Y \times A;\mathbb{R})$
be a family of function $f : \mathbb{R} \times Y \times A \to \mathbb{R}$ such that $f(x,y,\alpha)$ is
$k,l$ times continuously differentiable with respect to $x,y$ and $\alpha$.

\item Let $p$ be any nutural number and 
$v$ be a $\mathbb{R}^{p}$-valued vector.
Set $A := (a_{i})_{1 \leq i \leq p}, B := (b_{i,j})_{1 \leq i,j \leq p}$
 and $C := (c_{i,j,k})_{1 \leq i,j,k \leq p}$ as a 1, 2 and 3 - dimensional tensors, respectively. 
 The symbols $A[v]$, $B[v^{\otimes 2}]$ and $C[v^{\otimes 3}]$ define 
\begin{align}
    A[v] &:= \sum_{i = 1}^{p}a_{i}v_{i}, \nonumber\\
    B[v^{\otimes 2}] &:= \sum_{1 \leq i,j \leq p}b_{ij}v_{i}v_{j}, \nonumber\\
    C[v^{\otimes 3}] &:= \sum_{1 \leq i,j,k \leq p}c_{ijk}v_{i}v_{j}v_{k}, \nonumber 
\end{align}
respectively.
\end{itemize}

\subsection{Assumption}
We make the following assumptions:

\assum \label{a1} For $k = 0,1,...,4$, $\partial_{\mu}^{k}a(x,\mu_{0})$, $\partial_{\sigma}^{k}b(x,\sigma_{0})$, and $\partial_{\alpha}^{k}c(x,\alpha_{0})$ are Lipschitz continuous in $x$.
\assum \label{a2} 
\begin{align}
    &a(x,\mu) \in C_{\uparrow}^{6,4}(\mathbb{R} \times \overline{\Theta}_{1},\mathbb{R}),\quad b(x,\sigma) \in C_{\uparrow}^{6,4}(\mathbb{R} \times \overline{\Theta}_{2},\mathbb{R}),\quad c(x,\alpha) \in C_{\uparrow}^{4,1}(\mathbb{R} \times \overline{\Theta}_{3},\mathbb{R}).\nonumber
\end{align}
\assum \label{a3}
\begin{align}
    &\inf_{(x,\sigma) \in \mathbb{R} \times \overline{\Theta}_{1}} b(x,\sigma) > 0,\quad \inf_{(x,\alpha) \in \mathbb{R} \times \overline{\Theta}_{3}} | c(x,\alpha) | > 0. \nonumber
\end{align}
\rem \label{r_a_2}
Under \Aref{a1} and \Aref{a3}, there exists a constant \( C > 0 \) such that 
\begin{align}
    &\,\lvert c(x,\alpha_{0})^{-1} - c(y,\alpha_{0})^{-1} \rvert \leq   C \lvert x - y \rvert, \nonumber \nonumber
\end{align}
and 
\begin{align}
    &\sup_{\alpha \in \overline{\Theta}_{3}}\lvert 
    \log c(x,\alpha) \rvert = R(x,1). \nonumber
\end{align}

\assum \label{a4}
For any $p \geq 0$, $f_{\alpha_{0}} : \mathbb{R} \to \mathbb{R}$ satisfies
\begin{align}
    &\,\int_{\mathbb{R}}\lvert z \rvert ^{p}f_{\alpha_{0}}(z)dz < \infty. \nonumber
\end{align}
To describe \Aref{a5}, we add some symbols. Set $E := \mathbb{R}$ or $E := \mathbb{R}^{+}$ and denote 
$\psi : \mathbb{R} \times E \times \overline{\Theta}_{3} \to \mathbb{R}$ by
\begin{align}
    &\psi(x,y,\alpha) := 
    \left\{
\begin{array}{cc}
    \log \lvert \big(1/c(x,\alpha)\big)f_{\alpha} \big(x,y/c(x,\alpha)\big) \rvert & \mathrm{if} \quad \lvert \big(1/c(x,\alpha)\big)f_{\alpha} \big(x,y/c(x,\alpha)\big) \rvert \neq 0,\\
0 & \mathrm{otherwise}. \\
\end{array}
\right.\nonumber
\end{align}
In addition, let $(x_{t})_{t \in [0,1]}$ be a solution of the following ODE with the initial value $x \in \mathbb{R}$ :
\begin{align}
    \frac{dx_{t}}{dt} = a(x_{t},\mu_{0}),\quad x_{0} = x. \label{ode}
\end{align}

\assum \label{a5}
There exist positive constants $\chi_{0},\chi_{1},\chi_{2},\chi_{3}$ such that for all $\mu \in \overline{\Theta}_{1}, \sigma \in \overline{\Theta}_{2}, \alpha \in \overline{\Theta}_{3}$,
\begin{align}
    \int_{0}^{1} -\frac{\lvert a(x_{t},\mu) - a(x_{t},\mu_{0}) \rvert^{2}}{2} dt&\,\leq - \chi_{0} \lvert \lvert \mu - \mu_{0} \rvert \rvert, \nonumber\\
    -\frac{1}{2}\int_{0}^{1}\bigg[\frac{b^{2}(x_{t},\sigma_{0})}{b^{2}(x_{t},\sigma)} + \log b^{2}(x_{t},\sigma) \bigg]dt &\,\leq - \chi_{1} \lvert \lvert \sigma - \sigma_{0} \rvert \rvert,  \nonumber\\
    -\frac{1}{2}\int_{0}^{1}\bigg[\frac{\lvert a(x_{t},\mu) - a(x_{t},\mu_{0}) \rvert^{2}}{b^{2}(x_{t},\sigma_{0})} \bigg] dt&\,\leq - \chi_{2}\lvert \lvert \mu - \mu_{0} \rvert \rvert, \nonumber
\end{align}
and 
\begin{align}
    \int_{0}^{1}\int_{E} \psi(x_{t},y,\alpha)p(x_{t},y,\alpha_{0}) dydt&\,\leq -\chi_{3} \lvert \lvert \alpha - \alpha_{0} \rvert \rvert. \nonumber
\end{align}

\assum \label{a6}
For each $\alpha \in \overline{\Theta}_{3}$, $f_{\alpha}$ is positibe and continuous on \( \mathbb{R} \). Moreover, there exist constants $C > 0,\quad q > 0$ and $\delta >0$ such that 
\begin{align}
&\sup_{(x,\alpha) \in I_{x}^{\gamma} \times \Theta_{3}} \bigg\lvert \frac{\partial \psi}{\partial y}(x,y,\alpha)\bigg\rvert \leq C(1 + \lvert y \rvert^{q}),\quad y \in \mathbb{R}. \nonumber
\end{align}

\assum \label{a7}
For each $\alpha \in \overline{\Theta}_{3}$, $f_{\alpha}$ is positibe and continuous on \( \mathbb{R}^{+} \). Moreover, let $(f_{\alpha})_{\alpha \in \overline{\Theta}_{3}}$ satify the following conditions:
\begin{enumerate}
\item[(1)] \label{a7_2_b} For constants $q \geq 0$ and $\delta > 0$ such that,
\begin{align}
&\sup_{(x,\alpha) \in I_{x}^{\gamma} \times \overline{\Theta}_{3}}\bigg \lvert \frac{\partial \psi}{\partial y}(x,y,\alpha)\bigg \rvert \leq O\bigg(\frac{1}{\lvert y \rvert^{q}}\bigg)\quad as \quad q \to 0.\nonumber
\end{align}
\item[(2)] \label{a7_2_c} There exists a positive constant $\delta > 0$ such that for any $C_{1} > 0$ and $C_{2} \geq 0$, the mapping from $\mathbb{R}$ to $I_{x}^{\gamma}$ given by
\begin{align}
&x \mapsto \int_{E} \sup_{\alpha \in \Theta_{3}}\bigg\lvert \frac{\partial \psi}{\partial y}(x,C_{1}y + C_{2},\alpha)\bigg\rvert f_{\alpha_{0}}(y)dy, \nonumber
\end{align}
is continuous.
\end{enumerate}

\assum \label{a8}
The function $\psi(x,y,\alpha) : \mathbb{R} \times E \times \overline{\Theta}_{3} \to \mathbb{R}$ belong to $C^{0,1,4}(\mathbb{R},E,\overline{\Theta}_{3})$,
 where $E = \mathbb{R}_{+}$ or $E = \mathbb{R}$. 
 For each $k = 0,1,...,4$, there exists a constant $C > 0$ such that for any multi-index $\bm{k} := (k_{1},\cdots,k_{d_{3}})$ with $\lvert \bm{k} \rvert \leq 4$,
\begin{align}
&\sup_{\alpha \in \overline{\Theta}{3}} \int_{E}\bigg\lvert \partial_{\bm{\alpha}}^{\bm{k}}\psi(x,y,\alpha) \bigg\rvert^{p}f_{\alpha_{0}}(y)dy = R(x,1), \nonumber
\end{align}
and for $k = 0,1,2$, the mapping
\begin{align}
&x \mapsto \int_{E}\partial_{\alpha}^{k}\psi(x,c(x,y),\alpha)f_{\alpha_{0}}(y)dy, \nonumber
\end{align}
is continuous on $\mathbb{R} \times \overline{\Theta}_{3}$.

\assum \label{a9}
\begin{align}
&1/n\varepsilon^{2} = O(1),\quad \varepsilon \to 0,\quad n \to \infty, \nonumber
\end{align}
and for some $\delta \in (0,1)$,
\begin{align}
&\lambda^{2}/n^{1 - \delta} = o(1),\quad \lambda \to \infty,\quad n \to \infty, \nonumber\\
&\varepsilon \lambda^{1 + \delta} = o(1),\quad \lambda \to \infty,\quad \varepsilon \to 0. \nonumber
\end{align}

\rem
In the following, if there is no declaration for the variable taking the limit,
 it is assumed that the situation with $n \to \infty, \varepsilon \downarrow 0, \lambda \to \infty$ is considered.

To desciribe \Aref{a10}, we add some symbols. For $\mu \in \overline{\Theta}_{2}$ and $\theta \in \Theta$, let $d_{2} \times d_{2}$ matrix $I_{0}(\mu)$ and the $d \times d$ matrix $I(\theta)$ be respectively given by
\begin{align}
    I_{0}(\mu)& := \left(\int_{0}^{1}\partial_{\mu_{i}}a(x_{t},\mu)\partial_{\mu_{j}}a(x_{t},\mu) dt\right)_{1 \leq i,j \leq d_{1}}, \nonumber\\
    I(\theta) & :=
    \begin{pmatrix}
    I_{1}(\theta) & 0 & 0\\
    0 & I_{2}(\theta) & 0\\
    0 & 0 & I_{3}(\theta)
    \end{pmatrix},
    \nonumber
\end{align}\\
    with
\begin{align}
    I_{1}(\theta) &= \left(\int_{0}^{1}\frac{\partial_{\mu_{i}}a(x_{t},\mu_{0})\partial_{\mu_{j}}
    a(x_{t},\mu_{0})}{b^{2}(x_{t},\sigma_{0})}dt\right)_{1 \leq i,j \leq d_{1}}, \nonumber\\
    I_{2}(\theta) &= \left(2\int_{0}^{1}\frac{\partial_{\sigma_{i}}b(x_{t},\sigma_{0})
    \partial_{\sigma_{j}}b(x_{t},\mu_{0})}{b^{2}(x_{t},\sigma_{0})}dt\right)_{1 \leq i,j \leq d_{2}}, \nonumber\\
    I_{3}(\theta) &= \left(\int_{0}^{1}\int \frac{\partial \psi}{\partial \alpha_{i}}
    \frac{\partial \psi}{\partial \alpha_{j}}(x_{t},c(x_{t},\alpha_{0})z,\alpha)f_{\alpha_{0}}(z)dzdt\right)_{1 \leq i,j \leq d_{3}}. \nonumber
\end{align}

\assum \label{a10}
If $\mu_{0} \in \Theta_{2}$ and $\theta_{0} \in \Theta_{0}$, then $I(\mu_{0})$ and $I(\theta_{0})$ are positive definite.

\section{Main results}
\subsection{Main theorem}
First, we introduce how to construct adaptive Bayes estimator. Denote $\bm{X}_{n} := \{X_{t_{k}^{n}}\}_{k = 0}^{n}$
and $p_{\theta}(\bm{x}_{n}), \bm{x}_{n} \in \mathbb{R}^{n}$
as joint probability density function of $\bm{X}_{n}$. Also, the parameter $\theta = (\sigma,\mu,\alpha)$ is assumed to follow a prior distribution $\pi(\theta)$.
In this case, the simultaneous Bayes estimator $\widehat{\theta}_{n,\varepsilon,\lambda}^{S}$ obtained by minimizing the least squares loss has log likelihood
 $\log p_{\theta}(\bm{X}_{n})$ is given below:
\begin{align}
  \widehat{\theta}_{n,\varepsilon,\lambda}^{S} &:= \frac{\int_{\overline{\Theta}}\theta \exp\{\log p_{\theta}(\bm{X}_{n})\}
\pi(\theta) d\theta}{\int_{\overline{\Theta}}
\exp\{\log p_{\theta}(\bm{X}_{n})\}
\pi(\theta)d\theta}. \nonumber
\end{align}
As in Yoshida \cite{17} and Ogiwara and Yoshida \cite{11},
an estimator in which the log likelihood is replaced 
by a contrast function is expected to show the same asymptotic behavior 
as the maximum likelihood estimator based on the quasi-log likelihood.
However, in this paper, we propose an initial Bayes estimator and an adaptive Bayes estimator
that adaptively estimate each parameter in order to reduce the cost of the integral calculations as in Uchida and Yoshida \cite{16}.

Second, we define the contrast functions $\Psi_{n,\varepsilon}^{(0)}(\mu)$, $\Psi_{n,\varepsilon}^{(1)}(\mu,\sigma)$, and $\Psi_{\lambda}^{(2)}(\alpha)$ proposed by Kobayashi and Shimizu \cite{9} and Shimizu
\cite{17} for estimating the parameters $(\sigma,\mu,\alpha)$:
\begin{align}
   \Psi_{n,\varepsilon}^{(0)}(\mu)&\,:= -\frac{1}{2}\sum_{k = 1}^{n}\bigg\{\bigg\lvert 
   \Delta_{k}^{n}X^{\theta,\varepsilon} - \frac{1}{n}a(X_{t_{k-1}^{n}}^{\theta,\varepsilon},\mu)
    \bigg\rvert^{2} \bigg/ \bigg(\frac{1}{n}\varepsilon^{2}\bigg)\bigg\} \bm{1}_{C_{k}^{n}}, \label{contrast0}\\
    \Psi_{n,\varepsilon}^{(1)}(\mu,\sigma)&\,:= - \frac{1}{2}\sum_{k = 1}^{n}\bigg\{\frac{\lvert \Delta_{k}^{n}X^{\theta,\varepsilon} - \frac{1}{n} a(X_{t_{k-1}^{n}}^{\theta,\varepsilon},
    \mu) \rvert^{2}}{\frac{1}{n}\varepsilon^{2}b^{2}(X_{t_{k-1}^{n}}^{\theta,\varepsilon},\sigma)}
     + \log b^{2}(X_{t_{k-1}^{n}}^{\theta,\varepsilon},\sigma)\bigg\}\bm{1}_{C_{k}^{n}}, \label{contrast1}\\
    \Psi_{\lambda}^{(2)}(\alpha)&\,:= \sum_{k = 1}^{n}\psi(X_{t_{k-1}^{n}}^{\theta,\varepsilon},\Delta_{k}^{n}X^{\theta,\varepsilon} / \varepsilon,\alpha)\bm{1}_{D_{k}^{n}}, \label{contrast2}
\end{align}
where $n \in \mathbb{N}$ and $k = 1,2,...,n$, and $\Delta_{k}^{n}X^{\theta,\varepsilon}$ is given by
\begin{align}
\Delta_{k}^{n}X^{\theta,\varepsilon} &:= X_{t_{k}^{n}}^{\theta,\varepsilon} - X_{t_{k-1}^{n}}^{\theta,\varepsilon}, \nonumber
\end{align}
and 
\begin{align}
    C_{k}^{n} :=&\, \left\{
        \begin{array}{ll}
            \{\lvert \Delta_{k}^{n}X^{\theta,\varepsilon} \rvert < \varepsilon v_{nk}/n^{\rho} \}&\, \mathrm{under}\,\Aref{a6}\\  
            \{ \Delta_{k}^{n}X^{\theta,\varepsilon} < \varepsilon v_{nk}/n^{\rho}\}&\,\mathrm{under}\,\Aref{a7}
        \end{array}
        \right. ,\nonumber\\
        D_{k}^{n}:=&\, \left\{
            \begin{array}{ll}
                \{\lvert \Delta_{k}^{n}X^{\theta,\varepsilon} \rvert \geq \varepsilon v_{nk}/n^{\rho} \}&\, \mathrm{under}\,\Aref{a6}\\  
                \{ \Delta_{k}^{n}X^{\theta,\varepsilon} \geq \varepsilon v_{nk}/n^{\rho}\}&\,\mathrm{under}\,\Aref{a7}
            \end{array}
            \right. ,
        \nonumber
\end{align}
where $\rho \in (0,\frac{1}{2})$ and $v_{n1},\cdots,v_{nn}$ are such that $v_{nk}$ is $\mathcal{F}_{t_{k-1}^{n}}$-measurable for each $n \in \mathbb{N}$ and $k = 1,2,...,n$, with positive constants $v_{1}$ and $v_{2}$ satisfying $0 < v_{1} \leq v_{nk} \leq v_{2}$.

Using the above contrast functions, we now define initial Bayes estimator $\widehat{\mu}_{n,\varepsilon}^{(0)}$ for the drift term parameter $\mu$ 
and adaptive Bayes estimator $\widehat{\theta}_{n,\varepsilon,\lambda} := (\widehat{\sigma}_{n,\varepsilon},\widehat{\mu}_{n,\varepsilon},\widehat{\alpha}_{\lambda})$ for the parameter $\theta = (\sigma,\mu,\alpha)$. At first,
 initial Bayes estimator $\widehat{\mu}_{n,\varepsilon}^{(0)}$ is defined as follows:
\begin{align}
\widehat{\mu}_{n,\varepsilon}^{(0)} &= \left(\int_{\overline{\Theta}_{2}}\mu \exp\{\Psi_{n,\varepsilon}^{(0)}(\mu)\}\pi_{2}(\mu)d\mu\right) \cdot \left(\int_{\overline{\Theta}_{2}}\exp\{\Psi_{n,\varepsilon}^{(0)}(\mu)\}\pi_{2}(\mu)d\mu\right)^{-1},\nonumber
\end{align}
and using this, adaptive Bayes estimator $\widehat{\theta}_{n,\varepsilon,\lambda} := (\widehat{\sigma}_{n,\varepsilon},\widehat{\mu}_{n,\varepsilon},\widehat{\alpha}_{\lambda})$ is given by
\begin{align}
\widehat{\sigma}_{n,\varepsilon} &= \left(\int_{\overline{\Theta}_{1}}\sigma \exp\{\Psi_{n,\varepsilon}^{(1)}(\sigma,\widehat{\mu}_{n,\varepsilon}^{(0)})\} \pi_{1}(\sigma)d\sigma\right) \cdot \left(\int_{\overline{\Theta}{1}}\exp\{\Psi_{n,\varepsilon}^{(1)}(\sigma,\widehat{\mu}_{n,\varepsilon}^{(0)}\}\pi_{1}(\sigma)d\sigma\right)^{-1},\nonumber \\
\widehat{\mu}_{n,\varepsilon} &= \left(\int_{\overline{\Theta}_{2}}\mu \exp\{\Psi_{n,\varepsilon}^{(1)}(\widehat{\sigma}_{n,\varepsilon},\mu)\}\pi_{2}(\mu)d\mu\right) \cdot \left(\int_{\overline{\Theta}_{2}}\exp\{\Psi_{n,\varepsilon}^{(1)}(\widehat{\sigma}_{n,\varepsilon},\mu)\}\pi_{2}(\mu)d\mu\right)^{-1},\nonumber \\
\widehat{\alpha}_{\lambda} &= 
\left(\int_{\overline{\Theta}_{3}}\alpha \exp\{\Psi_{\lambda}^{(2)}(\alpha)\} \pi_{3}(\alpha)d\alpha\right) \cdot \left(\int_{\overline{\Theta}_{3}}\exp\{\Psi_{\lambda}^{(2)}(\alpha)\}\pi_{3}(\alpha)d\alpha\right)^{-1},\nonumber
\end{align}
where $\pi_{1},\pi_{2},\pi_{3}$ are the prior density functions of $\sigma,\mu,\alpha$. 

\rem 
The adaptive Bayes estimator is essentially defined by using the simultaneous likelihood $H_{n,\varepsilon,\lambda}(\theta) := \Psi_{n,\varepsilon}(\sigma,\mu) + \Psi_{\lambda}(\alpha)$  as follows :
\begin{align}
    \overline{\sigma}_{n,\varepsilon} &= \left(\int_{\overline{\Theta}_{1}}\sigma \exp\{H_{n,\varepsilon,\lambda}
    (\sigma,\mu^{\star},\alpha^{\star})\} \pi_{1}(\sigma)d\sigma\right) \cdot
     \left(\int_{\overline{\Theta}{1}}\exp\{H_{n,\varepsilon,\lambda}(\sigma,\mu^{\star},\alpha^{\star})\}\pi_{1}(\sigma)d\sigma\right)^{-1},\nonumber \\
    \overline{\mu}_{n,\varepsilon} &= \left(\int_{\overline{\Theta}_{2}}\mu \exp\{H_{n,\varepsilon,\lambda}(\widehat{\sigma}_{n,\varepsilon},\mu,\alpha^{\star})\}\pi_{2}(\mu)d\mu\right) \cdot \left(\int_{\overline{\Theta}_{2}}\exp\{H_{n,\varepsilon,\lambda}(\widehat{\sigma}_{n,\varepsilon},\mu,\alpha^{\star})\}\pi_{2}(\mu)d\mu\right)^{-1},\nonumber \\
    \overline{\alpha}_{\lambda} &= 
\left(\int_{\overline{\Theta}_{3}}\alpha \exp\{H_{n,\varepsilon,\lambda}(\widehat{\sigma}_{n,\varepsilon},\widehat{\mu}_{n,\varepsilon},\alpha)\} 
\pi_{3}(\alpha)d\alpha\right) \cdot \left(\int_{\overline{\Theta}_{3}}\exp\{H_{n,\varepsilon,\lambda}(\widehat{\sigma}_{n,\varepsilon},\widehat{\mu}_{n,\varepsilon},\alpha)\}\pi_{3}(\alpha)d\alpha\right)^{-1},\nonumber
    \end{align}
where $\mu^{\star},\alpha^{\star}$ is any dummy of $\mu \in \overline{\Theta}_{2},\sigma \in \overline{\Theta}_{3}$. However, for technical reasons, when considering the small noise framework,
 it is necessary to assign an estimator of moment convergence to $\mu_{0}$ to $H(\sigma,\mu,\alpha)$ for the part of $\mu$ when computing $\overline{\sigma}_{n,\varepsilon}$ at the beginning. 
Therefore, it is necessary to propose an initial Bayes estimator $\widehat{\mu}_{n,\varepsilon}$
 to estimate $\mu$ separately. For this purpose, we constructed an
  initial Bayes estimator $\widehat{\mu}_{n,\varepsilon}^{(0)}$ based on the least-squares error function $\Psi_{n,\varepsilon}^{(0)}$ proposed by Shimizu \cite{14}.

In this paper, we aim to demonstrate the asymptotic normality of $\widehat{\theta}_{n,\varepsilon,\lambda}$.

\begin{thm}\label{thm1}
Suppose \Aref{a1} - \Aref{a5}, \Aref{a8} - \Aref{a10}, and either \Aref{a6} or \Aref{a7}. Take $\rho$ as follows:

\begin{itemize}
    \item[(i)] Under \Aref{a6}, $\rho \in (0,\frac{1}{2})$.
    \item[(ii)] Under \Aref{a7}, $\rho \in (0,\frac{1}{4q})$ where $q$ is the constant in \Aref{a7} (2).
\end{itemize}
Then,
\begin{align}
\varepsilon^{-1}(\widehat{\mu}_{n,\varepsilon}^{(0)} - \mu_{0}) \ind \mathcal{N}(0,I_{0}(\mu_{0})^{-1}), \nonumber\\
\begin{pmatrix}
\sqrt{n}(\widehat{\sigma}_{n,\varepsilon} - \sigma_{0})\\
\varepsilon^{-1}(\widehat{\mu}_{n,\varepsilon} - \mu_{0})\\
\sqrt{\lambda}(\widehat{\alpha}_{\lambda} - \alpha_{0})
\end{pmatrix}
\ind \mathcal{N}\big(0,I(\theta_{0})^{-1}\big), \nonumber
\end{align}
as $n \to \infty$, $\varepsilon \to 0$, $\lambda \to \infty$, and $\lambda \int_{\lvert z \rvert \leq 4v_{2}/cn^{\rho}}f_{\alpha_{0}}(z)dz \to 0$, where 
$c = \sup_{t \in [0,1]}\lvert c(x_{t},\alpha_{0})\rvert.$

\end{thm}
\rem \label{rem_th_3_1}
Unlike Kobayashi and Shimizu \cite{9}, we imposed a condition 
that there exists $\delta \in (0,1)$ such that $\lambda^{2}/n^{1 - \delta} = o(1)$ under \Aref{a8} and \Aref{a9},
 which affects the conditions on the threshold $\rho$. However, as is stated in Section 3 of Kobayashi and Shimizu \cite{9},
  the value of the threshold is essentially dependent on the value of $q$.
   Therefore, the constraints on the threshold essentially depends on the same as those in Kobayashi and Shimizu \cite{9}.
\subsection{Examples}
In this section, we shall give examples of jump distributions that satisfy \Aref{a6} - \Aref{a8} under \Aref{a1} - \Aref{a5}.
We omit the discussion of \Aref{a7} and \Aref{a8} as it is addressed in Chapter 5 of Kobayashi and Shimizu \cite{9}. 
We make the simplifying assumption that $c(x,\alpha) > 0,\quad (x,\alpha) \in \mathbb{R} \times \overline{\Theta}_{3}$ and that the derivatives with respect to $x$ and $\alpha$ are uniformly continuous. Note that the set $E$ is defined as $E =\mathbb{R}$ or $E = \mathbb{R}^{+}$ 
and represents the domain of $\{f_{\alpha}\}_{\alpha \in \overline{\Theta}_{3}}$. 
From \Aref{a5}, it follows that for any $(x,\alpha) \in \mathbb{R} \times \Theta_{3}$, 
$y \in E$ if and only if $y / c(x,\alpha) \in E$. We treat the normal, gamma, and inverse Gaussian distributions as concrete examples of jump distributions, which are discussed in Examples 3.2.1, 3.2.2, and 3.2.3, respectively.
Similarly, it can be confirmed that Weibull distribution and log-normal distribution also satisfies \Aref{a8}.
\subsubsection{Normal Distribution}\label{ex1}
Let $\Theta_{3}$ be a subset of $\mathbb{R} \times \mathbb{R}_{+} \times \mathbb{R}^{d_{2} - 2}$, which is
open and convex with a smooth boundary and let $f_{\alpha}(y)$ be as follows:
\begin{align}
    f_{\alpha}(y) = &\, - \frac{1}{\sqrt{2 \pi \alpha_{2}}}\exp\left\{-\frac{(y/c(x,\alpha) - \alpha_{1})^{2}}{ 2\alpha_{2}}\right\},\quad \alpha \in \overline{\Theta}_{3}. \nonumber
\end{align} 
Then,
\begin{align}
    \psi(x,y,\alpha) =&\, - \log c(x,\alpha) - \frac{1}{2}\log (2\pi \alpha_{2}) - \frac{\lvert y/c(x,\alpha) - \alpha_{1} \rvert^{2}}{2\alpha_{2}},\quad (x,y,\alpha) \in \mathbb{R} \times \mathbb{R} \times \overline{\Theta}_{3}. \nonumber
\end{align}
Let us confirm \Aref{a8}. For $k = 0,1,...,4$ and $p \in \mathbb{N}$, we have
\begin{align}
   \partial_{\alpha}^{k}\psi(x,y,\alpha) =&-\partial_{\alpha}^{k} \log c(x,\alpha) - \frac{1}{2\alpha_{2}}\left\{y^{2}\cdot \partial_{\alpha}^{k}\left(c(x,\alpha)^{-2} \right)  - 2 \alpha_{1} y \partial_{\alpha}^{k}(c(x,\alpha)^{-1}) + \partial_{\alpha}^{k}(\alpha_{2}^{2})\right\}, \nonumber\\
    \partial_{\alpha_{2}}^{k}\psi(x,y,\alpha) =&\, - \frac{1}{2} \partial_{\alpha_{2}}^{k} \left(\log \alpha_{2} \right) + 
    \partial_{\alpha_{2}}^{k}(\alpha_{2}^{-1})\left\{-\frac{1}{2} 
     \left(\frac{y}{c(x,\alpha)} - \alpha_{1}\right)^{2} \right\}, \nonumber
\end{align}
and 
\begin{align}
    \partial_{\alpha_{1}}\psi(x,y,\alpha) =&\, \frac{1}{\alpha_{2}}\left(\alpha_{1} - \frac{y}{c(x,\alpha)}\right), \quad \partial_{\alpha_{1}}^{2}\psi(x,y,\alpha) = \frac{1}{\alpha_{2}},\quad \partial_{\alpha_{1}}^{3}\psi(x,y,\alpha) = \partial_{\alpha_{1}}^{4}\psi(x,y,\alpha) = 0. \nonumber
\end{align}
\Aref{a8} is satisfied, by Remark 2.1 and \Aref{a2}.
\subsubsection{Gamma Distribution}\label{ex2}
Let $\Theta_{3}$ be a subset of $\mathbb{R}_{+} \times (1,\infty) \times \mathbb{R}^{d_{3} - 2}$, which is
open and convex with a smooth boundary and let $f_{\alpha}(y)$ be as follows:
\begin{align}
    f_{\alpha}(z) & =
            \frac{1}{\Gamma(\alpha_{2})\alpha_{1}^{\alpha_{2}}}z^{\alpha_{2} -1}\exp\{- z/\alpha_{1}\}, \quad (z \geq 0),
         \nonumber
\end{align}
for $\alpha \in \overline{\Theta}_{3}$. Then,
\begin{align}
    \psi(x,y,\alpha) =&\,-\log c(x,\alpha) - \log \Gamma(\alpha_{2}) - \alpha_{2}\log \alpha_{1} + (\alpha_{2} - 1)\log y - \frac{y}{\alpha_{1}},\nonumber
\end{align}
for $(x,y,\alpha) \in \mathbb{R} \times \mathbb{R}_{+} \times \overline{\Theta}_{3}$. Let us confirm \Aref{a8}. The derivatives of $\psi(x,y,\alpha)$ on the parameters $\alpha,\alpha_{1}$ are given by
\begin{align}
    \partial_{\alpha}^{k}\psi(x,y,\alpha) =&\, \alpha_{2}\partial_{\alpha}^{k}\log c(x,\alpha)  -\frac{y}{\alpha_{1}}
    \partial_{\alpha}^{k}\left(c^{-1}(x,\alpha)\right), \nonumber\\
    \partial_{\alpha_{1}}^{k}\psi(x,y,\alpha) =&\, -\partial_{\alpha_{1}}^{k}\left( \Gamma(\alpha_{1}) \right) - \alpha_{2}\partial_{\alpha_{1}}^{k}\left(\log \alpha_{1} \right) - \frac{y}{c(x,\alpha)}\partial_{\alpha_{1}}^{k}(\alpha_{1}^{-1}), \nonumber
\end{align}
for $k = 1,2,...,4$. Therefore, we have that, for $k = 1,2,...,4$ and any $p \in \mathbb{N}$, 
\begin{align}
    &\,\int_{\mathbb{R}} \left\lvert 
    \partial_{\alpha}^{k}\left( \psi(x,y,\alpha) \right) \right\rvert^{p} f_{\alpha_{0}}(y)dy = R(x,1), \nonumber\\
    &\,\int_{\mathbb{R}} \left\lvert \partial_{\alpha_{1}}^{k}\left( \psi(x,y,\alpha) \right) \right\rvert^{p} f_{\alpha_{0}}(y)dy = R(x,1). \nonumber
\end{align}
The derivatives of $\psi(x,y,\alpha)$ on the parameters $\alpha_{2}$ are also given by
\begin{align}
    &\,\partial_{\alpha_{2}}\psi(x,y,\alpha) = \log c(x,\alpha) - \log \Gamma(\alpha_{1}) + \log y, \nonumber\\
    &\,\partial_{\alpha_{2}}^{2}\psi(x,y,\alpha) = \partial_{\alpha_{2}}^{3}\psi(x,y,\alpha) = \partial_{\alpha_{2}}^{4}\psi(x,y,\alpha) = 0 .  \nonumber
\end{align}
For any $p \in \mathbb{N}$ and $\alpha \in \overline{\Theta}_{3}$, we obtain 
\begin{align}
    &\int_{1}^{\infty}\left\lvert \log y \right\rvert^{p}f_{\alpha}(y)dy \leq \int_{1}^{\infty}\left\lvert y - 1 \right\rvert^{p}f_{\alpha}(y)dy \leq 2^{p-1} \int_{1}^{\infty}(1 + \lvert y \rvert^{p})f_{\alpha}(y)dy < \infty. \tag{3.2.1} \label{3.2.1}
\end{align}
Moreover, since $\left\lvert \log y \right\rvert^{p}f_{\alpha_{0}}(y) \to 0 \quad as \quad y \downarrow 0$, 
$\left\lvert \log y \right\rvert^{p}f_{\alpha_{0}}(y)$ is bounded on $[0,1]$ for any $\alpha \in \overline{\Theta}_{3}$.
Thus, for any $p \in \mathbb{N}$, it is easy to get that 
\begin{align}
    &\int_{0}^{1}\left\lvert \log y \right\rvert^{p}f_{\alpha}(y)dy < \infty. \tag{3.2.2} \label{3.2.2}
\end{align}
Combining \eqref{3.2.1} and \eqref{3.2.2}, we have that
\begin{align}
     &\int_{\mathbb{R}} \left\lvert \partial_{\alpha_{2}}^{k} \psi(x,y,\alpha) \right\rvert^{p} f_{\alpha_{0}}(y)dy = R(x,1). \nonumber
\end{align}
Thus, \Aref{a8} is satisfied.
\subsubsection{Inverse Gaussian Distribution}\label{ex3}
Let $\Theta_{3}$ be a subset of $\mathbb{R}_{+} \times (1,\infty) \times \mathbb{R}^{d_{3} - 2}$, which is
open and convex with a smooth boundary and let $f_{\alpha}(y)$ be as follows:
\begin{align}
    f_{\alpha}(y)&\, = \left\{ \begin{array}{ll}
            \sqrt{\frac{\alpha_{2}}{2\pi y^3}}\exp\left\{-\frac{\alpha_{2}(y - \alpha_{1})^{2}}{2 \alpha_{1}^{2}y}\right\}  &\,(y > 0),\\
            0&\, (y \leq 0).           
        \end{array}
        \right.\nonumber
\end{align}
for $\alpha \in \overline{\Theta}_{3}$. Then,
\begin{align}
    \psi(x,y,\alpha) =&\,\frac{1}{2c(x,\alpha)}\left\{\log \alpha_{2} - \log (2\pi) - 3\log y + 3\log c(x,\alpha) \right\} - \frac{\alpha_{2} \left\lvert \frac{y}{c(x,\alpha)} - \alpha_{1}\right \rvert^{2}}{2\alpha_{1}^{2}y}, \nonumber
\end{align}
for $(x,y,\alpha) \in \mathbb{R} \times \mathbb{R}_{+} \times \overline{\Theta}_{3}$.
Let us confirm \Aref{a8}. For any $p \in \mathbb{N}$ and $(\alpha_{1}, \alpha_{2}) \in \mathbb{R}_{+} \times (1,\infty)$, 
\begin{align}
    &\left(\frac{1}{y}\right)^{p}\exp\left\{- \frac{\alpha_{2}\alpha_{1}^{2}}{y}\right\} \rightarrow 0, \nonumber\\
    &\left\lvert \log y \right\rvert^{p}\exp\left\{- \frac{\alpha_{2}\alpha_{1}^{2}}{y}\right\} \rightarrow 0, \nonumber
\end{align}
as $y \downarrow 0$ and thus 
$\frac{1}{z^{p}}f_{\alpha_{0}}(z)$ and $\lvert \log z \rvert^{p} f_{\alpha_{0}}(z)$ are bounded on $[0,1]$, which implies 
\begin{align}
    &\int_{0}^{\infty}\frac{1}{z^{p}}f_{\alpha_{0}}(z)dz < \infty, \tag{3.2.3}\label{3.2.3} \\
    &\int_{0}^{\infty}\left\lvert \log z \right\rvert^{p} f_{\alpha_{0}}(z)dz < \infty. \tag{3.2.4}\label{3.2.4}
\end{align}
By applying inequalities \eqref{3.2.3} and \eqref{3.2.4} and discussing them in the same way as in example \ref{ex2},
it is confirmed that \Aref{a8} is satisfied.
\section{Numerical result}

In this section, we undertake a numerical verification of the asymptotic normality for
the initial Bayes estimator $\widehat{\mu}_{n,\varepsilon}^{(0)}$ and the adaptive Bayes estimators
 $\widehat{\sigma}_{n,\varepsilon}$, $\widehat{\mu}_{n,\varepsilon}$ and 
$\widehat{\alpha}_{\lambda}$
 in Theorem 3.1 of Section 3. Firstly, I will describe the setting of the simulation, 
 including the algorithms for initial and adaptive Bayes estimators, as well as the execution environment.
Additionally, I will address the crucial aspect of determining the threshold for jump estimation. Finally, I will conduct a discussion based on numerical result.

We consider Ornstein-Uhlenbeck (OU) process:
\begin{align}            
dX_{t}&= - \mu X_{t} dt + \varepsilon \sigma dW_{t} + \varepsilon dJ_{t}^{\lambda},\quad X_{0}^{\theta,\varepsilon} = 1.0 , \label{sde_simu} 
\end{align}
where $J_{t}^{\lambda} = \sum_{k = 1}^{N_{t}^{\lambda}}V_{\tau_{k}}$ represents
 a compound Poisson process with an intensity of $\lambda = 100$. Let $\{V_{\tau_{k}}\}_{k = 1,2,3,...}$ be a sequence of random variables
with the inverse Gaussian distribution as in Example \ref{ex3}:
\begin{align}
    f_{\alpha}(y)&\, = \left\{ \begin{array}{ll}
            \sqrt{\frac{\alpha_{2}}{2\pi y^3}}\exp\left\{-\frac{\alpha_{2}(y - \alpha_{1})^{2}}{2 \alpha_{1}^{2}y}\right\}  &\,(y > 0),\\
            0&\, (y \leq 0).           
        \end{array}
        \right.\nonumber
\end{align}
Let us denote the jump parameter $\alpha := (\alpha_{1},\alpha_{2})$. The parameter $\theta = (\sigma,\mu,\alpha_{1},\alpha_{2})$ remains unknown,
 with the true parameter set as $\theta_{0} = (1.0,2.0,1.2,0.5)$. Our parameter space is $\Theta = [0.01,50.0]^{4}$. Under combinations of $n = 1000, 2000, 5000$ and $\varepsilon = 0.1, 0.01$, we generated 1000 sample paths $\{X_{t_{k}^{n}}^{\theta,\varepsilon}\}_{k = 0}^{n}$ each following Equation \eqref{sde_simu}. Subsequently, we calculated the initial Bayes estimator $\widehat{\mu}_{n,\varepsilon}^{(0)}$ and the adaptive Bayes estimator $\widehat{\theta}_{n,\varepsilon,\lambda} = (\widehat{\sigma}_{n,\varepsilon},\widehat{\mu}_{n,\varepsilon},\widehat{\alpha}_{1,\lambda},\widehat{\alpha}_{2,\lambda})$. The means and standard deviations of the initial Bayes estimator $\widehat{\mu}_{n,\varepsilon}^{(0)}$ and the adaptive Bayes estimator $\widehat{\theta}_{n,\varepsilon,\lambda}$ are summarized in Table 1.

We computed the estimator using the MpCN method proposed by Kamatani \cite{6}. The algorithm for the MpCN method is given as follows:
\begin{itemize}
    \item Choose $\bm{z} \in \mathbb{R}^{d_{i}}$ for $i = 0,1,2,3$ with $d_{i_{0}} = d_{i_{2}}$.
    \item Generate $\bm{z}_{0} := \rho^{1/2}\bm{z} + (1 - \rho)^{1/2}\big(\lvert \lvert x \rvert \rvert^{2}/\lvert \lvert w_{1} \rvert \rvert^{2} \big)w_{2}$, where $w_{1},w_{2} \sim \mathcal{N}(0,I_{d_{i}}),\quad (i = 0,1,2,3)$.
    \item Accept $\bm{z}_{0}$ as $\bm{z}$ with probability $\min \{1,\big(p_{i}(\bm{z}_{0})\lvert \lvert \bm{z} \rvert \rvert^{2}\big)/\big(p_{i}(\bm{z})\lvert \lvert \bm{z}_{0} \rvert \rvert^{2} \big)\}$ for $i = 0,1,2,3$. Otherwise, discard $\bm{z}_{0}$.
\end{itemize}
Here, the target density $p_{i}$ is given by
\begin{align}
\rho_{0}(\mu) &:= \bigg(\exp\{\Psi_{n,\varepsilon}^{(0)}(\mu)\}\pi_{2}(\mu)\bigg)\cdot \bigg(\int_{\overline{\Theta}_{2}}\exp\{\Psi_{n,\varepsilon}^{(0)}(\mu)\}\pi_{2}(\mu)d\mu \bigg)^{-1},\nonumber\\
\rho_{1}(\sigma) &:= \bigg(\exp\{\Psi_{n,\varepsilon}^{(1)}(\sigma,\widehat{\mu}_{n,\varepsilon}^{(0)})\}\pi_{1}(\sigma)\bigg)\cdot \bigg(\int_{\overline{\Theta}_{1}}\exp\{\Psi_{n,\varepsilon}^{(1)}(\sigma,\widehat{\mu}_{n,\varepsilon})\}\pi_{1}(\sigma)d\sigma \bigg)^{-1},\nonumber\\
\rho_{2}(\mu) &:= \bigg(\exp\{\Psi_{n,\varepsilon}^{(1)}(\widehat{\sigma}_{n,\varepsilon},\mu\}\pi_{2}(\mu)\bigg)\cdot \bigg(\int_{\overline{\Theta}_{2}}\exp\{\Psi_{n,\varepsilon}^{(1)}(\widehat{\sigma}_{n,\varepsilon},\mu)\}\pi_{2}(\mu)d\mu \bigg)^{-1},\nonumber
\end{align}
and 
\begin{align}
\rho_{3}(\alpha) &:= \bigg(\exp\{\Psi_{\lambda}^{(2)}(\alpha)\}\pi_{3}(\alpha)\bigg)\cdot \bigg(\int_{\overline{\Theta}_{3}}\exp\{\Psi_{\lambda}^{(2)}(\alpha)\}\pi_{3}(\alpha)d\alpha \bigg)^{-1},\nonumber
\end{align}
where $\Psi_{n,\varepsilon}^{(0)},\Psi_{n,\varepsilon}^{(1)},\Psi_{n,\varepsilon}^{(2)}$ are defined in \eqref{contrast0},\eqref{contrast1} and \eqref{contrast2} in Section 3.1
and $\pi_{1},\pi_{2},\pi_{3}$ are uniform distribution on $[0.01,50]$. 
We set $\rho = 0.8$ and used $10^{4}$ Markov chains and $525$ burn-in iterations. 

Next, we are going to mention the environment used to perform numerical validation. Since the order of target density $\rho_{i}$ for $i = 0,1,2,3$ is $10^{-3000}$ , which is ignored in Python or R language and thus the estimators $\widehat{\mu}_{n,\varepsilon}^{0},\widehat{\sigma}_{n,\varepsilon},\widehat{\mu}_{n,\varepsilon},\widehat{\alpha}_{1,\lambda}$ and $\widehat{\alpha}_{2,\lambda}$
are not accurately approximated, we used a programming language called Julia.
\begin{table}[h]
 \caption{Sample means (with standard deviations in parentheses) of adaptive Bayes estimator $\widehat{\theta}_{n,\varepsilon,\lambda}$'s based on $1000$ sample paths 
 from the OU processes \eqref{sde_simu} with Inverse Gaussian distribution with $(\mu_{0},\sigma_{0},\alpha_{1,0},\alpha_{2,0}) = (1.0,2.0,1.2,0,5)$ and with $N_{D} = \lambda (= 100)$.}
 \label{table1}
 \centering
  \begin{tabular}{cccccc}
   \hline
  & $\varepsilon$\,\textbackslash\,$n$ & $1000$ & $2000$ & $5000$ & true\\
  \hline
$\widehat{\mu}_{n,\varepsilon}^{(0)}$&$0.1$& 1.00436&1.00141& 0.99933& 1.0   \\
     &       &(0.03582)  &(0.03754)  &(0.03337)&       \\
      &$0.01$& 1.00884   & 1.00358   &1.00047        &       \\
      &      &(0.01889)  &(0.01786)  &(0.01689)&       \\
   \hline 
$\widehat{\sigma}_{n,\varepsilon}$& $0.1$& 1.93419         & 1.99398     &   2.03236       &2.0     \\
&                  & (0.14514)      & (0.13461) & (0.13005)       &        \\
&$0.01$&    1.90790    & 1.98385        &    2.02131      &        \\
&                  &(0.13450)       &(0.11648)        &(0.10481)    &         \\
   \hline
  $\widehat{\mu}_{n,\varepsilon}$&$0.1$&1.00423 &1.00121     &0.99913& 1.0   \\
     &                   &(0.03586)  &(0.03767)    &(0.03327)&       \\
      &$0.01$&  1.00870  & 1.00349  &  1.00042        &       \\
      &                  &(0.01882) &(0.01790)  &(0.01689)&       \\
                 \hline
$\alpha_{1}$& $0.1$& 1.36909         &  1.37308        &         1.36279 &1.2     \\
            &                  &  (0.44260)          & (0.74012)        &(0.60665)   &        \\
            &0.01  & 3.33982         &      1.74292   &1.49780    &      \\
            &               &  (4.04837)             & (2.20131) &(1.15035)&         \\
   \hline
$\alpha_{2}$&0.1&0.37298        &0.44819    &0.48990     &0.5     \\
            &                  & (0.14194)        &(0.18721)   &(0.22850)            &        \\
&0.01&   0.15692      &  0.34289       &0.44299   &        \\
&                  & (0.12892)        & (0.17483)       &(0.21777)   &         \\
   \hline
  \end{tabular}
\end{table}

We are going to discuss the way to determine the optimal threshold rate. As in Shimizu \cite{13},
to estimate diffusion parameter $\sigma \in \overline{\Theta}_{1}$ and jump parameter $\alpha \in \overline{\Theta}_{3}$ accurately,
it is necessary to choose $\rho$ such that
\begin{align}
    \sum_{k = 1}^{n}\bm{1}_{D_{k}^{n}}& = \lambda. \label{opt}
\end{align}
Denote $\rho_{opt}$ as the threshold $\rho$ which satisfies the above equality. 
Since we consider OU process, we can apply the method proposed in Shimizu \cite{13} to obtain $\rho_{opt}$.
Thus, it is still possible 
to define a filter that satisfies \eqref{opt} assuming $\lambda$ is known. In the following,
 we will define filters $\{\widehat{C}_{k}^{N_{D}}\}_{k = 1}^{n}$ and $\{\widehat{D}_{k}^{N_{D}}\}_{k = 1}^{n}$ using a method different from the one proposed in Chapter 3:
\begin{align}
\widehat{C}_{k}^{N_{D}}& := \bigg\{\Delta_{k}^{n}X^{\varepsilon}\,\mathrm{is} \,\mathrm{not}\,\mathrm{contained}\,\mathrm{in}\,\mathrm{the}\,\lceil N_{D} \rceil\,\mathrm{largest}\,\mathrm{positive}\,\mathrm{numbers}\,\mathrm{of}\, 
\{\Delta_{j}^{n}X\}_{j = 1,...,n}  \bigg\}, \nonumber\\
\widehat{D}_{k}^{N_{D}}& := \bigg\{\Delta_{k}^{n}X^{\varepsilon}\,\mathrm{is} \,\mathrm{one}\,\mathrm{of}\,\mathrm{the}\,\lceil N_{D} \rceil\,\mathrm{largest}\,\mathrm{positive}\,\mathrm{numbers}\,\mathrm{of}\, 
\{\Delta_{j}^{n}X\}_{j = 1,...,n}  \bigg\}, \nonumber
\end{align}
where $N_{D}$ is any interger which the observer determines for himself. In this case, we set $N_{D} = \lambda = 100$.

The mean of $\widehat{\mu}_{n,\varepsilon}^{(0)},\widehat{\sigma}_{n,\varepsilon},\widehat{\mu}_{n,\varepsilon}$ and $\widehat{\alpha}_{2,\lambda}$ are closed to the true value $1.0,2.0,1.0$ and $0.5$ respectively. The standard deviation of $\widehat{\mu}_{n,\varepsilon},\widehat{\sigma}_{n,\varepsilon}$ go to $0$ 
as $\varepsilon \to 0, n \to \infty$ respectively. However, on the parameter $\alpha_{2}$, 
the standard deviation of $\widehat{\alpha}_{2,\lambda}$ grows larger as $n$ goes to $\infty$ 
and the mean of $\widehat{\alpha}_{1,\lambda}$ gets far 
from true value and standard deviation of $\widehat{\alpha}_{1,\lambda}$ grows larger as $n \to \infty, \varepsilon \to 0$.
The reason why it happens is expected as follows: 
The order of expectation of diffusion part and jump part is $\varepsilon/\sqrt{n}$ and $\varepsilon \lambda/n$ respectively. Thus, as $n \to \infty$ and $\varepsilon \to 0$, the jump part is too small to be distinguished from the diffusion part.  

\section{Lemma}
Let us start to add some symbols to state Lemma \ref{ideal_AN} - \ref{moment_ineq_jump}. Denote $J_{k,0}^{n},J_{k,1}^{n}$ and $J_{k,2}^{n}$ for $k = 1,2,...,n$ and $n \in \mathbb{N}$,
\begin{align}
    &J_{k,0}^{n} := \{\Delta_{k}^{n}N^{\lambda} = 0\},\quad J_{k,1}^{n} := \{\Delta_{k}^{n}N^{\lambda} = 1\},\quad  J_{k,2}^{n} := \{\Delta_{k}^{n}N^{\lambda} \geq 2\}, \nonumber
\end{align}
where $\Delta_{k}^{n}N^{\lambda} := N_{t_{k}^{n}}^{\lambda} - N_{t_{k-1}^{n}}^{\lambda}$. Based on the above symbols, we define the ``ideal" contrast functions $\tilde{\Psi}_{n,\varepsilon}^{(0)}(\mu)$, $\tilde{\Psi}(\sigma,\mu,\alpha)$,
 $\tilde{\Psi}_{n,\varepsilon}^{(1)}(\sigma,\mu)$, 
 and $\tilde{\Psi}_{\lambda}^{(2)}(\alpha)$ as follows:
\begin{align}
& \tilde{\Psi}_{n,\varepsilon}^{(0)}(\mu) := -\frac{1}{2}\sum_{k = 1}^{n}\left\{\left\lvert \Delta_{k}^{n}X^{\theta,\varepsilon} - \frac{1}{n}a(X_{t_{k-1}^{n}}^{\theta,\varepsilon},\mu)\right\rvert^{2}\middle/ \left(\frac{1}{n}\varepsilon^{2}\right)\right\}\mathbf{1}_{J_{k,0}^{n}},\nonumber\\
& \tilde{\Psi}_{n,\varepsilon,\lambda}(\sigma,\mu,\alpha):= \tilde{\Psi}_{n,\varepsilon}^{(1)}(\sigma,\mu)+ \tilde{\Psi}_{\lambda}^{(2)}(\alpha), \nonumber\\
& \tilde{\Psi}_{n,\varepsilon}^{(1)}(\sigma,\mu) := - \frac{1}{2}\sum_{k = 1}^{n} \left\{\frac{\left\lvert \Delta_{k}^{n}X^{\theta,\varepsilon} - \frac{1}{n}a(X_{t_{k-1}^{n}}^{\theta,\varepsilon},\mu) \right\rvert^{2}}{\frac{1}{n}\varepsilon^{2}b^{2}(X_{t_{k-1}^{n}}^{\theta,\varepsilon},\sigma)} + \log b^{2}(X_{t_{k-1}^{n}}^{\theta,\varepsilon},\sigma)\right\}\mathbf{1}_{J_{k,0}^{n}},\nonumber\\
& \tilde{\Psi}_{\lambda}^{(2)}(\alpha):= \sum_{k = 1}^{n} \psi(X_{t_{k-1}^{n}}^{\theta,\varepsilon},c(X_{t_{k-1}^{n}}^{\theta,\varepsilon},\alpha_{0})V_{\tau_{k}},\alpha) \mathbf{1}_{J_{k,1}^{n}}. \nonumber
\end{align}
Based on the above contrast functions $\tilde{\Psi}_{n,\varepsilon}^{(0)}(\mu)$, $\tilde{\Psi}_{n,\varepsilon}^{(1)}(\sigma,\mu)$, and $\tilde{\Psi}_{\lambda}^{(2)}(\alpha)$, we define the ``ideal" estimators $\tilde{\mu}_{n,\varepsilon}^{(0)}$, $\tilde{\sigma}_{n,\varepsilon}$, $\tilde{\mu}_{n,\varepsilon}$, and $\tilde{\alpha}_{\lambda}$ as follows:
\begin{align}
\tilde{\mu}_{n,\varepsilon}^{(0)}&= \left(\int_{\overline{\Theta}_{2}}\mu \exp\{\tilde{\Psi}_{n,\varepsilon}^{(0)}(\mu)\}\pi_{2}(\mu)d\mu\right)\cdot \left(\int_{\overline{\Theta}_{2}}\exp\{\tilde{\Psi}_{n,\varepsilon}^{(0)}(\mu)\}\pi_{2}(\mu)d\mu\right)^{-1},\nonumber\\
\tilde{\sigma}_{n,\varepsilon}&:= \left(\int_{\overline{\Theta}_{1}}\sigma \exp\{\tilde{\Psi}_{n,\varepsilon}^{(1)}(\sigma,\tilde{\mu}_{n,\varepsilon}^{(0)})\}\pi_{1}(\sigma)d\sigma\right)\cdot \left(\int_{\overline{\Theta}_{1}}\exp\{\tilde{\Psi}_{n,\varepsilon}^{(1)}(\sigma,\tilde{\mu}_{n,\varepsilon}^{(0)})\}\pi_{1}(\sigma)d\sigma\right)^{-1},\nonumber\\
\tilde{\mu}_{n,\varepsilon}&= \left(\int_{\overline{\Theta}_{2}}\mu \exp\{\tilde{\Psi}_{n,\varepsilon}(\sigma_{n,\varepsilon},\mu,\tilde{\alpha}_{\lambda})\}\pi_{2}(\mu)d\mu\right)\cdot \left(\int_{\overline{\Theta}_{2}}\exp\{\tilde{\Psi}_{n,\varepsilon}(\sigma_{n,\varepsilon},\mu,\tilde{\alpha}_{\lambda})\}\pi_{2}(\mu)d\mu\right)^{-1},\nonumber\\
\tilde{\alpha}_{\lambda}&= \left(\int_{\overline{\Theta}_{3}}\alpha \exp\{\tilde{\Psi}_{\lambda}^{(2)}(\alpha)\}\pi_{3}(\alpha)d\alpha\right)\cdot \left(\int_{\overline{\Theta}_{3}}\exp\{\tilde{\Psi}_{\lambda}^{(2)}(\alpha)\}\pi_{3}(\alpha)d\alpha\right)^{-1}. \nonumber
\end{align}
We also define the ``ideal" likelihood ratios $\tilde{Z}_{0}(u_{2};\mu_{0}),\tilde{Z}_{1}(u_{1};\sigma_{0}),\tilde{Z}_{2}(u_{2};\mu_{0})$ and $\tilde{Z}_{3}(u_{3};\alpha_{0})$ by 
\begin{align}
    & \tilde{Z}_{n,\varepsilon,\lambda}^{(0)}[u_{2}] := \exp \bigg\{\tilde{\Psi}_{n,\varepsilon}^{(0)}(\mu_{0} + \varepsilon u_{2}) - \tilde{\Psi}_{n,\varepsilon}^{(0)}(\mu_{0}) \bigg\},\quad u_{2} \in V_{\varepsilon}(\mu_{0}), \nonumber\\
    & \tilde{Z}_{n,\varepsilon,\lambda}^{(1)}[u_{1}] := \exp \bigg\{\tilde{\Psi}_{n,\varepsilon}^{(1)}(\sigma_{0} + u_{1}/\sqrt{n},\tilde{\mu}_{n,\varepsilon}) - \tilde{\Psi}_{n,\varepsilon}^{(1)}(\sigma_{0},\tilde{\mu}_{n,\varepsilon})\bigg\},\quad u_{1} \in V_{n}(\sigma_{0}), \nonumber\\
    & \tilde{Z}_{n,\varepsilon,\lambda}^{(2)}[u_{2}] := \exp \bigg\{\tilde{\Psi}_{n,\varepsilon}^{(1)}(\tilde{\sigma}_{n,\varepsilon},\mu_{0} + \varepsilon u_{2}) - \tilde{\Psi}_{n,\varepsilon}^{(1)}(\tilde{\sigma}_{n,\varepsilon},\mu_{0})\bigg\},\quad u_{2} \in V_{\varepsilon}(\mu_{0}), \nonumber\\
    & \tilde{Z}_{n,\varepsilon,\lambda}^{(3)}[u_{3}] := \exp \bigg\{\tilde{\Psi}_{\lambda}^{(2)}(\alpha_{0} + u_{3}/\sqrt{\lambda}) - \tilde{\Psi}_{\lambda}^{(2)}(\alpha_{0}) \bigg\},\quad u_{3} \in V_{\lambda}(\alpha_{0}), \nonumber
\end{align}
where for $n\in \mathbb{N},\varepsilon > 0$\, and \,$\lambda > 0$,\, the sets $V_{n}(\sigma_{0}),V_{\varepsilon}(\mu_{0}),V_{\lambda}(\alpha_{0})$ are given by 
\begin{align}
   V_{n}(\sigma_{0})&\,:= \bigg\{u_{1} \in \mathbb{R}^{d_{1}} \mid u_{1}/ \sqrt{n} + \sigma_{0} \in \overline{\Theta}_{1}\bigg\},\nonumber\\
   V_{\varepsilon}(\mu_{0})&\,:= \bigg\{u_{2} \in \mathbb{R}^{d_{2}} \mid \varepsilon u_{2} + \mu_{0} \in \overline{\Theta}_{2}\bigg\},\nonumber\\
   V_{\lambda}(\alpha_{0})&\,:= \bigg\{u_{3} \in \mathbb{R}^{d_{3}} \mid u_{3}/\sqrt{\lambda} + \alpha_{0} \in \overline{\Theta}_{3}\bigg\}.\nonumber
\end{align}
In addition, we also define the likelihood ratios $Z_{0}(u_{2};\mu_{0}), Z_{1}(u_{1};\sigma_{0}), Z_{2}(u_{2};\mu_{0})$ 
 and $Z_{3}(u_{3};\alpha_{0})$ by
\begin{align}
    & Z_{n,\varepsilon,\lambda}^{(0)}[u_{2}] := \exp \bigg\{\Psi_{n,\varepsilon}^{(0)}(\mu_{0} + \varepsilon u_{2}) - \Psi_{n,\varepsilon}^{(0)}(\mu_{0}) \bigg\}, \nonumber\\
    & Z_{n,\varepsilon,\lambda}^{(1)}[u_{1}] := \exp \bigg\{\Psi_{n,\varepsilon}^{(1)}(\sigma_{0} + u_{1}/\sqrt{n},\tilde{\mu}_{n,\varepsilon}) - \Psi_{n,\varepsilon}^{(1)}(\sigma_{0},\tilde{\mu}_{n,\varepsilon})\bigg\}, \nonumber\\
    & Z_{n,\varepsilon,\lambda}^{(2)}[u_{2}] := \exp \bigg\{\Psi_{n,\varepsilon}^{(1)}(\tilde{\sigma}_{n,\varepsilon},\mu_{0} + \varepsilon u_{2}) - \Psi_{n,\varepsilon}^{(1)}(\tilde{\sigma}_{n,\varepsilon},\mu_{0})  \bigg\}, \nonumber
\end{align}
and
\begin{align}
    & Z_{n,\varepsilon,\lambda}^{(3)}[u_{3}] := \exp \bigg\{\Psi_{\lambda}^{(2)}(\alpha_{0} + u_{3}/\sqrt{\lambda}) - \Psi_{\lambda}^{(2)}(\alpha_{0}) \bigg\}, \nonumber
\end{align}
\rem
The above contrast functions were introduced to prove the asymptotic normality of adaptive Bayes estimator
 $\widehat{\theta}_{n,\varepsilon,\lambda}$ defined in Chapter 3, and cannot be used to actually compute estimators.
This is because we are considering a situation in which estimation is performed using discrete observations $\{X_{t_{k}^{n}}^{\theta,\varepsilon}\}_{k = 0}^{n}$, 
and thus the value of the set $J_{k,0}^{n}$ is unknown.

In Section 6, we prove Theorem \ref{thm1} mainly based on Lemma \ref{error_conv_Z_and_tildeZ} and Lemma \ref{ideal_AN}. 
\begin{Lemma}\label{error_conv_Z_and_tildeZ}
    Suppose \Aref{a1}, \Aref{a3}, \Aref{a4}, \Aref{a8}, \Aref{a9} and either \Aref{a6} or \Aref{a7}. Take $\rho$ as follows:
\begin{itemize}
    \item[(i)] Under \Aref{a6}, $\rho \in (0,\frac{1}{2})$.
    \item[(ii)] Under \Aref{a7}, $\rho \in (0,\frac{1}{4q})$ where $q$ is the constant in \Aref{a7} (2).
\end{itemize}
Then, the following stochasitc convergence holds jointly:
\begin{align}
    \begin{pmatrix}
        \int_{V_{\varepsilon}(\mu_{0})}u_{2}\mathfrak{Z}_{n,\varepsilon,\lambda}^{(0)}[u_{0}]\pi_{2}(\varepsilon u_{2} + \mu_{0})du_{2}\\
        \int_{V_{\varepsilon}(\mu_{0})}\mathfrak{Z}_{n,\varepsilon,\lambda}^{(0)}[u_{0}]\pi_{2}(\varepsilon u_{2} + \mu_{0})du_{2}
    \end{pmatrix}
 &\inp 0,\label{error_conv_z0_and_tildez0}
\end{align}
and 
\begin{align}
    \begin{pmatrix}
        \int_{V_{n}(\sigma_{0})}u_{1}\mathfrak{Z}_{n,\varepsilon,\lambda}^{(1)}[u_{1}]\pi_{1}(u_{1}/\sqrt{n} + \sigma_{0})du_{1}\\
         \int_{V_{n}(\sigma_{0})}\mathfrak{Z}_{n,\varepsilon,\lambda}^{(1)}[u_{1}]\pi_{1}(u_{1}/\sqrt{n} + \sigma_{0})du_{1}
    \end{pmatrix}
    &\inp 0,\label{error_conv_z1_and_tildez1}\\
    \begin{pmatrix}
        \int_{V_{\varepsilon}(\mu_{0})}u_{2}\mathfrak{Z}_{n,\varepsilon,\lambda}^{(2)}[u_{2}]\pi_{2}(\varepsilon u_{2}+ \mu_{0})du_{2}
        \\
        \int_{V_{\varepsilon}(\mu_{0})}\mathfrak{Z}_{n,\varepsilon,\lambda}^{(2)}[u_{2}]\pi_{2}(\varepsilon u_{2}+ \mu_{0})du_{2}       
    \end{pmatrix}
    &\inp 0, \label{error_conv_z2_and_tildez2}\\
    \begin{pmatrix}
        \int_{V_{\lambda}(\alpha_{0})}u_{3}\mathfrak{Z}_{n,\varepsilon,\lambda}^{(3)}[u_{3}]\pi(u_{3}/\sqrt{\lambda} + \alpha_{0})du_{3}\\
        \int_{V_{\lambda}(\alpha_{0})}\mathfrak{Z}_{n,\varepsilon,\lambda}^{(3)}[u_{3}]\pi(u_{3}/\sqrt{\lambda} + \alpha_{0})du_{3}
    \end{pmatrix}
    &\inp 0, \label{error_conv_z3_and_tildez3} 
\end{align}
where $\mathfrak{Z}_{n,\varepsilon,\lambda}^{(i)}[u_{i}] := Z_{n,\varepsilon,\lambda}^{(i)}[u_{i}] - \tilde{Z}_{n,\varepsilon,\lambda}^{(i)}[u_{i}]$
for $i = 0,1,2,3$.
\end{Lemma}
\rem
Lemma \ref{error_conv_Z_and_tildeZ} shows that the asymptotic normality 
of the ``ideal" Bayes estimator $\tilde{\mu}_{n,\varepsilon}^{(0)}$
 and $\tilde{\theta}_{n,\varepsilon,\lambda}$, which is shown in Lemma \ref{ideal_AN}, is carried over to the initial Bayes estimator $\widehat{\mu}_{n,\varepsilon}$ and 
adaptive Bayes estimator $\widehat{\theta}_{n,\varepsilon,\lambda}$, respectively. 

 We prove Lemma \ref{error_conv_Z_and_tildeZ} based on Lemma \ref{lemma_in_kobayashi} and \ref{error_conv_contrast3}. Lemma \ref{error_conv_contrast3} follows from Lemma \ref{lemma_in_kobayashi}. 
 \begin{Lemma}(Lemma 4.7 in Kobayashi and Shimizu \cite{9})\label{lemma_in_kobayashi}
     Suppose \Aref{a1}, \Aref{a3}, \Aref{a4}, and \Aref{a9}. We further assume that $0 < \varepsilon \leq 1$, $\lambda \leq 1$, and $\varepsilon \lambda \leq 1$. Additionally, let $2 \leq p$ and $\rho \in (0, \frac{1}{2})$.
     We have the following conditional probabilities:
         \begin{align}
             \mathbb{P}(C_{k,0}^{n}\mid \mathcal{F}_{t_{k-1}^{n}}) &\,= 1 - R\bigg(X_{t_{k-1}^{n}}^{\theta,\varepsilon},\frac{1}{\varepsilon^{p}n^{p(1 - \rho)}} + \frac{1}{n^{p(1 - \rho)}}\bigg), \nonumber\\
             \mathbb{P}(D_{k,0}^{n}\mid \mathcal{F}_{t_{k-1}^{n}}) &\,= R\bigg(X_{t_{k-1}^{n}}^{\theta,\varepsilon},\frac{1}{\varepsilon^{p}n^{p(1 - \rho)}} + \frac{1}{n^{p(1 - \rho)}}\bigg), \nonumber\\
             \mathbb{P}(C_{k,1}^{n}\mid \mathcal{F}_{t_{k-1}^{n}}) &\,= \frac{\lambda}{n}\bigg\{R\bigg(X_{t_{k-1}^{n}}^{\theta,\varepsilon},\frac{1}{\varepsilon^{p}n^{p(1 - \rho)}} + \frac{1}{n^{p(1 - \rho)}}+ \frac{\lambda}{n}\bigg) + \int_{\lvert z \rvert \leq \frac{4}{c_{1}n^{\rho}}}f_{\alpha_{0}}(z)dz\bigg\}, \nonumber\\
             \mathbb{P}(D_{k,1}^{n}\mid \mathcal{F}_{t_{k-1}^{n}}) &\,= \frac{\lambda}{n}\bigg\{R\bigg(X_{t_{k-1}^{n}}^{\theta,\varepsilon},\frac{1}{\varepsilon^{p}n^{p(1 - \rho)}} + \frac{1}{n^{p(1 - \rho)}}+ \frac{\lambda}{n}\bigg) + 1 \bigg\}, \nonumber\\
             \mathbb{P}(C_{k,2}^{n}\mid \mathcal{F}_{t_{k-1}^{n}}) &\,\leq \frac{\lambda^{2}}{n^{2}},  \nonumber\\
             \mathbb{P}(D_{k,2}^{n}\mid \mathcal{F}_{t_{k-1}^{n}}) &\,\leq \frac{\lambda^{2}}{n^{2}},  \nonumber
         \end{align} 
         where $\displaystyle c_{1} := \sup_{t \in [0,1]}\lvert c(x_{t},\alpha_{0}) \rvert >0 ,\quad c_{2} :=  \sup_{t \in [0,1]}\lvert c(x_{t},\alpha_{0}) \rvert > 0$\,.
     \end{Lemma}
     
     \begin{Lemma}\label{error_conv_contrast3}
        Suppose \Aref{a1}, \Aref{a3}, \Aref{a4}, \Aref{a6}, \Aref{a8}, \Aref{a9}. Take $\rho \in (0, \frac{1}{2})$. Then,
        \begin{align}
         &\,\sup_{\alpha \in \overline{\Theta}_{3}} \bigg\lvert \Psi_{\lambda}^{(2)}(\alpha) - \tilde{\Psi}_{\lambda}^{(2)}(\alpha) \bigg\rvert = O_{p}\bigg(\frac{1}{\varepsilon n^{1 - 1/p}} + \frac{1}{n^{1/2 - 1/p}}\bigg), \nonumber
     \end{align}
     as $n \to \infty,\quad \varepsilon \downarrow 0,\quad \lambda \to \infty$ and $\lambda \int_{\lvert z \rvert \leq 4v_{2}/c_{1}n^{\rho}}f_{\alpha_{0}}(z)dz \to 0$
     where $p \in \mathbb{N}$ satisfies that $1/p < \delta$ with $\delta > 0$ in \Aref{a9}. Moreover, suppose \Aref{a7} in place of \Aref{a6} and take $\rho \in (0,\frac{1}{4q})$ where $q$ 
     is a constant in \Aref{a7}. 
     If, for any $p \geq \frac{2}{1 - 2qp}$,
             \begin{align}
                 &\varepsilon n^{1 - q\rho - 1/p} \to 0, \nonumber
             \end{align}
             then,
             \begin{align}
                 &\sup_{\alpha \in \overline{\Theta}_{3}} \bigg\lvert \Psi_{\lambda}^{(2)}(\alpha) - \tilde{\Psi}_{\lambda}^{(2)}(\alpha) \bigg\rvert = O_{p}\bigg(\dfrac{1}{n^{1 -1/p_{1} - q\rho}} + \dfrac{1}{n^{1/2 - 1/p_{1} - q\rho}} \bigg), \nonumber
             \end{align}
             as $n \to \infty, \varepsilon \to 0, \lambda \to \infty$ and $\lambda \int_{\lvert z \rvert \leq 4v_{2}/c_{1}n^{\rho}}f_{\alpha_{0}}(z)dz \to 0$.
     \end{Lemma}

     \begin{Lemma}\label{ideal_AN}
        Suppose \Aref{a1} - \Aref{a3}, \Aref{a6}, \Aref{a8} and \Aref{a10}. Then, we obtain that
        \begin{align}
            \begin{pmatrix}
                \int_{V_{\varepsilon}(\mu_{0})}u_{2}\tilde{Z}_{n,\varepsilon,\lambda}^{(0)}[u_{2}] \pi_{2}(\mu_{0} + \varepsilon u_{2})du_{2} \\
                \int_{V_{\varepsilon}(\mu_{0})}\tilde{Z}_{n,\varepsilon,\lambda}^{(0)}[u_{2}] \pi_{2}(\mu_{0} + \varepsilon u_{2})du_{2}
            \end{pmatrix}
            \ind 
            \begin{pmatrix}
                \mathcal{N}(0,I_{0}(\mu_{0}))\\
                I_{0}(\mu_{0})
            \end{pmatrix}
            ,\label{AN_ideal_initial}
        \end{align}
        and for any $p \in \mathbb{N}$,
        \begin{align}
            &\,\sup_{n,\varepsilon,\lambda}\mathbb{E}_{\theta_{0}}\big[\lvert \varepsilon^{-1}(\tilde{\mu}_{n,\varepsilon} - \mu_{0})\rvert^{p}\big] < \infty. 
             \label{moment_conv_ideal_initial} 
        \end{align}
       In addition, suppose \Aref{a1} - \Aref{a5} and \Aref{a8} - \Aref{a10}. Then, the following convergence holds: 
    \begin{align}
    \begin{pmatrix}
        \int_{V_{n}(\sigma_{0})}u_{1}\tilde{Z}_{n,\varepsilon,\lambda}^{(1)}[u_{1}]\pi_{1}(u_{1}/\sqrt{n} + \sigma_{0})du_{1}\\
        \int_{V_{n}(\sigma_{0})}\tilde{Z}_{n,\varepsilon,\lambda}^{(1)}[u_{1}]\pi_{1}(u_{1}/\sqrt{n} + \sigma_{0})du_{1} 
    \end{pmatrix}
    &\ind 
    \begin{pmatrix}
        \mathcal{N}(0,I_{1}(\sigma_{0}))\\
        I_{1}(\sigma_{0})
    \end{pmatrix}
    ,\label{AN_ideal_adaptive_sigma}\\
    \begin{pmatrix}
        \int_{V_{\varepsilon}(\mu_{0})}u_{2}\tilde{Z}_{n,\varepsilon,\lambda}^{(2)}[u_{2}]\pi_{2}(\varepsilon u_{2}+ \mu_{0})du_{2}\\
        \int_{V_{\varepsilon}(\mu_{0})}\tilde{Z}_{n,\varepsilon,\lambda}^{(2)}[u_{2}]\pi_{2}(\varepsilon u_{2}+ \mu_{0})du_{2}
    \end{pmatrix}
    &\ind 
    \begin{pmatrix}
        \mathcal{N}(0,I_{2}(\mu_{0}))\\
        I_{2}(\mu_{0})
    \end{pmatrix}
    ,\label{AN_ideal_adaptive_mu}
    \end{align}
    and
    \begin{align}
    \begin{pmatrix}
        \int_{V_{\lambda}(\alpha_{0})}u_{3}\tilde{Z}_{n,\varepsilon,\lambda}^{(3)}[u_{3}]\pi(u_{3}/\sqrt{\lambda} + \alpha_{0})du_{3}\\
        \int_{V_{\lambda}(\alpha_{0})}\tilde{Z}_{n,\varepsilon,\lambda}^{(3)}[u_{3}]\pi(u_{3}/\sqrt{\lambda} + \alpha_{0})du_{3} 
    \end{pmatrix}
    &\ind 
    \begin{pmatrix}
        \mathcal{N}(0,I_{3}(\alpha_{0}))\\
        I_{3}(\alpha_{0})
    \end{pmatrix}
    ,\label{AN_ideal_adaptive_alpha}
    \end{align}
    and these convergence holds jointly. Moreover, for any $p \in \mathbb{N}$,
      \begin{align}
       &\sup_{n,\varepsilon,\lambda}
       \mathbb{E}_{\theta_{0}}\big[\lvert
        \sqrt{n}(\tilde{\sigma}_{n,\varepsilon} - \sigma_{0})\rvert^{p}\big] < \infty,
        \label{moment_conv_adaptive_sigma}\\
        &\sup_{n,\varepsilon,\lambda}\mathbb{E}_{\theta_{0}}
        \big[\lvert \varepsilon^{-1}(\tilde{\mu}_{n,\varepsilon}
         - \mu_{0})\rvert^{p}\big] < \infty,\label{moment_conv_adaptive_mu}\\ 
     &\sup_{n,\varepsilon,\lambda}\mathbb{E}_{\theta_{0}}
     \big[\lvert \sqrt{\lambda}(\tilde{\alpha}_{\lambda} - 
     \alpha_{0})\rvert^{p}\big] < \infty \label{moment_conv_adaptive_alpha}. 
      \end{align} 
     \end{Lemma}

We prove Lemma \ref{ideal_AN} combining Theorem 10 in Yoshida \cite{17} and
Lemmas \ref{moment_est_contrast0} - \ref{moment_est_contrast3}, which follow from Lemmas \ref{approx_for_conditional_e} - \ref{moment_ineq_jump}. 
We also show the joint convergence of 
$\tilde{\mu}_{n,\varepsilon},\tilde{\sigma}_{n,\varepsilon}$ and $\tilde{\mu}_{n,\varepsilon}$
and $\tilde{\alpha}_{\lambda}$ using thereom 10 (c) in Yoshida \cite{11}, 
which is applicable when we obtain Lemmas \ref{moment_est_contrast0} - \ref{moment_est_contrast3}.

\begin{Lemma}\label{moment_est_contrast0}
    Suppose \Aref{a1} - \Aref{a3} and \Aref{a9}. Then, there exists $\eta_{0} \in (0,\frac{1}{2})$ for any $p \in \mathbb{N}$ 
    such that
    \begin{itemize}
        \item[(1)]\label{moment_est_contrast0_1} $\displaystyle \sup_{n,\varepsilon,\lambda}\mathbb{E}_{\theta_{0}}\bigg[\bigg\lvert \varepsilon 
        \partial_{\mu}\tilde{\Psi}_{n,\varepsilon}^{(0)}(\mu_{0})\bigg\rvert^{p}\bigg] < \infty,$
        \item[(2)]\label{moment_est_contrast0_2} $\displaystyle \sup_{n,\varepsilon,\lambda}\mathbb{E}_{\theta_{0}}
        \bigg[\bigg\lvert \varepsilon^{-\eta_{0}}\cdot
         \bigg\{\varepsilon^{2} \partial_{\mu}^{2} \tilde{\Psi}_{n,\varepsilon}^{(0)}(\mu_{0})
          + I_{0}(\mu_{0})\bigg\}\bigg\lvert^{p}\bigg] < \infty,$
        \item[(3)]\label{moment_est_contrast0_3} $\displaystyle \sup_{n,\varepsilon,\lambda}\sup_{\mu \in \overline{\Theta}_{2}}
        \mathbb{E}_{\theta_{0}}\bigg[\bigg\lvert \varepsilon^{2}\partial_{\mu}^{3}\tilde{\Psi}_{n,\varepsilon}^{(0)}
        (\mu)\bigg\rvert^{p}\bigg] < \infty.$
    \end{itemize}
\end{Lemma}
\begin{Lemma}\label{moment_est_contrast1}
    Suppose \Aref{a1} - \Aref{a3}, \Aref{a9} and \eqref{moment_conv_ideal_initial} holds for any $p \in \mathbb{N}$. Then, there exists $\eta_{1} \in (0,\frac{1}{2})$ for any $p \in \mathbb{N}$ such that
    \begin{enumerate}
        \item[(1)]\label{moment_est_contrast1_1} $\displaystyle \sup_{n,\varepsilon,\lambda}\mathbb{E}_{\theta_{0}}\bigg[\bigg\lvert \frac{1}{\sqrt{n}} \partial_{\sigma}\tilde{\Psi}_{n,\varepsilon}^{(1)}(\sigma_{0},\tilde{\mu}_{n,\varepsilon}^{(0)})\bigg\rvert^{p}\bigg] < \infty,$
        \item[(2)]\label{moment_est_contrast1_2} $\displaystyle \sup_{n,\varepsilon,\lambda}
        \mathbb{E}_{\theta_{0}}\bigg[\bigg\lvert n^{\eta_{1}}\cdot 
        \bigg\{\frac{1}{n} \partial_{\sigma}^{2} \tilde{\Psi}_{n,\varepsilon}^{(1)}
        (\sigma_{0},\tilde{\mu}_{n,\varepsilon}^{(0)})
         + I_{1}(\sigma_{0})\bigg\}\bigg\lvert^{p}\bigg] < \infty,$ 
        \item[(3)]\label{moment_est_contrast1_3} $\displaystyle \sup_{n,\varepsilon,\lambda}\mathbb{E}_{\theta_{0}}\bigg[\sup_{\sigma \in \overline{\Theta}_{2}}\bigg\lvert \frac{1}{n}\partial_{\sigma}^{3}\tilde{\Psi}_{n,\varepsilon}^{(1)}(\sigma,\tilde{\mu}_{n,\varepsilon}^{(0)})\bigg\rvert^{p}\bigg] < \infty.$
    \end{enumerate} 
\end{Lemma}
\begin{Lemma}\label{moment_est_contrast2}
    Suppose \Aref{a1} - \Aref{a3}, \Aref{a9} and \eqref{moment_conv_adaptive_sigma} holds for any $p \in \mathbb{N}$. There exists $\eta_{2} \in (0,\frac{1}{2})$ for any $p \in \mathbb{N}$ such that
\begin{enumerate}
    \item[(1)]\label{moment_est_contrast2_1} $\,\displaystyle \sup_{n,\varepsilon,\lambda}\mathbb{E}_{\theta_{0}}\bigg[\bigg\lvert \varepsilon \partial_{\mu}\tilde{\Psi}_{n,\varepsilon}^{(1)}(\tilde{\sigma}_{n,\varepsilon},\mu_{0})\bigg\rvert^{p}\bigg] < \infty,$
    \item[(2)]\label{moment_est_contrast2_2} $\displaystyle \sup_{n,\varepsilon,\lambda}\mathbb{E}_{\theta_{0}}\bigg[\bigg\lvert \varepsilon^{-\eta_{2}}\cdot \bigg\{\varepsilon^{2} \partial_{\mu}^{2} \tilde{\Psi}_{n,\varepsilon}^{(1)}(\tilde{\sigma}_{n,\varepsilon},\mu_{0})
         + I_{2}(\mu_{0})\bigg\}\bigg\lvert^{p}\bigg] < \infty,$
    \item[(3)]\label{moment_est_contrast2_3}
    $\displaystyle \sup_{n,\varepsilon,\lambda}\mathbb{E}_{\theta_{0}}\bigg[ \bigg\lvert \sup_{\mu \in \overline{\Theta}_{2}} \varepsilon^{2}\partial_{\mu}^{3}\tilde{\Psi}_{n,\varepsilon}^{(1)}(\tilde{\sigma}_{n,\varepsilon},\mu)\bigg\rvert^{p}\bigg] < \infty.$
\end{enumerate}
\end{Lemma}
\begin{Lemma}\label{moment_est_contrast3}
    Suppose \Aref{a3}, \Aref{a4}, \Aref{a8} and \Aref{a9}. There exists $\eta_{3} \in (0,\frac{1}{2})$ for any $p \in \mathbb{N}$ such that
    \begin{enumerate}
        \item[(1)]\label{moment_est_contrast3_1}
        $\displaystyle \sup_{n,\varepsilon,\lambda}\mathbb{E}_{\theta_{0}}\bigg[\bigg\lvert \frac{1}{\sqrt{\lambda}} \partial_{\alpha}\tilde{\Psi}_{\lambda}^{(2)}(\alpha_{0})\bigg\rvert^{p}\bigg] < \infty.$
        \item[(2)]\label{moment_est_contrast3_2}
        $\displaystyle \sup_{n,\varepsilon,\lambda}\mathbb{E}_{\theta_{0}}\bigg[\bigg\lvert \lambda^{\eta_{3}}\cdot 
        \bigg\{\frac{1}{\lambda} \partial_{\alpha}^{2}\tilde{\Psi}_{\lambda}^{(2)}(\alpha_{0}) + I_{3}(\alpha_{0})\bigg\}\bigg\lvert^{p}\bigg] < \infty.$
        \item[(3)]\label{moment_est_contrast3_3}
        $\displaystyle \sup_{n,\varepsilon,\lambda}\mathbb{E}_{\theta_{0}}\bigg[ \bigg(\sup_{\alpha} \bigg \lvert \frac{1}{\lambda}\partial_{\alpha}^{3}\Psi_{\lambda}^{(2)}(\alpha) \bigg \rvert \bigg)^{p}\bigg] < \infty.$
    \end{enumerate}
\end{Lemma}
Let us add the symbol needed to state Lemma \ref{approx_for_conditional_e} - \ref{moment_ineq_jump}. Denote $\tilde{\eta}_{k}^{n}(\mu)$ by
\begin{align}
    &\tilde{\eta}_{k}^{n}(\mu):= \bigg(\Delta_{k}^{n}X^{\theta,\varepsilon} - \frac{1}{n}a(X_{t_{k-1}^{n}}^{\theta,\varepsilon},\mu)\bigg)\bm{1}_{J_{k,0}^{n}}, \nonumber
\end{align}
for $k = 1,2,...,n,\quad n \in \mathbb{N}$.
\begin{Lemma}\label{approx_for_conditional_e}
   Suppose \Aref{a1} - \Aref{a3} and \Aref{a9}. Then, the following holds:
    \begin{enumerate}
        \item[(1)]$\displaystyle \mathbb{E}_{\theta_{0}}\bigg[ \Delta_{k}^{n}X^{\theta,\varepsilon} \bm{1}_{J_{k,0}^{n}} \mid \mathcal{F}_{t_{k-1}^{n}}\bigg] 
        = R(X_{t_{k-1}^{n}}^{\theta,\varepsilon},\frac{1}{n})$.
        \item[(2)]$\displaystyle \mathbb{E}_{\theta_{0}}\bigg[
            \tilde{\eta}_{k}^{n}(\mu_{0}) \mid \mathcal{F}_{t_{k-1}^{n}}\bigg] = R(X_{t_{k-1}^{n}}^{\theta,\varepsilon},\frac{1}{n^{2}}) $.
        \item[(3)]$\displaystyle \mathbb{E}_{\theta_{0}}\bigg[\tilde{\eta}_{k}^{n}(\mu_{0})^{2} 
        - \frac{\varepsilon^{2}}{n}b^{2}(X_{t_{k-1}^{n}},\sigma_{0})\bm{1}_{J_{k,0}^{n}}\mid \mathcal{F}_{t_{k-1}^{n}}\bigg] = R(X_{t_{k-1}^{n}}^{\theta,\varepsilon},\frac{\varepsilon^{2}}{n^{3/2}} \lor \frac{\varepsilon^{2}}{n^{2}}) $.
        \item[(4)]$\displaystyle \mathbb{E}_{\theta_{0}}
        \bigg[\sup_{u \in [t_{k-1}^{n},t_{k}^{n}]} 
        \lvert X_{u} - x_{u}\rvert^{p} \mid \mathcal{F}_{t_{k-1}^{n}} \bigg]  = R(X_{t_{k-1}^{n}}^{\theta,\varepsilon},\frac{1}{n^{p}} + \frac{\varepsilon^{p}}{n^{p/2}})$.
        \item[(5)]$\mathbb{E}_{\theta_{0}}\bigg[\sup_{t \in [t_{k-1}^{n},t_{k}^{n}]} \lvert X_{t}^{\theta,\varepsilon} -
        X_{t_{k-1}^{n}}^{\theta,\varepsilon} \rvert^{p} \bm{1}_{J_{k,0}^{n}}\mid \mathcal{F}_{t_{k-1}^{n}} \bigg]
         = R(X_{t_{k-1}^{n}}^{\theta,\varepsilon},\frac{1}{n^{p}} + \frac{\varepsilon^{p}}{n^{p/2}})$.  
    \end{enumerate}
\end{Lemma}  
\begin{Lemma}\label{moment_INEQ_diffusion}
    Let $f$ be a function $\mathbb{R} \times \Theta \to \mathbb{R}$. \, 
       Suppose \ref{a1} - \ref{a3}, \ref{a9} and $f \in C_{\uparrow}^{1,1}(\mathbb{R} \times \Theta;\mathbb{R})$. 
       Then, we have that for any $p \in \mathbb{N}$, 
    \begin{enumerate}
    \item[(1)]$\displaystyle \sup_{n,\varepsilon,\lambda}\mathbb{E}_{\theta_{0}}\bigg[\bigg(\sup_{\theta}\bigg\lvert \frac{1}{n}\sum_{k = 1}^{n}f(X_{t_{k-1}^{n}}^{\theta,\varepsilon},\theta) \bigg\rvert \bigg)^{p}\bigg] < \infty$.
    \item[(2)]$\displaystyle \sup_{n,\varepsilon,\lambda}\mathbb{E}_{\theta_{0}}
    \bigg[\bigg(\sup_{\theta}\varepsilon^{-1}\bigg\lvert \frac{1}{n}\sum_{k = 1}^{n}
    f(X_{t_{k-1}^{n}}^{\theta,\varepsilon},\theta)\bm{1}_{J_{k,0}^{n}} - \int_{0}^{1}f(x_{t},\theta)dt \bigg\rvert \bigg)^{p}\bigg] < \infty$.
    \item[(3)] $\displaystyle \sup_{n,\varepsilon,\lambda}\mathbb{E}_{\theta_{0}}\bigg[\bigg(\sup_{\theta}\varepsilon^{-1}\bigg \lvert \sum_{k = 1}^{n}f(X_{t_{k-1}^{n}}^{\theta,\varepsilon},\theta)\tilde{\eta}_{k}^{n}(\mu_{0})\bigg\rvert\bigg)^{p}\bigg] < \infty$.
    \item[(4)] $\displaystyle \sup_{n,\varepsilon,\lambda}\mathbb{E}_{\theta_{0}}\bigg[\bigg(\sup_{\theta}\bigg\lvert \sum_{k = 1}^{n} f(X_{t_{k-1}^{n}}^{\theta,\varepsilon},\theta)\tilde{\eta}_{k}^{n}(\mu)\bigg\rvert\bigg)^{p}  \bigg] < \infty$\,.
    \item[(5)] $\displaystyle \sup_{n,\varepsilon,\lambda}
    \mathbb{E}_{\theta_{0}}\bigg[\bigg(\sup_{\theta}\sqrt{n}\bigg\lvert \varepsilon^{-2} 
    \sum_{k = 1}^{n}f(X_{t_{k-1}^{n}}^{\theta,\varepsilon},\theta)\big(\tilde{\mu}_{k}^{n}(\mu_{0})^{2} - \frac{\varepsilon^{2}}{n}b^{2}(X_{t_{k-1}^{n}}^{\theta,\varepsilon},\sigma_{0})\big)\bigg\rvert \bigg)^{p} \bigg] < \infty$.
    \item[(6)] $\displaystyle \sup_{n,\varepsilon,\lambda}
    \mathbb{E}_{\theta_{0}}\bigg[\sup_{\theta}\bigg\lvert \varepsilon^{-2}\sum_{k = 1}^{n}f(X_{t_{k-1}^{n}}^{\theta,\varepsilon},\theta)\tilde{\mu}_{k}^{n}(\mu_{0})\bigg\rvert \bigg] < \infty.$
\end{enumerate}
\end{Lemma}
\begin{Lemma}\label{moment_ineq_jump}
    Suppose \Aref{a3}, \Aref{a4}, \Aref{a8} and \Aref{a9} and let $g$ be a
     function: $\mathbb{R} \times E \times \overline{\Theta}_{3}$ to $\mathbb{R}$ such that for any $p \in \mathbb{N}$ there exists $C > 0$,
    \begin{align}
        &\,\sup_{\alpha \in \mathbb{R} \times \overline{\Theta}_{3}}\int_{\mathbb{R}} \big\lvert g(x,y,\alpha) \big\rvert^{p}f_{\alpha_{0}}(y)dy = R(x,1), \nonumber
    \end{align}
    and denote 
    \begin{align}
        g_{k-1}\big(c(X_{t_{k-1}^{n}}^{\theta,\varepsilon},\alpha_{0})V_{\tau_{k}},\alpha \big) := &\, g \big(X_{t_{k-1}^{n}}^{\theta,\varepsilon}, 
        c(X_{t_{k-1}^{n}}^{\theta,\varepsilon},\alpha_{0})V_{\tau_{k}},\alpha\big), \nonumber 
    \end{align}
    where $E = \mathbb{R}$ or $E = \mathbb{R}_{+}$. Then, we obtain 
    \begin{align}
        &\mathbb{E}_{\theta_{0}}\bigg[ \bigg\lvert \frac{1}{\lambda}\sum_{k = 1}^{n} g_{k-1}\bigg(c(X_{t_{k-1}^{n}}^{\theta,\varepsilon},\alpha_{0})V_{\tau_{k}},\alpha\bigg)\bigg \rvert^{p}\bigg] < \infty. \label{5.13}
    \end{align}
    \normalsize
\end{Lemma}
\section{Proof}
\subsection{Notation for Proof}
Before the proof of Theorem \ref{thm1} and Lemmas \ref{error_conv_Z_and_tildeZ} - \ref{ideal_AN},
 let us add some symbols. At the begining, for any $t \geq 0$, $\Delta X_{t}^{\theta,\varepsilon}$ is defined as follows:
\begin{align}
\Delta X_{t}^{\theta,\varepsilon}&:= X_{t}^{\theta,\varepsilon} - X_{t-}^{\theta,\varepsilon}. \nonumber
\end{align}
In addition, for any $n \in \mathbb{N}$ and $k = 1,2,...,n$, $\Delta_{k}^{n}W$ and $\tau_{k}$ are defined as follows:
\begin{align}
&\Delta_{k}^{n}W := W_{t_{k}^{n}} - W_{t_{k-1}^{n}},\nonumber\\
&\tau_{k} := \inf\{t \in [t_{k-1}^{n},t_{k}^{n}]\mid \Delta X_{t}^{\theta,\varepsilon} \neq 0 \quad \mathrm{or} \quad t = t_{k}^{n}\}. \nonumber
\end{align}
and we also restate the symbol $\Delta_{k}^{n}N^{\lambda}$ by
\begin{align}
&\Delta_{k}^{n}N^{\lambda} := N_{t_{k}^{n}}^{\lambda} - N_{t_{k-1}^{n}}^{\lambda}, \nonumber
\end{align}
and the sets $J_{k,i}^{n}$ for $i = 0,1,2$ by 
\begin{align}
    &J_{k,0}^{n} := \{\Delta_{k}^{n}N^{\lambda} = 0\},\quad J_{k,1}^{n} := \{\Delta_{k}^{n}N^{\lambda} = 1\},\quad J_{k,2}^{n} := \{\Delta_{k}^{n}N^{\lambda} \geq 2\}. \nonumber
\end{align}
Based on the above symbols, denote 
\begin{align}
&C_{k,i}^{n}:= C_{k}^{n} \cap J_{k,i}^{n},\quad D_{k,i}^{n}:= D_{k}^{n} \cap J_{k,i}^{n}, \nonumber
\end{align}
and for any $\gamma > 0$,
\begin{align}
&\tilde{C}_{k,i}^{n,\gamma} := C_{k,i}^{n} \cap \tilde{I}_{x}^{\gamma},\quad \tilde{D}_{k,i}^{n,\delta} := D_{k,i}^{n} \cap \tilde{I}_{x}^{\delta}. \nonumber
\end{align}
\subsection{Proof of Main Theorem}
\begin{proof}[Proof of Theorem \ref{thm1}]
From Lemma \ref{error_conv_Z_and_tildeZ} and \ref{ideal_AN}, the following stochasitc convergence holds jointly :
\begin{align}
    \begin{pmatrix}
        \int_{V_{\varepsilon}(\mu_{0})}u_{2}\tilde{Z}_{n,\varepsilon,\lambda}^{(0)}[u_{2}] \pi_{2}(\mu_{0} + \varepsilon u_{2})du_{2} \\
        \int_{V_{\varepsilon}(\mu_{0})}\tilde{Z}_{n,\varepsilon,\lambda}^{(0)}[u_{2}] \pi_{2}(\mu_{0} + \varepsilon u_{2})du_{2}
    \end{pmatrix}
    &\ind 
    \begin{pmatrix}
        \mathcal{N}(0,I_{0}(\mu_{0}))\\
        I_{0}(\mu_{0})
    \end{pmatrix}
    ,\nonumber
\end{align}
and
\begin{align}
\begin{pmatrix}
\int_{V_{n}(\sigma_{0})}u_{1}\tilde{Z}_{n,\varepsilon,\lambda}^{(1)}[u_{1}]\pi_{1}(u_{1}/\sqrt{n} + \sigma_{0})du_{1}\\
\int_{V_{n}(\sigma_{0})}\tilde{Z}_{n,\varepsilon,\lambda}^{(1)}[u_{1}]\pi_{1}(u_{1}/\sqrt{n} + \sigma_{0})du_{1} 
\end{pmatrix}
&\ind 
\begin{pmatrix}
\mathcal{N}(0,I_{1}(\sigma_{0}))\\
I_{1}(\sigma_{0})
\end{pmatrix}
,\nonumber\\
\begin{pmatrix}
\int_{V_{\varepsilon}(\mu_{0})}u_{2}Z_{n,\varepsilon,\lambda}^{(2)}[u_{2}]\pi_{2}(\varepsilon u_{2}+ \mu_{0})du_{2}\\
\int_{V_{\varepsilon}(\mu_{0})}Z_{n,\varepsilon,\lambda}^{(2)}[u_{2}]\pi_{2}(\varepsilon u_{2}+ \mu_{0})du_{2}
\end{pmatrix}
&\ind 
\begin{pmatrix}
\mathcal{N}(0,I_{2}(\mu_{0}))\\
I_{2}(\mu_{0})
\end{pmatrix}
, \nonumber\\
\begin{pmatrix}
\int_{V_{\lambda}(\alpha_{0})}u_{3}\tilde{Z}_{n,\varepsilon,\lambda}^{(3)}[u_{3}]\pi(u_{3}/\sqrt{\lambda} + \alpha_{0})du_{3}\\
\int_{V_{\lambda}(\alpha_{0})}\tilde{Z}_{n,\varepsilon,\lambda}^{(3)}[u_{3}]\pi(u_{3}/\sqrt{\lambda} + \alpha_{0})du_{3} 
\end{pmatrix}
&\ind 
\begin{pmatrix}
\mathcal{N}(0,I_{3}(\alpha_{0}))\\
I_{3}(\alpha_{0})
\end{pmatrix}
.
\nonumber
\end{align}
On the otherhand, it is easy to obtain that
\begin{align}
    \varepsilon^{-1}(\widehat{\mu}_{n,\varepsilon}^{(0)} - \mu_{0}) & 
    = \frac{\int_{V_{\varepsilon}(\mu_{0})}u_{2}Z_{n,\varepsilon,\lambda}^{(0)}[u_{2}] 
    \pi_{2}(\varepsilon u_{2} + \mu_{0})du_{2}}
    {\int_{V_{\varepsilon}(\mu_{0})} Z_{n,\varepsilon,\lambda}^{(0)}[u_{2}] 
    \pi_{2}(\varepsilon u_{2} + \mu_{0})du_{2}}, \nonumber
\end{align}
and 
\begin{align}
        \sqrt{n}(\widehat{\sigma}_{n,\varepsilon} - \sigma_{0})& =
    \frac{\int_{V_{n}(\sigma_{0})}u_{1}Z_{n,\varepsilon,\lambda}^{(1)}[u_{1}] \pi_{1}(u_{1}/\sqrt{n} + \sigma_{0})du_{1}}{ \int_{V_{n}(\sigma_{0})} Z_{n,\varepsilon,\lambda}^{(1)}[u_{1}] \pi_{1}(u_{1}/\sqrt{n} + \sigma_{0})du_{1}}, \nonumber\\
    \varepsilon^{-1}(\widehat{\mu}_{n,\varepsilon} - \mu_{0}) &= \frac{\int_{V_{\varepsilon}(\mu_{0})}u_{2}Z_{n,\varepsilon,\lambda}^{(2)}[u_{2}] \pi_{2}(\varepsilon u_{2} + \mu_{0})du_{2}}{\int_{V_{\varepsilon}(\mu_{0})} Z_{n,\varepsilon,\lambda}^{(2)}[u_{2}] \pi_{2}(\varepsilon u_{2} + \mu_{0})du_{2}}, \nonumber\\
    \sqrt{\lambda}(\widehat{\alpha}_{\lambda} - \alpha_{0}) & = \frac{\int_{V_{\varepsilon}(\alpha_{0})}u_{3}Z_{n,\varepsilon,\lambda}^{(3)}[u_{3}] \pi_{3}( u_{3}/\sqrt{\lambda} + \alpha_{0})du_{3}}{
    \int_{V_{\lambda}(\alpha_{0})} Z_{n,\varepsilon,\lambda}^{(3)}[u_{3}] \pi_{3}( u_{3}/\sqrt{\lambda}+ \alpha_{0})du_{3}}. \nonumber
    \end{align}
Using continuous mapping theorem, we obtain the desired result.
\end{proof}
\begin{proof}[proof of Lemma \ref{error_conv_Z_and_tildeZ}]
First of all, we show \eqref{error_conv_z0_and_tildez0}. We easily obtain that
   \begin{align}
     &\begin{pmatrix}
     \int_{V_{n}(\sigma_{0})}u_{1}Z_{n,\varepsilon,\lambda}^{(1)}[u_{1}]\pi_{1}(u_{1}/\sqrt{n} + \sigma_{0})du_{1} \\
     \int_{V_{n}(\sigma_{0})}Z_{n,\varepsilon,\lambda}^{(1)}[u_{1}]\pi_{1}(u_{1}/\sqrt{n} + \sigma_{0})du_{1} 
     \end{pmatrix}
     - 
     \begin{pmatrix}
     \int_{V_{n}(\sigma_{0})}u_{1}\tilde{Z}_{n,\varepsilon,\lambda}^{(1)}[u_{1}]\pi_{1}(u_{1}/\sqrt{n} + \sigma_{0})du_{1}\\
     \int_{V_{n}(\sigma_{0})}\tilde{Z}_{n,\varepsilon,\lambda}^{(1)}[u_{1}]\pi_{1}(u_{1}/\sqrt{n} + \sigma_{0})du_{1}
     \end{pmatrix}
      \nonumber\\
     =&
     \begin{pmatrix}
        \int_{V_{n}(\sigma_{0})}u_{1}Z_{n,\varepsilon,\lambda}^{(1)}[u_{1}]\pi_{1}(u_{1}/\sqrt{n} + \sigma_{0})du_{1} \\
        \int_{V_{n}(\sigma_{0})}Z_{n,\varepsilon,\lambda}^{(1)}[u_{1}]\pi_{1}(u_{1}/\sqrt{n} + \sigma_{0})du_{1} 
        \end{pmatrix}
        - 
        \begin{pmatrix}
        \int_{V_{n}(\sigma_{0})}u_{1}\overline{Z}_{n,\varepsilon,\lambda}^{(1)}[u_{1}]\pi_{1}(u_{1}/\sqrt{n} + \sigma_{0})du_{1}\\
        \int_{V_{n}(\sigma_{0})}\overline{Z}_{n,\varepsilon,\lambda}^{(1)}[u_{1}]\pi_{1}(u_{1}/\sqrt{n} + \sigma_{0})du_{1}
        \end{pmatrix}
        \nonumber\\
     + &
     \begin{pmatrix}
        \int_{V_{n}(\sigma_{0})}u_{1}\overline{Z}_{n,\varepsilon,\lambda}^{(1)}[u_{1}]\pi_{1}(u_{1}/\sqrt{n} + \sigma_{0})du_{1} \\
        \int_{V_{n}(\sigma_{0})}\overline{Z}_{n,\varepsilon,\lambda}^{(1)}[u_{1}]\pi_{1}(u_{1}/\sqrt{n} + \sigma_{0})du_{1} 
        \end{pmatrix}
        - 
        \begin{pmatrix}
        \int_{V_{n}(\sigma_{0})}u_{1}\tilde{Z}_{n,\varepsilon,\lambda}^{(1)}[u_{1}]\pi_{1}(u_{1}/\sqrt{n} + \sigma_{0})du_{1}\\
        \int_{V_{n}(\sigma_{0})}\tilde{Z}_{n,\varepsilon,\lambda}^{(1)}[u_{1}]\pi_{1}(u_{1}/\sqrt{n} + \sigma_{0})du_{1}
        \end{pmatrix}
        \nonumber\\
     =:&\,A_{n,1} + A_{n,2}, \nonumber
   \end{align}
   where 
   \begin{align}
     \overline{Z}_{n,\varepsilon,\lambda}^{(1)}[u_{1}] &\,= \exp\{\Psi_{n,\varepsilon,\lambda}(\sigma_{0} + u_{1}/\sqrt{n},\tilde{\mu}_{n,\varepsilon},\alpha) - \Psi_{n,\varepsilon,\lambda}(\sigma_{0},\tilde{\mu}_{n,\varepsilon},\alpha)\}. \nonumber
   \end{align}
   On the other hand, for any $n \in \mathbb{N}, \varepsilon > 0, \lambda > 0$, we have that 
\begin{align}
&\bigg\{\sum_{k = 1}^{n}\bm{1}_{C_{i,0}^{n}} < 1/2 \bigg\} \cap \bigg\{\sum_{k = 1}^{n}\bm{1}_{C_{i,2}^{n}} < 1/2 \bigg\} \subset \bigg\{\Psi_{n,\varepsilon}^{(0)}(\mu) = \tilde{\Psi}_{n,\varepsilon}^{(0)}(\mu),(\mu \in \overline{\Theta}_{2})\bigg\}, \nonumber\\
&\bigg\{\Psi_{n,\varepsilon}^{(0)}(\mu) = \tilde{\Psi}_{n,\varepsilon}^{(0)}(\mu),(\mu \in \overline{\Theta}_{2})\bigg\} \subset \bigg\{ \lvert \widehat{\mu}_{n,\varepsilon} - \tilde{\mu}_{n,\varepsilon} \rvert = 0 \bigg\} \subset \bigg\{ A_{n,1} = 0 \bigg\},\nonumber
\end{align}
and
\begin{align}
&\bigg\{\sum_{k = 1}^{n}\bm{1}_{C_{i,0}^{n}} < 1/2 \bigg\} \cap \bigg\{\sum_{k = 1}^{n}\bm{1}_{C_{i,2}^{n}} < 1/2 \bigg\} \subset \bigg\{\Psi_{n,\varepsilon}^{(1)}(\sigma,\mu) 
 = \tilde{\Psi}_{n,\varepsilon}^{(1)}(\sigma,\mu), (\sigma,\mu) \in \overline{\Theta}_{1} \times \overline{\Theta}_{2} \bigg\}, \nonumber\\
&\bigg\{\Psi_{n,\varepsilon}^{(1)}(\sigma,\mu) = \tilde{\Psi}_{n,\varepsilon}^{(1)}(\sigma,\mu), (\sigma,\mu) \in \overline{\Theta}_{1} \times \overline{\Theta}_{2} \bigg\} \subset \bigg\{A_{n,2} = 0\bigg\}. \nonumber
\end{align}
Thus, under \Aref{a9}, as $n \to \infty, \varepsilon \to 0, \lambda \to \infty,
\lambda \int_{\lvert z \rvert \leq 4v_{2}/c_{1}n^{\rho}}f_{\alpha_{0}}(z)dz \to 0$,
we obtain $A_{n,1} = A_{n,2} = o_{p}(1)$, which leads to \eqref{error_conv_z1_and_tildez1}. 
In the same way, we have \eqref{error_conv_z0_and_tildez0} and \eqref{error_conv_z3_and_tildez3}.
Now, we prove \eqref{error_conv_z3_and_tildez3}.
Let us start considering that  
\begin{align}
    &\int_{V_{\lambda}(\alpha_{0})}u_{3} 
    \mathfrak{Z}_{\lambda}^{(3)}[u_{3}]
    \pi(u_{3}/\sqrt{\lambda} + \alpha_{0})du_{3} = o_{p}(1), \nonumber
\end{align}
as $n \to \infty, \varepsilon \to 0, \lambda \to \infty, \lambda 
\int_{\lvert z \rvert \leq 4v_{2}/c_{1}n^{\rho}}f_{\alpha_{0}}(z)dz \to 0$.
We observe that
\begin{align}
    &\int_{V_{\lambda}(\alpha_{0})}u_{3} 
    \mathfrak{Z}_{\lambda}^{(3)}[u_{3}]
    \pi(u_{3}/\sqrt{\lambda} + \alpha_{0})du_{3} \nonumber\\
    =& \int_{V_{\lambda}(\alpha_{0})}u_{3} 
    \mathfrak{Z}_{\lambda}^{(3)}[u_{3}]
    \pi(u_{3}/\sqrt{\lambda} + \alpha_{0})du_{3}\bm{1}_{\{\lvert \chi[u_{3}] \rvert \leq 1\}} + \int_{V_{\lambda}(\alpha_{0})}u_{3} 
    \mathfrak{Z}_{\lambda}^{(3)}[u_{3}]
    \pi(u_{3}/\sqrt{\lambda} + \alpha_{0})du_{3}\bm{1}_{\{\lvert \chi[u_{3}] \rvert > 1\}}, \nonumber 
\end{align}
where 
\begin{align}
 \chi_{n,\varepsilon,\lambda}[u_{3}]&\,:= 
 \bigg(\Psi_{\lambda}^{(2)}(u_{3}/\sqrt{\lambda} + \alpha_{0})
 - \Psi_{\lambda}^{(2)}(\alpha_{0})\bigg)-  \bigg(\tilde{\Psi}_{\lambda}^{(2)}(u_{3}/\sqrt{\lambda} + \alpha_{0}) - \tilde{\Psi}_{\lambda}^{(2)}(\alpha_{0})\bigg). \nonumber 
   \end{align} 
 From Lemma \ref{error_conv_contrast3}, the second term goes to zero. Let us consider the first term. 
We easily get that 
        \begin{align}
         \int_{V_{\lambda}(\alpha_{0})}u_{3}\mathfrak{Z}_{\lambda}^{(3)}[u_{3}]\pi(u_{3})du_{3} &=
          \int_{V_{\lambda}(\alpha_{0})}u_{3}\tilde{Z}_{\lambda}^{(3)}[u_{3}]
         \bigg(\exp\{\chi_{n,\varepsilon,\lambda}[u_{3}]\} - 1\bigg)\pi_{3}(u_{3}/\sqrt{\lambda} + \alpha_{0})du_{3}, \nonumber
         \end{align}
      Since $u_{3}/\sqrt{\lambda} + \alpha_{0} \in \overline{\Theta}_{3}$, we have that 
        \begin{align}
         &\,\sup_{u_{3} \in V_{\lambda}(\alpha_{0})}\lvert \chi_{n,\varepsilon,\lambda}[u_{3}]\rvert \leq 2\sup_{\alpha \in \overline{\Theta}_{3}} \lvert  \Psi_{\lambda}^{(2)}(\alpha) -  \tilde{\Psi}_{\lambda}^{(2)}(\alpha) \rvert, \nonumber 
        \end{align}
        for any $u_{3} \in V_{\lambda}(\alpha_{0})$, then we obtain 
        \begin{align}
         &\bigg\lvert \int_{V_{\lambda}(\alpha_{0})}u_{3} 
    \mathfrak{Z}_{\lambda}^{(3)}[u_{3}]
    \pi(u_{3}/\sqrt{\lambda} + \alpha_{0})du_{3}
         \bm{1}_{\{\lvert\chi_{n,\varepsilon,\lambda}[u_{3}]\rvert < 1\}} \bigg\rvert \nonumber\\
         \leq\,&\,\int_{\overline{\Theta}_{3}}u_{3}\tilde{Z}_{n,\varepsilon,\lambda}^{(3)}[u_{3}]\bigg(\bigg \rvert\exp \bigg\{\chi_{n,\varepsilon,\lambda}[u_{3}] \bigg\} - 1 \bigg\lvert\bigg) \pi(\alpha_{0} + u_{3}/\sqrt{\lambda})du_{3}\bm{1}_{\{\lvert\chi_{n,\varepsilon,\lambda}[u_{3}]\rvert < 1\}}  \nonumber \\
         \leq\,&\,2\bigg(\sup_{\alpha \in \overline{\Theta}_{3}}\lvert \Psi_{\lambda}^{(2)}(\alpha) - 
         \tilde{\Psi}_{\lambda}^{(2)}(\alpha)\rvert\bigg) \cdot \bigg(\int_{V_{\lambda}(\alpha_{0})}\lvert u_{3} \rvert Z_{n,\varepsilon,\lambda}^{(3)}[u_{3}]\pi(u_{3})du_{3}\bigg). \nonumber
        \end{align}
     Combining Lemma \ref{error_conv_contrast3}, \ref{moment_est_contrast3} and continuous mapping theorem, 
     \begin{align}
        &\int_{V_{\lambda}(\alpha_{0})}u_{3}\mathfrak{Z}_{\lambda}^{(3)}[u_{3}]
             \pi(u_{3}/\sqrt{\lambda} + \alpha_{0})du_{3} 
              \bm{1}_{\{\lvert\chi_{n,\varepsilon,\lambda}[u_{3}]\rvert < 1\}} \inp 0. \nonumber
        \end{align}   
        as $n \to \infty, \varepsilon \to 0, \lambda \to \infty, \lambda 
\int_{\lvert z \rvert \leq 4v_{2}/c_{1}n^{\rho}}f_{\alpha_{0}}(z)dz \to 0$.
In the same way, we obtain that 
\begin{align}
    &\int_{V_{\lambda}(\alpha_{0})} 
    \mathfrak{Z}_{\lambda}^{(3)}[u_{3}]
    \pi(u_{3}/\sqrt{\lambda} + \alpha_{0})du_{3} \inp 0, \nonumber
\end{align}
as $n \to \infty, \varepsilon \to 0, \lambda \to \infty, \lambda \int_{\lvert z \rvert \leq 4v_{2}/c_{1}n^{\rho}}f_{\alpha_{0}}(z)dz \to 0$.  
Thus, we obtain \eqref{error_conv_z3_and_tildez3}.
 \end{proof}
\begin{proof}[Proof of Lemma \ref{error_conv_contrast3}]
    We prove Lemma \ref{error_conv_contrast3} imitating the proof of Lemma 4.12 in Kobayashi and Shimizu \cite{9}.
    Suppose \ref{a1}, \ref{a3}, \ref{a4}, \ref{a6}, \ref{a8}, \ref{a9}. We easily have that
    \begin{align}
        &\,\sup_{\alpha\in \overline{\Theta}_{3}}\lvert \Psi_{\lambda}^{(2)}(\alpha) 
        - \tilde{\Psi}_{\lambda}^{(2)}(\alpha) \rvert\nonumber\\
        \leq &\,\sup_{\alpha\in \overline{\Theta}_{3}} \bigg\lvert
         \sum_{k = 1}^{n}\psi(X_{t_{k-1}^{n}}^{\theta,\varepsilon},
         \Delta_{k}^{n}X^{\theta,\varepsilon}/\varepsilon,\alpha)\bm{1}_{D_{k}^{n}} 
         - \sum_{k =1}^{n}\psi(X_{t_{k-1}^{n}}^{\theta,\varepsilon},
         \Delta_{k}^{n}X^{\theta,\varepsilon}/\varepsilon,\alpha)\bm{1}_{\tilde{D}_{k,1}^{n,\gamma}} \bigg\rvert \nonumber\\
        +\,&\,\sup_{\alpha\in \overline{\Theta}_{3}} \bigg\lvert 
        \sum_{k =1}^{n}\psi(X_{t_{k-1}^{n}}^{\theta,\varepsilon},
        \Delta_{k}^{n}X^{\theta,\varepsilon}/\varepsilon,\alpha)
        \bm{1}_{\tilde{D}_{k,1}^{n,\gamma}} - \sum_{k =1}^{n}
        \psi(X_{t_{k-1}^{n}}^{\theta,\varepsilon},c(X_{t_{k-1}^{n}}^{\theta,\varepsilon},\alpha_{0})V_{\tau_{k}},\alpha)
        \bm{1}_{\tilde{D}_{k,1}^{n,\gamma}} \bigg\rvert \nonumber\\
       +\,&\,\sup_{\alpha\in \overline{\Theta}_{3}} 
       \bigg\lvert \sum_{k =1}^{n}
       \psi(X_{t_{k-1}^{n}}^{\theta,\varepsilon},c(X_{t_{k-1}^{n}}^{\theta,\varepsilon},\alpha_{0})V_{\tau_{k}},\alpha)
       \bm{1}_{\tilde{D}_{k,1}^{n,\gamma}} - \sum_{k =1}^{n}\psi(X_{t_{k-1}^{n}}^{\theta,\varepsilon},c(X_{t_{k-1}^{n}}^{\theta,\varepsilon},\alpha_{0})V_{\tau_{k}},\alpha)\bm{1}_{J_{k,1}^{n}} \bigg\rvert \nonumber\\
        =:&\, B_{n,1} + B_{n,2} + B_{n,3}, \nonumber
    \end{align}
    where $\gamma > 0$ is a constant in \Aref{a6}
    By Remark 4.1 and Lemma 4.12 in Kobayashi and Shimizu \cite{9}, we obtain that
    \begin{align}
        &\,B_{n,1} = B_{n,3} = o_{p}(1), \nonumber
    \end{align}
as $n \to \infty, \varepsilon \to 0, \lambda \to \infty$ under \Aref{a1}, \Aref{a3}, \Aref{a4}, \Aref{a9}. 
It remains to prove that $B_{n,2} = O_{p}(r_{n,\varepsilon})$ 
where $r_{n,\varepsilon} := \displaystyle \frac{1}{\varepsilon^{p}n^{p}} + \frac{1}{n^{p/2}} $ for any $p \in \mathbb{N}$.
We have that for any $M > 0$,
    \begin{align}
        &\,\mathbb{P}\bigg(\sup_{\alpha}\bigg\lvert \sum_{k = 1}^{n}
        \bigg\{g_{k}\bigg(\frac{\Delta_{k}^{n}X^{\theta,\varepsilon}}{\varepsilon},\alpha\bigg) 
        - g_{k}\bigg(\frac{\Delta X_{\tau_{k}}}{\varepsilon},\alpha\bigg)\bigg\}\bm{1}_{\tilde{D}_{k,1}^{n,\gamma}} \bigg\rvert > Mr_{n,\varepsilon}\bigg)\nonumber\\
        \leq\,&\,\mathbb{P}\bigg(\sup_{1 \leq k \leq n}\lvert Y_{k}^{\varepsilon}\lambda \rvert > 1\bigg) + \mathbb{P}\bigg(\sup_{\alpha \in \overline{\Theta}_{3}} \bigg\lvert \sum_{k = 1}^{n}\int_{0}^{1}
        \frac{\partial g_{k}}{\partial y}
        \bigg(\frac{\Delta X_{\tau_{k}}}{\varepsilon} + \eta Y_{k}^{\varepsilon},\alpha \bigg)d\eta Y_{k}^{\varepsilon}
        \bm{1}_{\{\lvert Y_{k}^{\varepsilon} \rvert \lambda \leq 1\}\cap \tilde{D}_{k,1}^{n,\gamma}\}}\bigg \rvert > Mr_{n,\varepsilon}\bigg). \nonumber
    \end{align}
From Lemma 4.6 in Kobayashi and Shimizu \cite{9} and \Aref{a9}, the first term converges to zero. On the second term, we have that
    \begin{align}
        &\,\sup_{\alpha \in \overline{\Theta}_{3}} \bigg\lvert \sum_{k = 1}^{n}
        \int_{0}^{1}\frac{\partial g_{k}}{\partial y}
        \bigg(\frac{\Delta X_{\tau_{k}}}{\varepsilon} + \eta Y_{k}^{\varepsilon},\alpha \bigg)d\eta 
        \bm{1}_{\{\{\lvert Y_{k}^{\varepsilon} \rvert \lambda \leq 1\}\cap \tilde{D}_{k,1}^{n,\gamma}\}}Y_{k}^{\varepsilon}\bigg \rvert \nonumber\\
        \leq &\,\sup_{\alpha \in \overline{\Theta}_{3}}
         \bigg\lvert \frac{1}{\lambda}\sum_{k = 1}^{n}
         \int_{0}^{1}\frac{\partial g_{k}}{\partial y}
         \bigg(\frac{\Delta X_{\tau_{k}}}{\varepsilon} 
         + \eta Y_{k}^{\varepsilon},\alpha \bigg)d\eta 
         \bm{1}_{\{\{\lvert Y_{k}^{\varepsilon} \rvert \lambda\leq 1\}\cap \tilde{D}_{k,1}^{n,\gamma}\}}\lambda Y_{k}^{\varepsilon}\bigg \rvert \nonumber\\
         \leq&\,\sup_{\alpha \in \overline{\Theta}_{3}}
         \bigg \lvert \frac{1}{\lambda}\int_{0}^{1} \frac{\partial g_{k}}{\partial y}
         \bigg(\frac{\Delta X_{\tau_{k}}}{\varepsilon} + tY_{k}^{\varepsilon},\theta\bigg) d\eta
          \bm{1}_{\{
          \{\lvert Y_{k} ^{\varepsilon} \rvert \lambda\leq 1\} \cap \tilde{D}_{k,1}^{n,\gamma} \}} \bigg \rvert.   \nonumber
        \end{align}
        Suppose \Aref{a6} and put $r_{n,\varepsilon} := \frac{1}{\varepsilon^{p}n^{p}} + \frac{1}{n^{p/2}}$ where $p \in \mathbb{N}$ satisfies that $\frac{1}{p} < \delta$ with $\delta$ 
        in \Aref{a9}. 
        As in the proof of Lemma 4.12 in Kobayashi and Shimizu \cite{9}, we have that
        \begin{align}
            &\,\sup_{\alpha \in \overline{\Theta}_{3}}
             \bigg \lvert \frac{1}{\lambda}\int_{0}^{1} \frac{\partial g_{k}}{\partial y}
             \bigg(\frac{\Delta X_{\tau_{k}}}{\varepsilon} + tY_{k}^{\varepsilon},\theta\bigg) d\eta
              \bm{1}_{\{\{\lvert Y_{k} ^{\varepsilon} \rvert \lambda \leq 1\} \cap \tilde{D}_{k,1}^{n,\gamma}\}} 
              \bigg \rvert = O_{p}(r_{n,\varepsilon}), \label{main_term_of_error_between_z3_and_tildez3}
        \end{align}
 which implies the desired result.
        If you suppose \Aref{a7} in place of \Aref{a6}, put $r_{n,\varepsilon}^{'} := \frac{1}{n^{1 - 1/p - q\rho}} + \frac{1}{n^{1/2 - 1/p - 1/\rho}}$ where $p \in \mathbb{N}$ satisfies $1/p < \delta$ with $\delta > 0$ in 
        \Aref{a9} and $q > 0$ in \Aref{a7} and $\rho \in (0,\frac{1}{4q})$. Then, we obtain the result by substituting $r_{n,\varepsilon}^{'}$ for $r_{n,\varepsilon}$ in \eqref{main_term_of_error_between_z3_and_tildez3}.
\end{proof}
\begin{proof}[Proof of Lemma \ref{ideal_AN}]
    Let us begin with showing that the convergence of each of \eqref{AN_ideal_initial}, \eqref{AN_ideal_adaptive_sigma} - \eqref{AN_ideal_adaptive_alpha},
     and \eqref{moment_conv_ideal_initial} and \eqref{moment_conv_adaptive_sigma} - \eqref{moment_conv_adaptive_alpha}. For \eqref{AN_ideal_initial}, in view of Theorem 10 (a) in Yoshida \cite{17}, it is sufficient to show the following:
\begin{itemize}
    \item For any $R > 0$
    and for any $u \in \mathbb{R}^{d_{2}}$ with $\lvert u \rvert \leq R$,
\begin{align}
    & \log \tilde{Z}_{n,\varepsilon,\lambda}^{0}[u_{0}] \ind \Delta_{0}[u_{0}] + \frac{1}{2}I_{0}[u_{0}^{\otimes 2}],\quad \Delta_{0} \sim \mathcal{N}(0,I_{0}(\mu_{0})). \label{dist_conv_contrast_z0} 
\end{align}
\item For any $L > 0$ there exists a constant $C_{L} > 0$ and for any $r > 0$, 
\begin{align}
    &\,\sup_{n,\varepsilon,\lambda}\mathbb{P}\bigg(\sup_{u_{2} \in V_{\varepsilon}(\mu_{0})}\tilde{Z}_{n,\varepsilon,\lambda}^{(0)}(\mu_{0}) \geq \exp\{-r/2\} \bigg) \leq \frac{C_{L}}{r^{L}}.\label{large_devi_ineq_z0}
\end{align}
\end{itemize}
For \eqref{dist_conv_contrast_z0}, by applying Taylor expansion and Lemma \ref{moment_est_contrast0}, we obtain that
    \begin{align}
        \log \tilde{Z}_{n,\varepsilon,\lambda}^{0}[u_{2}] = &\, \varepsilon \partial_{\mu}\tilde{\Psi}_{n,\varepsilon}^{(0)}(\mu_{0})[u] + \frac{1}{2}\varepsilon^{2}\partial_{\mu}^{2}\tilde{\Psi}_{n,\varepsilon}^{(0)}(\mu_{0})[(u_{2})^{\otimes 2}] \nonumber\\
        &\,\quad +  \int_{t = 0}^{t = 1} \frac{(1-t)^{2}}{2} \partial_{\mu}^{3}\tilde{\Psi}_{n,\varepsilon}^{(0)}(\mu_{0} + \varepsilon t) \bigg[\bigg(\varepsilon u_{2}\bigg)^{\otimes 3} \bigg] \nonumber\\
        = &\, \varepsilon \partial_{\mu}\tilde{\Psi}_{n,\varepsilon}^{(0)}(\mu_{0})[u]
         - \frac{1}{2}I_{0}(\mu_{0})[u_{2}^{\otimes 2}] + o_{p}(1). \nonumber
    \end{align}
Thus, we obtain \eqref{dist_conv_contrast_z0} applying Martingale central limit theorem.
Now, we show \eqref{large_devi_ineq_z0}.
In the view of theorem 3 in Yoshida \cite{17}, we obtain it combining 
\Aref{a5} and the result of Lemma \eqref{moment_est_contrast0},
which follow from \Aref{a1} - \Aref{a3}. Thus, we obtain \eqref{AN_ideal_initial}.
For \eqref{moment_conv_ideal_initial}, in view of Theorem 10 (b) in \cite{17} Yoshida, it suffices to show that for any $p \in \mathbb{N}$,
\begin{align}
    &\,\sup_{n,\varepsilon,\lambda}\mathbb{E}_{\theta_{0}}
    \bigg[\bigg \lvert\bigg(\int_{V_{\varepsilon}(\mu_{0})} Z_{n,\varepsilon}^{(0)}[u_{2}]\pi_{2}(\varepsilon u_{2} + \mu_{0})du_{2} \bigg)^{-1} \bigg \rvert^{p} \bigg] < \infty. \nonumber
\end{align}
The above inequality is obtained by combining the complement \eqref{moment_est_contrast0} and the Lemma 2 in Yoshida \cite{17}.
 We show \eqref{AN_ideal_adaptive_sigma} and \eqref{moment_conv_adaptive_sigma}. For \eqref{AN_ideal_adaptive_sigma},
 as in the proof of \eqref{AN_ideal_initial}, it suffices to show the following:
\begin{itemize}
    \item For any $R > 0$ and for any $u_{1} \in \mathbb{R}^{d_{1}}$ with $\lvert u_{1} \rvert \leq R$,
\begin{align}
    & \log \tilde{Z}_{n,\varepsilon,\lambda}^{0}[u_{1}] \ind \Delta_{1}[u_{1}] + \frac{1}{2}I_{1}[u_{1}^{\otimes 2}],\quad \Delta_{1} \sim \mathcal{N}(0,I(\sigma_{0})). \label{dist_conv_contrast_z1} 
\end{align}
\item For any $L > 0$, there exists a constant $C_{L} > 0$ such that for any $r > 0$, 
\begin{align}
    &\,\sup_{n,\varepsilon,\lambda}\mathbb{P}\bigg(\sup_{u_{1} \in V_{n}(\sigma_{0})}\tilde{Z}_{n,\varepsilon,\lambda}^{(1)}(\sigma_{0}) \geq \exp\{-r/2\} \bigg) \leq \frac{C_{L}}{r^{L}}.\label{large_devi_ineq_z1}
\end{align}
\end{itemize}
We obtain \eqref{dist_conv_contrast_z1} as in the proof \eqref{dist_conv_contrast_z0}.
As for \eqref{large_devi_ineq_z1},it follows from \Aref{a1} - \Aref{a3}, \Aref{a9},
 and \eqref{moment_conv_ideal_initial},
which imply the result of Lemma \eqref{moment_est_contrast0}.
As for \eqref{dist_conv_contrast_z1}, we obtain it as in the proof of \eqref{dist_conv_contrast_z0}.
Consequently, \eqref{AN_ideal_adaptive_sigma} and \eqref{moment_conv_adaptive_sigma} are established.

We show \eqref{AN_ideal_adaptive_mu}, \eqref{moment_conv_adaptive_mu} and \eqref{AN_ideal_adaptive_alpha},
\eqref{moment_conv_adaptive_alpha}.
As in the proof of \eqref{dist_conv_contrast_z0} - \eqref{large_devi_ineq_z1}, we obtain that
for any $R > 0$ and for any $u_{2} \in \mathbb{R}^{d_{2}}$ with  $\lvert u_{2} \rvert \leq R$:
    \begin{align}
        \log \tilde{Z}_{n,\varepsilon,\lambda}^{0}[u_{2}] = &\, \varepsilon \partial_{\mu}\tilde{\Psi}_{n,\varepsilon}^{(0)}(\mu_{0})[u] + \frac{1}{2}\varepsilon^{2}\partial_{\mu}^{2}\tilde{\Psi}_{n,\varepsilon}^{(0)}(\mu_{0})[(u_{2})^{\otimes 2}] \nonumber\\
        &\,\quad +  \int_{t = 0}^{t = 1} \frac{(1-t)^{2}}{2} \partial_{\mu}^{3}\tilde{\Psi}_{n,\varepsilon}^{(0)}(\mu_{0} + \varepsilon t) \bigg[\bigg(\varepsilon u_{2}\bigg)^{\otimes 3} \bigg] \nonumber\\
        = &\, \varepsilon \partial_{\mu}\tilde{\Psi}_{n,\varepsilon}^{(0)}(\mu_{0})[u] - \frac{1}{2}I_{0}(\mu_{0})[u_{2}^{\otimes 2}] + o_{p}(1) \nonumber\\
        \ind &\, \mathcal{N}(0,I_{0}(\mu_{0}))[u_{2}] - \frac{1}{2}I_{0}(\mu_{0})[u_{2}^{\otimes 2}],\quad \Delta_{2} \sim \mathcal{N}(0,I_{2}(\sigma_{0})), \label{dist_conv_contrast_z2}
    \end{align}
    and 
    \begin{align}
        \log \tilde{Z}_{n,\varepsilon,\lambda}^{0}[u_{2}] = &\, \varepsilon \partial_{\mu}\tilde{\Psi}_{n,\varepsilon}^{(0)}(\mu_{0})[u] + \frac{1}{2}\varepsilon^{2}\partial_{\mu}^{2}\tilde{\Psi}_{n,\varepsilon}^{(0)}(\mu_{0})[(u_{2})^{\otimes 2}] \nonumber\\
        &\,\quad +  \int_{t = 0}^{t = 1} \frac{(1-t)^{2}}{2} \partial_{\mu}^{3}\tilde{\Psi}_{n,\varepsilon}^{(0)}(\mu_{0} + \varepsilon t) \bigg[\bigg(\varepsilon u_{2}\bigg)^{\otimes 3} \bigg] \nonumber\\
        = &\, \varepsilon \partial_{\mu}\tilde{\Psi}_{n,\varepsilon}^{(0)}(\mu_{0})[u] - \frac{1}{2}I_{0}(\mu_{0})[u_{2}^{\otimes 2}] + o_{p}(1) \nonumber\\
        \ind &\, \mathcal{N}(0,I_{0}(\mu_{0}))[u_{2}] - \frac{1}{2}I_{0}(\mu_{0})[u_{2}^{\otimes 2}],\quad \Delta_{3} \sim \mathcal{N}(0,I_{3}(\alpha_{0})). \label{dist_conv_contrast_z3}
    \end{align}
In addition, for any $L > 0$, there exists a constant $C_{L} > 0$ such that for any $r > 0$, 
\begin{align}
    &\sup_{n,\varepsilon,\lambda}\mathbb{P}\bigg(\sup_{u_{2} \in V_{\varepsilon}(\mu_{0})}\tilde{Z}_{n,\varepsilon,\lambda}^{(0)}(\mu_{0}) \geq \exp\{-r/2\} \bigg) \leq \frac{C_{L}}{r^{L}},\label{large_devi_ineq_z2}\\
    &\sup_{n,\varepsilon,\lambda}\mathbb{P}\bigg(\sup_{u_{3} \in V_{\varepsilon}(\mu_{0})}\tilde{Z}_{n,\varepsilon,\lambda}^{(0)}(\mu_{0}) \geq \exp\{-r/2\} \bigg) \leq \frac{C_{L}}{r^{L}}.\label{large_devi_ineq_z3}
\end{align}
Combining \eqref{dist_conv_contrast_z2},\eqref{large_devi_ineq_z2} 
and \eqref{dist_conv_contrast_z3},\eqref{large_devi_ineq_z3}, 
we obtain \eqref{AN_ideal_adaptive_mu} and \eqref{AN_ideal_adaptive_alpha}. We also get \eqref{moment_conv_adaptive_sigma} and \eqref{moment_conv_adaptive_alpha}. 
Now, we show the joint convergence of \eqref{AN_ideal_adaptive_sigma} - \eqref{AN_ideal_adaptive_alpha}. To apply Theorem 10 (b) in Yoshida \cite{17},
we define the function $\tilde{H}(\mu_{1},\sigma,\mu_{2},\alpha) : \overline{\Theta}_{2} \times \overline{\Theta}_{1} \times \overline{\Theta}_{2} \times \overline{\Theta}_{3} \to \mathbb{R}$ by 
\begin{align}
    &\tilde{H}(\mu_{1},\sigma,\mu_{2},\alpha) := \tilde{\Psi}_{n,\varepsilon}^{(0)}(\mu_{1}) + \tilde{\Psi}_{n,\varepsilon}^{(1)}(\mu_{1},\sigma) + \tilde{\Psi}_{n,\varepsilon}^{(1)}(\mu_{2},\sigma) + \tilde{\Psi}_{\lambda}^{(2)}(\alpha), \nonumber  
\end{align}
where $(\mu_{1},\sigma,\mu_{2},\alpha) \in \overline{\Theta}_{2} \times \overline{\Theta}_{1} \times \overline{\Theta}_{2} \times \overline{\Theta}_{3}$.
In addition, based on the above symbol,
\begin{align}
    &\tilde{Z}_{n,\varepsilon,\lambda}[u_{1},u_{2},u_{3},u_{4}] := \tilde{H}(\varepsilon u_{1} + \mu_{0},u_{2}/\sqrt{n} + \sigma_{0},\varepsilon u_{3} + \mu_{0},u_{4}/\sqrt{\lambda} + \alpha_{0}) - \tilde{H}(\mu_{0},\sigma_{0},\mu_{0},\alpha_{0}), \nonumber 
\end{align}
where $(u_{1},u_{2},u_{3},u_{4}) \in V_{\varepsilon}(\mu_{0}) \times  V_{n}(\sigma_{0}) \times V_{\varepsilon}(\mu_{0}) \times V_{\lambda}(\alpha_{0})$.
In the view of theorem 10 (c) in Yoshdia \cite{17}, it suffices to prove the following:
\begin{itemize}
    \item For any $R > 0$ and $u_{0},u_{1},u_{2},u_{3}$ with 
    $\lvert u_{0} \rvert^{2} + \lvert u_{1} \rvert^{2} + \lvert u_{2} \rvert^{2} + \lvert u_{3} \rvert^{2} \leq R^{2}$ ,
    \begin{align}
        &\begin{pmatrix}
            \log \tilde{Z}_{n,\varepsilon}^{(0)}[u_{0}]\\
            \log \tilde{Z}_{n,\varepsilon}^{(1)}[u_{1}]\\
            \log \tilde{Z}_{n,\varepsilon}^{(2)}[u_{2}]\\
            \log \tilde{Z}_{n,\lambda}^{(3)}[u_{3}]
        \end{pmatrix}
        \ind 
        \begin{pmatrix}
            \xi_{0}[u_{0}] - \frac{1}{2}I_{0}(\mu_{0})[u_{0}^{\otimes 2}]\\
            \xi_{1}[u_{1}] - \frac{1}{2}I_{0}(\mu_{0})[u_{1}^{\otimes 2}]\\
            \xi_{2}[u_{2}] - \frac{1}{2}I_{0}(\mu_{0})[u_{2}^{\otimes 2}]\\
            \xi_{3}[u_{3}] - \frac{1}{2}I_{0}(\mu_{0})[u_{3}^{\otimes 2}]
        \end{pmatrix}
        .
        \label{joint_conv}
    \end{align}
    \item For any $L > 0$, there exists a constant $C_{L} > 0$ such that for any $r > 0$,
    \begin{align}
        &\mathbb{P}\bigg(\sup_{\substack{(u_{0},u_{1},u_{2},u_{3})\\ \in V_{\varepsilon}(\mu_{0}) \times V_{\varepsilon}(\sigma_{0}) \times V_{\varepsilon}(\mu_{0}) \times V_{\lambda}(\alpha_{0})}}\tilde{Z}_{n,\varepsilon,\lambda}^{(0)}(u_{0},u_{1},u_{2},u_{3}) \geq \exp\{-r/2\} \bigg) \leq \frac{C_{L}}{r^{L}}.\label{joint_large_devi_ineq}
    \end{align}
\end{itemize}
For \eqref{joint_conv}, it can be proven similarly to Ogihara and Yoshida \cite{11}
 from \eqref{dist_conv_contrast_z0}, \eqref{dist_conv_contrast_z1}, \eqref{dist_conv_contrast_z2}, and \eqref{dist_conv_contrast_z3}. \eqref{joint_large_devi_ineq} follows from \eqref{large_devi_ineq_z0}, \eqref{large_devi_ineq_z1}, \eqref{large_devi_ineq_z2}, and \eqref{large_devi_ineq_z3}. Consequently, we have obtained the desired conclusion as the joint convergence of \eqref{AN_ideal_initial} and \eqref{AN_ideal_adaptive_sigma}
  - \eqref{AN_ideal_adaptive_alpha} follows.
\end{proof}
\subsection{Proof of Lemma \ref{moment_est_contrast0} - \ref{moment_est_contrast3}}
\begin{proof}[Proof of Lemma \ref{moment_est_contrast0}]
Let us start to prove Lemma \ref{moment_est_contrast0} (1). By Lemma \ref{moment_INEQ_diffusion} (1) , we obtain that for any $p \in \mathbb{N}$, 
\begin{align}
    \sup_{n,\varepsilon,\lambda}
    \mathbb{E}_{\theta_{0}}\bigg[ \bigg\lvert \varepsilon \partial_{\mu}\tilde{\Psi}_{n,\varepsilon}^{(0)}(\mu_{0}) \bigg \rvert^{p} \bigg]
     =&\,  \sup_{n,\varepsilon,\lambda}\mathbb{E}_{\theta_{0}}
     \bigg[\bigg \lvert\varepsilon^{-1}\sum_{k = 1}^{n}\partial_{\mu}a(X_{t_{k-1}^{n}}^{\theta,\varepsilon},\mu_{0})
     \tilde{\eta}_{k}^{n}(\mu_{0}) \bigg \rvert^{p}\bigg]  < \infty. \nonumber
\end{align}

Concerning \ref{moment_est_contrast0} (2), we have that, from Lemma \ref{moment_INEQ_diffusion} (2) and (3), that for any $p \in \mathbb{N}$,
\begin{align}
&\mathbb{E}_{\theta_{0}}\bigg[\bigg \lvert \varepsilon^{2} \partial_{\alpha}^{2}\tilde{\Psi}_{n,\varepsilon}^{(0)}(\mu_{0}) 
 + I_{0}(\mu;\mu_{0})\bigg \rvert^{p} \bigg] \nonumber\\
 \lesssim &\mathbb{E}_{\theta_{0}}\bigg[\bigg \lvert \sum_{k = 1}^{n}\partial_{\mu}^{2}a(X_{t_{k-1}^{n}}^{\theta,\varepsilon},\mu_{0})\tilde{\eta}_{k}^{n}(\mu_{0}) \bigg \rvert^{p} \bigg] \nonumber\\
 +&\mathbb{E}_{\theta_{0}}\bigg[\bigg \lvert
  \frac{1}{n}\sum_{k = 1}^{n}\big(\partial_{\mu}a(X_{t_{k-1}^{n}}^{\theta,\varepsilon},\mu_{0}) \big)^{2}\bm{1}_{J_{k,0}^{n}} - I_{0}(\mu;\mu_{0})
 \bigg \rvert^{p} \bigg] \nonumber\\
 =&\,O(\varepsilon^{p}). \nonumber
\end{align}
Thus, we find that for any $\eta_{0} \in (0,1/2)$ and for any $p \in \mathbb{N}$,
\begin{align}
\sup_{n,\varepsilon,\lambda}\mathbb{E}_{\theta_{0}}\bigg[\bigg(\varepsilon^{\eta}\bigg \lvert \varepsilon^{2} \partial_{\alpha}^{2}\tilde{\Psi}_{n,\varepsilon}^{(0)}(\mu_{0}) 
 + I_{0}(\mu;\mu_{0})\bigg \rvert \bigg)^{p} \bigg]  \leq& \sup_{n,\varepsilon,\lambda}(\varepsilon^{(1 - \eta)p}) < \infty. \nonumber
\end{align}

Now, we show \ref{moment_est_contrast0} (3). We have that for any $p \in \mathbb{N}$,
\begin{align}
&\sup_{n,\varepsilon,\lambda}\mathbb{E}_{\theta_{0}}\bigg[\sup_{\theta \in \Theta} \bigg\lvert \varepsilon^{2}\partial_{\mu}^{3}\tilde{\Psi}_{n,\varepsilon}^{(0)}(\mu)\bigg \rvert^{p}\bigg] \nonumber\\
\lesssim&\,\mathbb{E}_{\theta_{0}} \bigg[\sup_{\theta \in \Theta}\bigg\lvert\sum_{k = 1}^{n}\partial_{\mu}^{3} a(X_{t_{k-1}^{n}}^{\theta,\varepsilon},\mu)
\tilde{\eta}_{k}^{n}(\mu_{0}) \bigg\rvert^{p}\bigg]\nonumber\\
 +& \mathbb{E}_{\theta_{0}}\bigg[\sup_{\theta \in \Theta}\bigg\lvert \frac{1}{n}\sum_{k = 1}^{n}\bigg(
 \partial_{\mu}^{2}a(X_{t_{k-1}^{n}}^{\theta,\varepsilon},\mu)\partial_{\mu}
 a(X_{t_{k-1}^{n}}^{\theta,\varepsilon},\mu)
 + \partial_{\mu}^{3} a^{2}(X_{t_{k-1}^{n}}^{\theta,\varepsilon},\mu)
 \bigg)
 \bigg\rvert^{p}\bigg] \nonumber\\
<&\infty. \nonumber
\end{align}
\end{proof}
\begin{proof}[Proof of Lemma \ref{moment_est_contrast1}]
At first, we are going to show \ref{moment_est_contrast1} (1). 
From Lemma \ref{moment_INEQ_diffusion} (1), \ref{moment_INEQ_diffusion} (4) and \ref{moment_INEQ_diffusion} (5), 
we obtain that for any $p \in \mathbb{N}$,
\begin{align}
&\mathbb{E}_{\theta_{0}}\bigg[\bigg\lvert \frac{1}{\sqrt{n}} \partial_{\sigma}\tilde{\Psi}_{n,\varepsilon}(\sigma_{0},\tilde{\mu}_{n,\varepsilon}) \bigg\rvert^{p}\bigg] \nonumber\\
\lesssim\,&\,\mathbb{E}_{\theta_{0}}\bigg[\bigg\lvert
\frac{\sqrt{n}}{\varepsilon^{2}}\sum_{k = 1}^{n} \frac{1}{2}
\bigg(\tilde{\eta}_{k}^{n}(\mu_{0})^{2} - \frac{\varepsilon^{2}}{n}b^{2}(X_{t_{k-1}^{n}},\sigma_{0})\bigg)\partial_{\sigma}b^{2}(X_{t_{k-1}^{n}},\sigma_{0})^{-1} 
\bigg\rvert^{p} \bigg] \nonumber\\
+\,&\,\mathbb{E}_{\theta_{0}}\bigg[\bigg(\frac{1}{\varepsilon\sqrt{n}}\bigg \lvert 
\varepsilon^{-1}\sum_{k = 1}^{n}
 \tilde{\eta}_{k}^{n}(\mu_{0})
 \big(a(X_{t_{k-1}^{n}}^{\theta,\varepsilon},\tilde{\mu}_{n,\varepsilon}) -  a(X_{t_{k-1}^{n}}^{\theta,\varepsilon},\mu_{0})\big) \partial_{\sigma} b^{2}(X_{t_{k-1}^{n}}^{\theta,\varepsilon},\sigma_{0})^{-1} 
\bigg\rvert\bigg)^{p}\bigg] \nonumber\\
+\,&\,\mathbb{E}_{\theta_{0}}
\bigg[\bigg(\frac{1}{\sqrt{n}}
\bigg\lvert\frac{1}{n}\sum_{k = 1}^{n}\partial_{\sigma}b^{2}(X_{t_{k-1}^{n}}^{\theta,\varepsilon},\sigma)^{-1}\bigg(
    \int_{0}^{1}\partial_{\mu}a\big(X_{t_{k-1}^{n}}^{\theta,\varepsilon},\mu_{0} + s(\tilde{\mu}_{n,\varepsilon} - \mu_{0})\big)ds\big) \cdot \varepsilon^{-1}(\tilde{\mu}_{n,\varepsilon}- \mu_{0}) \bigg)^{2}  
\bigg \rvert \bigg)^{p}\bigg] \nonumber\\
=\,&\,O(1) + O(1/(\sqrt{n}\varepsilon)^{p}) + O(1/n^{p/2}), \nonumber
\end{align}
which implies Lemma \ref{moment_est_contrast1} (1).

Secondary, we show Lemma \ref{moment_est_contrast1} (2).
 It is easy to obtain that for any $p \in \mathbb{N}$,
\begin{align}
&\,\mathbb{E}_{\theta_{0}}\bigg[\bigg(n^{\eta_{2}}\bigg\lvert 
    \frac{1}{n}\partial_{\sigma}^{2}\Psi_{n,\varepsilon}^{(1)}(\sigma_{0},\tilde{\mu}_{n,\varepsilon})
     + I_{1}(\sigma_{0})\bigg \rvert \bigg)^{p}\bigg]  \nonumber\\
     \lesssim\,& \mathbb{E}_{\theta_{0}}\bigg[\bigg\lvert\frac{1}{2}\sum_{k = 1}^{n}\frac{\tilde{\eta}_{k}^{n}(\mu_{0})^{2} - \varepsilon^{2}\frac{1}{n}b^{2}(X_{t_{k-1}^{n}},\sigma_{0})\bm{1}_{J_{k,0}^{n}}}{\frac{1}{n} \varepsilon^{2}} \partial_{\sigma}^{2}b^{2}(X_{t_{k-1}^{n}}^{\theta,\varepsilon},\sigma_{0}) b^{2}(X_{t_{k-1}^{n}}^{\theta,\varepsilon},\sigma_{0})^{-2}\bigg \rvert^{p} \bigg] \nonumber\\
     +\,&\,\mathbb{E}_{\theta_{0}}\bigg[ \bigg\lvert \sum_{k = 1}^{n}\frac{\tilde{\eta}_{k}^{n}(\mu_{0})^{2} - \varepsilon^{2}\frac{1}{n}b^{2}(X_{t_{k-1}^{n}},\sigma_{0})\bm{1}_{J_{k,0}^{n}} }{\frac{1}{n} \varepsilon^{2}} \partial_{\sigma}b^{2}(X_{t_{k-1}^{n}}^{\theta,\varepsilon},\sigma_{0})^{2}b^{2}(X_{t_{k-1}^{n}}^{\theta,\varepsilon},\sigma_{0})^{-3} \bigg\rvert^{p}\bigg] \nonumber\\
     +\,&\,\mathbb{E}_{\theta_{0}}\bigg[ \bigg \lvert I_{1}(\sigma_{0})-\frac{1}{2}\sum_{k = 1}^{n}\partial_{\sigma}b^{2}(X_{t_{k-1}^{n}}^{\theta,\varepsilon},\sigma_{0})^{2} b^{2}(X_{t_{k-1}^{n}}^{\theta,\varepsilon},\sigma_{0})^{-2} \bm{1}_{J_{k,0}^{n}}\bigg \rvert^{p} \bigg] \nonumber\\
+\,&\,\mathbb{E}_{\theta_{0}}\bigg[\bigg(\frac{\varepsilon}{n}\bigg \lvert 
\varepsilon^{-1} \sum_{k = 1}^{n}\tilde{\eta}_{k}^{n}(\mu_{0})\big(a(X_{t_{k-1}^{n}}^{\theta,\varepsilon},\tilde{\mu}_{n,\varepsilon}) -  a(X_{t_{k-1}^{n}}^{\theta,\varepsilon},\mu_{0})\big) \partial_{\sigma} b^{2}(X_{t_{k-1}^{n}}^{\theta,\varepsilon},\sigma_{0})^{-1} 
\bigg \rvert\bigg)^{p}\bigg] \nonumber\\
+\,&\,\mathbb{E}_{\theta_{0}}\bigg[\bigg(\frac{1}{n}
\bigg\lvert \frac{1}{n}\sum_{k = 1}^{n}\partial_{\sigma}b^{2}(X_{t_{k-1}^{n}}^{\theta,\varepsilon},\sigma)^{-1}\bigg( \int_{0}^{1}\partial_{\mu}
a\big(X_{t_{k-1}^{n}}^{\theta,\varepsilon},\mu_{0} + s(\tilde{\mu}_{n,\varepsilon} - \mu_{0})\big)ds \varepsilon^{-1}(\tilde{\mu}_{n,\varepsilon}- \mu_{0})\bigg)^{2}  \bigg\rvert \bigg)^{p}\bigg] \nonumber\\
=&\,O(n^{-p/2}) + O(\varepsilon^{p}/n^{p}) + O(1/n^{p}). \nonumber
\end{align}
Thus, we obtain Lemma \ref{moment_est_contrast1} (2) from \Aref{a9}.

Now, we show Lemma \ref{moment_est_contrast1} (3). Applying Lemma \ref{moment_INEQ_diffusion} (5),(6) and Lemma \ref{moment_INEQ_diffusion} (1) and Lemma \ref{approx_for_conditional_e} (2), we have that for any $p \in \mathbb{N}$,
\begin{align}
&\mathbb{E}_{\theta_{0}}\bigg[\bigg\lvert\frac{1}{n}\partial_{\sigma}^{3}\tilde{\Psi}_{n,\varepsilon}^{(1)}(\sigma,\tilde{\mu}_{n,\varepsilon})  \bigg \rvert^{p}\bigg]\nonumber\\
    \lesssim &\, \mathbb{E}_{\theta_{0}}\bigg[\bigg \lvert 
    -\frac{1}{2\varepsilon^{2}}\sum_{k = 1}^{n}\partial_{\sigma}^{3} 
    b^{2}(X_{t_{k-1}^{n}}^{\theta,\varepsilon},\sigma) 
     \bigg\{\tilde{\eta}_{k}^{n}(\sigma_{0})^{2} - \varepsilon^{2}\frac{1}{n}b^{2}
     (X_{t_{k-1}^{n}}^{\theta,\varepsilon},\sigma_{0})\bigg\}
    \bigg\rvert^{p} \bigg] \nonumber\\
      + &\, \mathbb{E}_{\theta_{0}}\bigg[\bigg\lvert \frac{1}{2n}\sum_{k = 1}^{n}\bigg\{
        \partial_{\sigma}^{3}
        b^{2}(X_{t_{k-1}^{n}}^{\theta,\varepsilon},\sigma)  b^{2}(X_{t_{k-1}^{n}}^{\theta,\varepsilon},
        \sigma_{0})
         + \partial_{\sigma}^{3}\log b^{2}(X_{t_{k-1}^{n}}^{\theta,\varepsilon},\sigma)\bigg\}
         \bm{1}_{J_{k,0}^{n}}\bigg\rvert^{p} \bigg] \nonumber\\
        + &\,\mathbb{E}_{\theta_{0}}\bigg[\bigg\lvert
         - \frac{1}{n\varepsilon^{2}}\sum_{k = 1}^{n}\big( 
        a(X_{t_{k-1}^{n}}^{\theta,\varepsilon},\mu) - a(X_{t_{k-1}^{n}}^{\theta,\varepsilon},\mu_{0})
        \big)\partial_{\sigma}^{3}\big(b^{2}(X_{t_{k-1}^{n}}^{\theta,\varepsilon},\sigma)^{-1} \big)\tilde{\eta}_{k}^{n}(\mu_{0}) 
        \bigg\rvert^{p} \bigg] \nonumber\\
       + &\,\mathbb{E}_{\theta_{0}}\bigg[\bigg\lvert
       - \frac{1}{2n\varepsilon^{2}} \frac{1}{n}\sum_{k = 1}^{n}
    \big(a(X_{t_{k-1}^{n}}^{\theta,\varepsilon},\tilde{\mu}_{n,\varepsilon}) 
    - a(X_{t_{k-1}^{n}}^{\theta,\varepsilon},\mu_{0}) \big)^{2}\partial_{\sigma}^{3}
      \big( b^{2}(X_{t_{k-1}^{n}}^{\theta,\varepsilon},\sigma)^{-1}\big)\bm{1}_{J_{k,0}^{n}} \bigg\rvert^{p} \bigg] \nonumber\\
      =\,&\, O(1/n^{p/2}) + O(1) + O(1/(n\varepsilon^{2})^{p}). \nonumber
\end{align}
Thus, we obtain Lemma \ref{moment_est_contrast1} (3).
\end{proof}
\begin{proof}[Proof of Lemma \ref{moment_est_contrast2}]
As in the proof of Lemma \ref{moment_est_contrast0} (1) and (3),
we obtain Lemma \ref{moment_est_contrast2} (1) and (3).
Thus, we only show Lemma \ref{moment_est_contrast2} (2).
From Lemma \ref{moment_INEQ_diffusion} (1), (2) and (3) and \eqref{moment_conv_adaptive_sigma},
we have that for any $p \in \mathbb{N}$,
\begin{align}
    &\,\mathbb{E}_{\theta_{0}}\bigg[\bigg(\varepsilon^{-\eta_{2}} \bigg \lvert
     \varepsilon^{2} \partial_{\mu}\tilde{\Psi}_{n,\varepsilon}^{(1)}(\tilde{\sigma}_{n,\varepsilon},\mu_{0})
     + I_{2}(\mu_{0}) \bigg \rvert \bigg)^{p}  \bigg] \nonumber\\
   \lesssim\,&\, \mathbb{E}_{\theta_{0}}\bigg[ \bigg( \varepsilon^{-\eta_{2}}
   \bigg \lvert - \sum_{k = 1}^{n}\tilde{\eta}_{k}^{n}(\mu_{0})
         \partial_{\mu}^{2}a(X_{t_{k-1}^{n}}^{\theta,\varepsilon},\mu_{0}) 
         b^{2}(X_{t_{k-1}^{n}}^{\theta,\varepsilon},\sigma_{0})^{-1} \bigg\rvert \bigg)^{p}\bigg] \nonumber\\ 
   + \,&\, \mathbb{E}_{\theta_{0}}\bigg[ \bigg( \varepsilon^{-\eta_{2}}\bigg \lvert 
   \frac{1}{n^{3/2}}\sum_{k = 1}^{n} 
        \big(\partial_{\mu}a(X_{t_{k-1}^{n}}^{\theta,\varepsilon},\mu_{0})\big)^{2}\int_{0}^{ 1}\partial_{\sigma}
            \big(b^{2}\big(X_{t_{k-1}^{n}}^{\theta,\varepsilon},\sigma_{0} + s(\tilde{\sigma}_{n,\varepsilon} - \sigma_{0}) \big)^{-1}\big) ds\big(\sqrt{n}(\tilde{\sigma}_{n,\varepsilon} - \sigma_{0}) \big)\bm{1}_{J_{k,0}^{n}}  
    \bigg\rvert \bigg)^{p}\bigg] \nonumber\\
    + \,&\,\mathbb{E}_{\theta_{0}}\bigg[ \bigg( \varepsilon^{-\eta_{2}}\bigg \lvert 
     - \frac{1}{n}\sum_{k = 1}^{n} \big(\partial_{\mu}a(X_{t_{k-1}^{n}}^{\theta,\varepsilon},
        \mu_{0})\big)^{2}b^{2}\big(X_{t_{k-1}^{n}}^{\theta,\varepsilon},\sigma_{0}\big)^{-1}\bm{1}_{J_{k,0}^{n}}
    + I_{2}(\mu_{0})\bigg\rvert \bigg)^{p}\bigg]\nonumber\\
     =\,&\,  O(1/(\varepsilon^{\eta_{2}}\sqrt{n})^p) 
     + O(\varepsilon^{(1 - \eta_{2})p}/(\sqrt{n}\varepsilon)^{p})
      +O(\varepsilon^{(1 -\eta_{2})p}). \nonumber
\end{align}
From \Aref{a9}, we obtain the proof of Lemma \ref{moment_est_contrast2} (2).
\end{proof}
\begin{proof}[Proof of Lemma \ref{moment_est_contrast3}]
    Set for any $n \in \mathbb{N}$ and $k = 1,...,n$,
    \begin{align}
        \psi_{k-1}( c(X_{t_{k-1}^{n}}^{\theta,\varepsilon},\alpha_{0})V_{\tau_{k}},\alpha)&
        := \psi\big(X_{t_{k-1}^{n}}^{\theta,\varepsilon},c(X_{t_{k-1}^{n}}^{\theta,\varepsilon},\alpha_{0})V_{\tau_{k}},\alpha\big). \nonumber
    \end{align}
    Let us start to show Lemma \ref{moment_est_contrast3} (1). As in Lemma 4.17 in Kobayashi and  Shimizu \cite{9}, we have that
    \begin{align}
        &\,\frac{1}{\sqrt{\lambda}} \sum_{k = 1}^{n} \mathbb{E}_{\theta_{0}}\bigg[ 
            \partial_{\alpha} \psi_{k-1}(c(X_{t_{k-1}^{n}}^{\theta,\varepsilon},\alpha_{0})V_{\tau_{k}},\alpha) \mid \mathcal{F}_{t_{k-1}^{n}}\bigg] = 0, \nonumber
    \end{align}
    and as in the proof of Lemma \ref{moment_ineq_jump}, we have that for any $p \in \mathbb{N}$,
    \begin{align}
        &\mathbb{E}_{\theta_{0}}\bigg[\bigg\lvert \sum_{k = 1}^{n}\frac{1}{\sqrt{\lambda}}
        \psi_{k-1}(c(X_{t_{k-1}^{n}}^{\theta,\varepsilon},\alpha_{0})V_{\tau_{k}},\alpha_{0}) \bigg\rvert^{p} \bigg] \lesssim \mathbb{E}_{\theta_{0}}\bigg[\bigg \lvert \frac{1}{\lambda}\sum_{k = 1}^{n}\psi_{k-1}^{2}(c(X_{t_{k-1}^{n}}^{\theta,\varepsilon},\alpha_{0})V_{\tau_{k}},\alpha_{0}) \bigg \rvert^{p/2} \bigg] < \infty, \nonumber
    \end{align}
    which completes the proof of Lemma \ref{moment_est_contrast2} (1).
    
    Concerning Lemma \ref{moment_est_contrast2} (2), as in the proof of Lemma \ref{moment_ineq_jump}, we obtain that for any $p \in \mathbb{N}$,
\begin{align}
    &\,\mathbb{E}_{\theta_{0}}\bigg[\bigg \lvert \sum_{k = 1}^{n}(1/\lambda) 
    \partial_{\alpha}^{2}\psi_{k-1}(c(X_{t_{k-1}^{n}}^{\theta,\varepsilon},\alpha_{0})V_{\tau_{k}},\alpha_{0})
     + I_{3}(\alpha_{0}) \bigg \rvert^{p} \bigg] \nonumber\\
   \lesssim \,&\,\mathbb{E}_{\theta_{0}}\bigg[\bigg \lvert \sum_{k = 1}^{n}\bigg\{(1/\lambda) 
    \partial_{\alpha}^{2}\psi_{k-1}(c(X_{t_{k-1}^{n}}^{\theta,\varepsilon},\alpha_{0})V_{\tau_{k}},\alpha_{0}) 
    \nonumber\\
    &\,- \mathbb{E}_{\theta_{0}}\bigg[ (1/\lambda) 
    \partial_{\alpha}^{2}\psi_{k-1}(c(X_{t_{k-1}^{n}}^{\theta,\varepsilon},\alpha_{0})V_{\tau_{k}},\alpha_{0})  \mid \mathcal{F}_{t_{k-1}^{n}}\bigg]\bigg\}^{2}\bigg \rvert^{p/2}\bigg] \nonumber\\
   +\,&\,\mathbb{E}_{\theta_{0}}\bigg[\bigg \lvert 
    \frac{1}{n}\sum_{k = 1}^{n}\int_{\mathbb{R}}\partial_{\alpha}^{2}\psi_{k-1} 
    (c(X_{t_{k-1}^{n}}^{\theta,\varepsilon},\alpha_{0})y,\alpha_{0})f_{\alpha_{0}}(y)dy + I_{3}(\alpha_{0}) \bigg \rvert^{p} \bigg]. \nonumber
\end{align}
From Lemma \ref{moment_ineq_jump}, we have that for any $p \in \mathbb{N}$,
    \begin{align}
        &\mathbb{E}_{\theta_{0}}\bigg[\bigg \lvert 
        \sum_{k = 1}^{n}\bigg\{(1/\lambda) 
        \partial_{\alpha}^{2}\psi_{k-1}(c(X_{t_{k-1}^{n}}^{\theta,\varepsilon},\alpha_{0})V_{\tau_{k}},\alpha_{0}) 
        \nonumber\\
        &- \mathbb{E}_{\theta_{0}}\bigg[ (1/\lambda) 
        \partial_{\alpha}^{2}\psi_{k-1}(c(X_{t_{k-1}^{n}}^{\theta,\varepsilon},\alpha_{0})V_{\tau_{k}},\alpha_{0})  \mid \mathcal{F}_{t_{k-1}^{n}}\bigg]\bigg\}^{2}\bigg \rvert^{p/2}\bigg] = O(1/\lambda^{p/2}). \nonumber  
    \end{align} 
From \Aref{a2} and \ref{r_a_2}, and Lemma \ref{approx_for_conditional_e} (4), we also get that for any $p \in \mathbb{N}$,
    \begin{align}
        &\, \mathbb{E}_{\theta_{0}}\bigg[\bigg \lvert \frac{1}{n}\sum_{k = 1}^{n}\int_{\mathbb{R}}\partial_{\alpha}^{2}\psi_{k-1} 
        (c(X_{t_{k-1}^{n}}^{\theta,\varepsilon},\alpha_{0})y,\alpha_{0})f_{\alpha_{0}}(y)dy
         + I_{3}(\alpha_{0})\bigg\rvert^{p}\bigg] \nonumber\\
        &\lesssim \,\sum_{i = 0}^{2}\mathbb{E}_{\theta_{0}}\bigg[\bigg(\sum_{k = 1}^{n}
            \int_{t_{k-1}^{n}}^{t_{k}^{n}}\bigg \lvert \partial_{\alpha}^{i}c(x_{s},\alpha_{0}) 
        - \partial_{\alpha}^{i}c(X_{t_{k-1}^{n}}^{\theta,\varepsilon},\alpha_{0})\bigg \rvert ds\bigg)^{p}\bigg] \nonumber\\
        &\lesssim \mathbb{E}_{\theta_{0}}\bigg[\sup_{t,s \in [0,1], \lvert t -s \rvert < 1/n}\bigg \lvert x_{s} - X_{t}^{\theta,\varepsilon} \bigg \rvert^{p}\bigg] = O\big(1/ n^{p}+ (\varepsilon \lambda)^{p}\big). \nonumber
    \end{align}
    From \Aref{a9}, we also find that for $\eta_{3} \in (0,p_{2})$,
\begin{align}
&\sup_{n,\varepsilon,\lambda}\mathbb{E}_{\theta_{0}}\bigg[\bigg( \lambda^{\eta_{3}} \lvert \frac{1}{n}\sum_{k = 1}^{n}\int_{\mathbb{R}}\partial_{\alpha}^{2}\psi_{k-1} 
        (c(X_{t_{k-1}^{n}}^{\theta,\varepsilon},\alpha_{0})y,\alpha_{0})f_{\alpha_{0}}(y)dy
         + I_{3}(\alpha_{0})\bigg\rvert \bigg)^{p}\bigg] < \infty. \nonumber
\end{align}
Thus, we obtain \ref{moment_est_contrast3} (2). 
Concerning Lemma \ref{moment_est_contrast3} (3), applying Lemma \ref{moment_ineq_jump}, we find that 
for any $p \in \mathbb{N}$,
\begin{align}
        &\,\sup_{n,\varepsilon,\lambda}\sup_{\alpha \in \overline{\Theta}_{3}}\mathbb{E}_{\theta_{0}}\bigg[\bigg\lvert \frac{1}{\lambda}\partial_{\alpha}^{3}\tilde{\Psi}_{\lambda}^{(2)}(\alpha)\bigg\rvert^{p}\bigg],\quad \,\sup_{n,\varepsilon,\lambda}\sup_{\alpha \in \overline{\Theta}_{3}}\mathbb{E}_{\theta_{0}}\bigg[\bigg\lvert \frac{1}{\lambda}\partial_{\alpha}^{4}\tilde{\Psi}_{\lambda}^{(2)}(\alpha)\bigg\rvert^{p}\bigg] < \infty. \nonumber
\end{align}
Then, by Burkholder and Davis inequality, we obtain the desired result.
\end{proof}
\subsection{proof of Lemma \ref{approx_for_conditional_e} - \ref{moment_ineq_jump}}
\begin{proof}[Proof of Lemma \ref{approx_for_conditional_e}]
    We show Lemma \ref{approx_for_conditional_e} (1). As in Lemma 6 \cite{8}, we have that for any $p \in \mathbb{N}$,
\begin{align}
    &\mathbb{E}_{\theta_{0}}
    \bigg[\big\lvert \Delta_{k}^{n}X^{\theta,\varepsilon}\bm{1}_{J_{k,0}^{n}}\big \rvert 
    \mid \mathcal{F}_{t_{k-1}^{n}} \bigg] \nonumber\\
    \leq& \bigg(\frac{1}{n} + 
    \frac{\varepsilon}{\sqrt{n}}\bigg)\int_{t_{k-1}^{n}}^{t_{k}^{n}}\mathbb{E}_{\theta_{0}}
    \bigg[\lvert X_{s}^{\theta,\varepsilon} - X_{t_{k-1}^{n}}^{\theta,\varepsilon} \rvert \bm{1}_{J_{k,0}^{n}}\bigg \rvert \mathcal{F}_{t_{k-1}^{n}}\bigg]  + \frac{1}{n}\lvert a(X_{t_{k-1}^{n}}^{\theta,\varepsilon},\mu_{0}) \rvert + \frac{\varepsilon}{\sqrt{n}} \lvert b(X_{t_{k-1}^{n}}^{\theta,\varepsilon},\sigma_{0}) \rvert,  \nonumber
\end{align}
and by Grownwall's inequality, we obtain Lemma \ref{approx_for_conditional_e} (1).

We show Lemma \ref{approx_for_conditional_e} (2). It is easy to obtain that 
\begin{align}
   \mathbb{E}_{\theta_{0}}
   \bigg[\tilde{\eta}_{k}^{n}(\mu_{0}) \mid \mathcal{F}_{t_{k-1}^{n}} \bigg]  & = \mathbb{E}_{\theta_{0}}\bigg[\int_{t_{k-1}^{n}}^{t_{k}^{n}}\bigg(a(X_{s}^{\theta,\varepsilon},\mu_{0}) - a(X_{t_{k-1}^{n}}^{\theta,\varepsilon},\mu_{0})\bigg) 
   \bm{1}_{J_{k,0}^{n}}ds \mid \mathcal{F}_{t_{k-1}^{n}} \bigg]. \nonumber
\end{align}
By Lemma \ref{approx_for_conditional_e} (1), we obtain 
\begin{align}
    &\,\mathbb{E}_{\theta_{0}}\bigg[ \int_{t_{k-1}^{n}}^{t_{k}^{n}}\bigg(a(X_{s}^{\theta,\varepsilon},\mu_{0}) - a(X_{t_{k-1}^{n}}^{\theta,\varepsilon},\mu_{0})\bigg) \bm{1}_{J_{k,0}^{n}}ds  \mid \mathcal{F}_{t_{k-1}^{n}} \bigg] = R(X_{t_{k-1}^{n}}^{\theta,\varepsilon},\frac{1}{n^{2}}), \nonumber
\end{align}
which immediately implies the result of Lemma \ref{approx_for_conditional_e} (2).

We prove Lemma \ref{approx_for_conditional_e} (3) imitating the proof of Kaino and Uchida \cite{4} and we have that
\begin{align}
    &\, \bigg\lvert \mathbb{E}_{\theta_{0}}\bigg[\bigg(\tilde{\eta}_{k}^{n}(\mu_{0})^{2} - \frac{\varepsilon^{2}}{n}b^{2}(X_{t_{k-1}^{n}}^{\theta,\varepsilon},\sigma_{0})\bigg)\bm{1}_{J_{k,0}^{n}} \mid \mathcal{F}_{t_{k-1}^{n}}\bigg] \bigg\rvert \nonumber\\
    \lesssim \,&\, \mathbb{E}_{\theta_{0}}\bigg[\bigg\{ \int_{t_{k-1}^{n}}^{t_{k}^{n}} \big(a(X_{s}^{\theta,\varepsilon},\mu_{0}) - a(X_{t_{k-1}^{n}},\mu_{0})\big)\bm{1}_{J_{k,0}^{n}} ds \bigg \}^{2} \mid \mathcal{F}_{t_{k-1}^{n}}\bigg] \nonumber\\
    +\,&\, \varepsilon^{2} \mathbb{E}_{\theta_{0}}\bigg[\bigg\{\int_{t_{k-1}^{n}}^{t_{k}^{n}}\big(b(X_{s}^{\theta,\varepsilon},\sigma_{0}) - b(X_{t_{k-1}^{n}},\sigma_{0})\big)\bm{1}_{J_{k,0}^{n}}dW_{s}\bigg\}^{2}\mid \mathcal{F}_{t_{k-1}^{n}}\bigg] \nonumber\\
    +\,&\, 2\varepsilon \lvert b(X_{t_{k-1}^{n}}^{\theta,\varepsilon},\sigma_{0}) \rvert \mathbb{E}_{\theta_{0}}\bigg[\bigg\lvert \Delta_{k}^{n}W \bigg\{ \int_{t_{k-1}^{n}}^{t_{k}^{n}} \big(a(X_{s}^{\theta,\varepsilon},\mu_{0}) - a(X_{t_{k-1}^{n}},\mu_{0})\big)\bm{1}_{J_{k,0}^{n}}ds\bigg\} \bigg \rvert \mid \mathcal{F}_{t_{k-1}^{n}}\bigg]\nonumber\\
    +\,&\, 2\varepsilon^{2} \lvert b(X_{t_{k-1}^{n}}^{\theta,\varepsilon},\sigma_{0}) \rvert \mathbb{E}_{\theta_{0}}\bigg[\bigg\lvert \Delta_{k}^{n}W \bigg\{ \int_{t_{k-1}^{n}}^{t_{k}^{n}} \big(b(X_{s}^{\theta,\varepsilon},\mu_{0}) - b(X_{t_{k-1}^{n}},\mu_{0})\big)\bm{1}_{J_{k,0}^{n}}dW_{s}\bigg\} \bigg \rvert \mid \mathcal{F}_{t_{k-1}^{n}}\bigg]\nonumber\\
    =&\, R(X_{t_{k-1}^{n}}^{\theta,\varepsilon},\frac{1}{n^{3}}) +R(X_{t_{k-1}^{n}}^{\theta,\varepsilon},\varepsilon^{2}\frac{1}{n^{3}}) +R(X_{t_{k-1}^{n}}^{\theta,\varepsilon},\varepsilon \frac{1}{n^{5/2}}) + R(X_{t_{k-1}^{n}}^{\theta,\varepsilon},\varepsilon^{2}\frac{1}{n^{2}}).\label{6.2.2}
\end{align}
We obtain Lemma \ref{approx_for_conditional_e} (3).

We obtain Lemma \ref{moment_INEQ_diffusion} (4) and (5) respectively as in the proof of Lemma 4.2 and 4.1 in Kobayashi and Shimizu \cite{9}. 
\end{proof}
\begin{proof}[Proof of Lemma \ref{moment_INEQ_diffusion}]
At first, we show Lemma \ref{moment_INEQ_diffusion} (1). As in the proof of Lemma 2 (1) in Kaino and Uchida \cite{4}, we obtain \ref{moment_INEQ_diffusion} (1).

We show Lemma \ref{moment_INEQ_diffusion} (2). It is easy to obtain that for any $p \in \mathbb{N}$, 
\begin{align}
    &\,\mathbb{E}_{\theta_{0}}\bigg[\sup_{\theta}\bigg\lvert \varepsilon^{-1} \bigg(\frac{1}{n}\sum_{k = 1}^{n}f(X_{t_{k-1}^{n}}^{\theta,\varepsilon},\theta) - \int_{0}^{1}f(x_{t},\theta)dt \bigg) \bigg\rvert^{p}\bigg] \nonumber\\
     \lesssim\,&\,\mathbb{E}_{\theta_{0}}\bigg[\sup_{\theta}\bigg\lvert \varepsilon^{-1} \sum_{k = 1}^{n}\int_{t_{k-1}^{n}}^{t_{k}^{n}}\bigg(f(X_{t_{k-1}^{n}}^{\theta,\varepsilon},\theta) - f(X_{t}^{\theta,\varepsilon},\theta)  \bigg)dt\bm{1}_{J_{k,0}^{n}}\bigg \rvert^{p} \bigg] \nonumber\\
    +\,&\,\mathbb{E}_{\theta_{0}}\bigg[\sup_{\theta} \bigg \lvert \varepsilon^{-1}\sum_{k = 1}^{n}\int_{t_{k-1}^{n}}^{t_{k}^{n}}\bigg(f(X_{t}^{\theta,\varepsilon},\theta) - f(x_{t},\theta)\bigg)ds\bm{1}_{J_{k,0}^{n}}\bigg \rvert^{p}\bigg] \nonumber\\
    +\,&\,\sup_{\theta \in \overline{\Theta}}\sup_{t \in [0,1]}\lvert f(x_{t},\theta) \rvert^{p} \mathbb{E}_{\theta_{0}}
    \bigg[\bigg(\varepsilon^{-1}\frac{1}{n}\sum_{k = 1}^{n}\bm{1}_{J_{k,1}^{n} \cup J_{k,2}^{n}}\bigg)^{p}\bigg]. \nonumber
\end{align}
At first, we show that for any $p \in \mathbb{N}$,
\begin{align}
    &\,\mathbb{E}_{\theta_{0}}\bigg[\sup_{\theta} \bigg\lvert \varepsilon^{-1}\sum_{k = 1}^{n}\int_{t_{k-1}^{n}}^{t_{k}^{n}}\bigg(f(X_{t_{k-1}^{n}}^{\theta,\varepsilon},\theta) - f(X_{t}^{\theta,\varepsilon},\theta)  \bigg)dt\bm{1}_{J_{k,0}^{n}}\bigg \rvert^{p} \bigg] < \infty. \nonumber
\end{align}
By Burkholder and Davis Inequality, it suffices to show that for any $p \in \mathbb{N}$,
\begin{align}
&\,\sup_{n,\varepsilon,\lambda}\sup_{\theta}\mathbb{E}_{\theta_{0}}\bigg[\bigg\lvert \varepsilon^{-1}\sum_{k = 1}^{n}\int_{t_{k-1}^{n}}^{t_{k}^{n}}\big(f(X_{t_{k-1}^{n}}^{\theta,\varepsilon},\theta) - f(X_{t}^{\theta,\varepsilon},\theta)  \big)dt\bm{1}_{J_{k,0}^{n}}\bigg \rvert^{p} \bigg]< \infty,\label{6.2.6} \\
&\,\sup_{n,\varepsilon,\lambda}\sup_{\theta}\mathbb{E}_{\theta_{0}}\bigg[\bigg\lvert \varepsilon^{-1}\sum_{k = 1}^{n}\int_{t_{k-1}^{n}}^{t_{k}^{n}}\big(\partial_{\theta}f(X_{t_{k-1}^{n}}^{\theta,\varepsilon},\theta) - \partial_{\theta}f(x_{t},\theta) \big)dt\bm{1}_{J_{k,0}^{n}}\bigg \rvert^{p}\bigg]<\infty.\label{6.2.7}
\end{align}
For \eqref{6.2.6}, 
from Taylor expansion, Lemma \ref{moment_INEQ_diffusion} (3) and Assumption $f \in C_{\uparrow}^{1,1}(\mathbb{R} \times \Theta)$, 
we have that 
\begin{align}
    &\,\sup_{\theta}E_{\theta_{0}}\bigg[\bigg\lvert \sum_{k = 1}^{n}\int_{t_{k-1}^{n}}^{t_{k}^{n}}\bigg(f(X_{t_{k-1}^{n}}^{\theta,\varepsilon},\theta) - f(X_{s}^{\theta,\varepsilon},\theta)\bigg)\bm{1}_{J_{k,0}^{n}}ds \bigg\rvert^{p}\bigg]\nonumber\\
    \lesssim\,&\,  (\varepsilon/\sqrt{n})^{p} n^{p} \frac{1}{n}\sum_{k = 1}^{n}\frac{1}{n^{p-1}}\int_{t_{k-1}^{n}}^{t_{k}^{n}}\mathbb{E}_{\theta_{0}}\bigg[\bigg((\varepsilon \frac{1}{n^{1/2}})^{-1}\sup_{t \in [t_{k-1}^{n},t_{k}^{n}]} \lvert X_{t}^{\theta,\varepsilon} - X_{t_{k-1}^{n}}^{\theta,\varepsilon} \rvert \bigg)^{2p}\bigg]^{1/2}dt \nonumber\\
    \lesssim\,&\,  (\varepsilon/\sqrt{n})^{p} n^{p} \frac{1}{n}\sum_{k = 1}^{n}\frac{1}{n^{p}} \mathbb{E}_{\theta_{0}}\bigg[\bigg((\varepsilon \frac{1}{n^{1/2}})^{-1}\sup_{t \in [t_{k-1}^{n},t_{k}^{n}]}\lvert X_{t}^{\theta,\varepsilon} - X_{t_{k-1}^{n}}^{\theta,\varepsilon} \rvert \bigg)^{2p}\bigg]^{1/2} \nonumber\\
    =\,&\,  O((\varepsilon/\sqrt{n})^{p}), \nonumber
\end{align}
which implies \eqref{6.2.6}. In the same way, we obtain \eqref{6.2.7}.

Second, let us show
\begin{align}
&\,\sup_{n,\varepsilon,\lambda}\mathbb{E}_{\theta_{0}}\bigg[\sup_{\theta} \bigg \lvert \varepsilon^{-1}\sum_{k = 1}^{n}\int_{t_{k-1}^{n}}^{t_{k}^{n}}\bigg(f(X_{t}^{\theta,\varepsilon},\theta) - f(x_{t},\theta)\bigg)ds\bm{1}_{J_{k,0}^{n}}\bigg \rvert^{p}\bigg] < \infty,\nonumber
\end{align}
which follows from that
\begin{align}
&\,\sup_{n,\varepsilon,\lambda}\sup_{\theta}\mathbb{E}_{\theta_{0}}\bigg[ \bigg \lvert 
\varepsilon^{-1}\sum_{k = 1}^{n}\int_{t_{k-1}^{n}}^{t_{k}^{n}}\bigg(f(X_{t}^{\theta,\varepsilon},\theta) 
- f(x_{t},\theta)\bigg)ds\bm{1}_{J_{k,0}^{n}}\bigg \rvert^{p}\bigg] < \infty,\label{6.2.9}\\    
&\sup_{n,\varepsilon,\lambda}\sup_{\theta}\mathbb{E}_{\theta_{0}}\bigg[ \bigg \lvert \varepsilon^{-1}\sum_{k = 1}^{n}\int_{t_{k-1}^{n}}^{t_{k}^{n}}\bigg(\partial_{x}f(X_{t}^{\theta,\varepsilon},\theta) - \partial_{x}f(x_{t},\theta)\bigg)ds\bm{1}_{J_{k,0}^{n}}\bigg \rvert^{p}\bigg] < \infty. \label{6.2.10}
\end{align}
From Taylor expansion, Lemma \ref{approx_for_conditional_e} (2) and Assumption $f \in C_{\uparrow}^{1,1}(\mathbb{R} \times \overline{\Theta};\mathbb{R})$ , we obtain that 
\begin{align}
    &\,\mathbb{E}_{\theta_{0}}\bigg[  \bigg \lvert \sum_{k = 1}^{n}\int_{t_{k-1}^{n}}^{t_{k}^{n}}\bigg(f(X_{t}^{\theta,\varepsilon},\theta) - f(x_{t},\theta)\bigg)ds\bm{1}_{J_{k,0}^{n}}\bigg \rvert^{p} \bigg] \nonumber\\
     \lesssim \,&\, \mathbb{E}_{\theta_{0}}\bigg[ \bigg \lvert \sum_{k = 1}^{n}\int_{t_{k-1}^{n}}^{t_{k}^{n}}\int_{\eta = 0}^{\eta = 1}\partial_{x}f\big(X_{s}^{\theta,\varepsilon} + \eta (X_{s}^{\theta,\varepsilon}- X_{t_{k-1}^{n}}^{\theta,\varepsilon}),\theta\big) (X_{s}^{\theta,\varepsilon} - X_{t_{k-1}^{n}}^{\theta,\varepsilon})\bm{1}_{J_{k,0}^{n}}d\eta ds \bigg \rvert\bigg]\nonumber\\
    \lesssim \,&\,  (\varepsilon \frac{1}{n^{1/2}})^{p}n^{p}\frac{1}{n}\sum_{k = 1}^{n}\frac{1}{n^{p}}\mathbb{E}_{\theta_{0}}\bigg[ \bigg( (\varepsilon \frac{1}{n^{1/2}})^{-1} \sup_{u \in [t_{k-1}^{n},t_{k}^{n}]}\big \lvert X_{u} - x_{u}\big \rvert\bigg)^{p} \bigg] \nonumber\\
     =\,&\, o_{p}((\varepsilon/\sqrt{n})^{p}). \nonumber
\end{align}
In the same way, we have that \eqref{6.2.9}. Finally, we have to show that for any $p \in \mathbb{N}$,
 \begin{align}
&\,\sup_{n,\varepsilon,\lambda}\mathbb{E}_{\theta_{0}}\bigg[\bigg(\varepsilon^{-1}\frac{1}{n} \sum_{k = 1}^{n} \bm{1}_{J_{k,0}^{c}}\bigg)^{p}\bigg] < \infty. \nonumber
\end{align}
By \Aref{a9}, our problem reduces to prove that for any integer $p$ with $p \geq 2$,
\begin{align}
    &\,\mathbb{E}_{\theta_{0}}\bigg[\bigg(\frac{1}{\lambda}\sum_{k = 1}^{n}\bm{1}_{J_{k,0}^{c}} \bigg)^{p} \bigg] < \infty.\label{j1_j2_moment}
\end{align}
We show \eqref{j1_j2_moment} applying Mathematical induction. For $p = 2$, we have
\begin{align}
    \mathbb{E}_{\theta_{0}}\bigg[\bigg(\frac{1}{\lambda}\sum_{k = 1}^{n}\bm{1}_{J_{k,0}^{c}} \bigg)^{2} \bigg] &\,\lesssim  \mathbb{E}_{\theta_{0}}\bigg[\sum_{k = 1}^{n}\bigg\{\frac{1}{\lambda}\bm{1}_{J_{k,0}^{c}} - \mathbb{P}(J_{k,0}^{c}\mid \mathcal{F}_{t_{k-1}^{n}}) \bigg\}^{2}\bigg] + O(1) < \infty. \nonumber
\end{align}
 For $p \geq 3$, suppose that
 \begin{align}
 &\,\sup_{n,\varepsilon,\lambda}\mathbb{E}_{\theta_{0}}\bigg[\bigg \lvert \frac{1}{\lambda}\sum_{k = 1}^{n} \bm{1}_{J_{k,0}^{n}}\bigg \rvert^{p-1} \bigg] < \infty. \nonumber 
 \end{align}
 By Burkholder and Davis inequality and since $p \geq 3$, we obtain that
 \begin{align}
&\sup_{n,\varepsilon,\lambda}\mathbb{E}_{\theta_{0}}\bigg[\bigg(\frac{1}{\lambda}\sum_{k = 1}^{n}\bm{1}_{J_{k,0}^{c}} \bigg)^{p} \bigg] \nonumber\\
    &\,\lesssim \sup_{n,\varepsilon,\lambda}\mathbb{E}_{\theta_{0}}\bigg[\bigg(\frac{1}{\lambda^{2}}\sum_{k = 1}^{n}\bm{1}_{J_{k,0}^{c}} \bigg)^{p/2} \bigg] + \sup_{n,\varepsilon,\lambda}\bigg( \frac{1}{\lambda}\sum_{k = 1}^{n}\mathbb{P}_{\theta_{0}}(J_{k,0}^{n} \mid \mathcal{F}_{t_{k-1}^{n}})\bigg)^{p} + \sup_{n,\varepsilon,\lambda}\bigg( \frac{1}{\lambda}\sum_{k = 1}^{n}(1 - \exp\{\lambda/n\})\bigg)^{p} \nonumber\\
    &\,\lesssim \sup_{n,\varepsilon,\lambda}\frac{1}{\lambda^{p}}\mathbb{E}_{\theta_{0}}\bigg[ \bigg(\sum_{k = 1}^{n}\bm{1}_{J_{k,0}^{c}}\bigg)^{p-1}\bigg] + O(1) + O(1) < \infty. \nonumber
 \end{align}
Thus, we obtain \eqref{j1_j2_moment}. 

We show Lemma \ref{moment_INEQ_diffusion} (3). As in Lemma \ref{approx_for_conditional_e} (1) and (2),
it suffices to show that for any $p \in \mathbb{N}$,
\begin{align}
&\,\sup_{(\mu,\sigma)}\sup_{n,\varepsilon,\lambda}\mathbb{E}_{\theta_{0}}\bigg[ \bigg\lvert \varepsilon^{-1}\sum_{k = 1}^{n}f(X_{t_{k-1}^{n}}^{\theta,\varepsilon},\theta)\tilde{\eta}_{k}^{n}(\mu_{0}) \bigg\rvert^{p} \bigg] <\infty,\label{6.2.13}\\
&\,\sup_{(\mu,\sigma)}\sup_{n,\varepsilon,\lambda}\mathbb{E}_{\theta_{0}}\bigg[ \bigg\lvert \varepsilon^{-1}\sum_{k = 1}^{n}\partial_{\theta}f(X_{t_{k-1}^{n}}^{\theta,\varepsilon},\theta)\tilde{\eta}_{k}^{n}(\mu_{0})\bigg \rvert^{p}\bigg] < \infty. \label{6.2.14}
\end{align}
We have that for any $p \in \mathbb{N}$,
\begin{align}
    &\mathbb{E}_{\theta_{0}}\bigg[\bigg 
    \lvert \Delta_{k}^{n}X^{\theta,\varepsilon}\bm{1}_{J_{k,0}^{n}} 
    - \mathbb{E}_{\theta_{0}}\bigg[ \Delta_{k}^{n}X^{\theta,\varepsilon}\bm{1}_{J_{k,0}^{n}}
     \mid \mathcal{F}_{t_{k-1}^{n}}\bigg] \bigg \rvert^{p} \bigg]\nonumber\\
     \lesssim&\,\frac{1}{n^{p-1}} \int_{t_{k-1}^{n}}^{t_{k}^{n}}\mathbb{E}_{\theta_{0}}\big[ \big\lvert a(X_{s}^{\theta,\varepsilon},\mu_{0}) - \mathbb{E}_{\theta_{0}}\big[a(X_{s}^{\theta,\varepsilon},\mu_{0}) 
    \mid \mathcal{F}_{t_{k-1}^{n}} \big \rvert^{p}\big] \big]ds \nonumber\\
    +&\,\frac{\varepsilon^{p}}{n^{p/2-1}}\int_{t_{k-1}^{n}}^{t_{k}^{n}}\mathbb{E}_{\theta_{0}}\bigg[ \lvert b(X_{s}^{\theta,\varepsilon},\sigma_{0})\rvert^{p} \bigg]ds \nonumber\\
    =&\,O(1/n^{p}) + O((\varepsilon/\sqrt{n})^{p}). \label{6.2.15}
\end{align}
Since 
\begin{align}
    \tilde{\eta}_{k}^{n}(\mu_{0}) =& \Delta_{k}^{n}X^{\theta,\varepsilon} \bm{1}_{J_{k,0}^{n}}
     - \mathbb{E}_{\theta_{0}}\bigg[\Delta_{k}^{n}X^{\theta,\varepsilon} \bm{1}_{J_{k,0}^{n}}\mid
      \mathcal{F}_{t_{k-1}^{n}}\bigg]  \nonumber\\
     +& \mathbb{E}_{\theta_{0}}\bigg[\Delta_{k}^{n}X^{\theta,\varepsilon} \bm{1}_{J_{k,0}^{n}}\mid \mathcal{F}_{t_{k-1}^{n}}\bigg] - \frac{1}{n}a(X_{t_{k-1}^{n}}^{\theta,\varepsilon},\mu_{0})\bm{1}_{J_{k,0}^{n}} \nonumber\\
    =:& N_{k}^{n} + R(X_{t_{k-1}^{n}}^{\theta,\varepsilon},\frac{1}{n^{2}}), \nonumber 
\end{align}
and from Lemma \ref{approx_for_conditional_e} (1) and \Aref{a9}, we find that
 for any $p \in \mathbb{N}$,
\begin{align}
&\,\sup_{\theta}\mathbb{E}_{\theta_{0}}\bigg[\bigg(\varepsilon^{-1}\bigg\lvert \sum_{k = 1}^{n}f(X_{t_{k-1}^{n}}^{\theta,\varepsilon},\theta)N_{k}^{n}
     \bigg\rvert\bigg)^{p} \bigg]\nonumber\\
     &\, \lesssim \varepsilon^{p}\sup_{\theta}\mathbb{E}_{\theta_{0}}\bigg[\bigg( \sum_{k = 1}^{n}\bigg\lvert f(X_{t_{k-1}^{n}}^{\theta,\varepsilon},\theta)N_{k}^{n} \bigg\rvert^{2} \bigg)^{p/2}\bigg] \nonumber\\
    &\,\lesssim \varepsilon^{-p} \sup_{\theta} n^{p/2}\frac{1}{n}
    \sum_{k = 1}^{n}\mathbb{E}_{\theta_{0}}
    \bigg[\bigg \lvert f(X_{t_{k-1}^{n}}^{\theta,\varepsilon},\theta) \bigg \rvert^{p} \bigg
     \lvert N_{k}^{n} \bigg \rvert^{p} \bigg] \nonumber\\
    &\,\lesssim \varepsilon^{-p} n^{p/2}\frac{1}{n}\sum_{k = 1}^{n}
    \sup_{\theta}\mathbb{E}_{\theta_{0}}\bigg[\bigg\lvert f(X_{t_{k-1}^{n}}^{\theta,\varepsilon},\theta) 
    \bigg\rvert^{2p}\bigg]^{1/2}
     \mathbb{E}_{\theta_{0}}\bigg[\bigg\lvert N_{k}^{n} \bigg\rvert^{2p} \bigg]^{1/2} \nonumber\\
    &\,= \varepsilon^{-p}n^{p/2}O(\varepsilon^{p}/n^{p/2}) = O(1). \nonumber
\end{align}
Similarly we obtain that
\begin{align}
&\mathbb{E}_{\theta_{0}}\bigg[\bigg(\varepsilon^{-1}\bigg\lvert \sum_{k = 1}^{n}\partial_{\theta}f(X_{t_{k-1}^{n}}^{\theta,\varepsilon},\theta)N_{k}^{n} \bigg\rvert \bigg)^{p}\bigg] = \varepsilon^{-p}n^{p/2}O(\varepsilon^{p}/{n}^{p/2}) = O(1) .\nonumber
\end{align}
On the otherhand, from Lemma \ref{approx_for_conditional_e} (1) and \Aref{a9}, we have that for any $p \in \mathbb{N}$,
\begin{align}  
    &\sup_{n,\varepsilon,\lambda}\mathbb{E}_{\theta_{0}}\bigg[\bigg(\varepsilon^{-1} \bigg \lvert \sum_{k = 1}^{n}f(X_{t_{k-1}^{n}}^{\theta,\varepsilon},\theta)N_{k}^{n} \bigg \rvert \bigg)^{p} \bigg] < \infty. \nonumber
\end{align}
and
\begin{align}
&\,\sup_{n,\varepsilon,\lambda}\mathbb{E}_{\theta_{0}}\bigg[\bigg(\sup_{\theta}\varepsilon^{-1}\bigg \lvert \sum_{k = 1}^{n}f(X_{t_{k-1}^{n}}^{\theta,\varepsilon},\theta)R(X_{t_{t_{k-1}^{n}}},\frac{1}{n^{2}}) \bigg \rvert \bigg)^{p}\bigg] < \infty ,\nonumber
\end{align}
which implies \eqref{6.2.13}. As in the proof of Lemma \eqref{6.2.13}, we can prove \eqref{6.2.14}.

Concerning Lemma \ref{moment_INEQ_diffusion}, we easily obtain that
\begin{align}
&\mathbb{E}_{\theta_{0}}\bigg[\bigg(\sup_{\theta}\varepsilon^{-2}\bigg \lvert \sum_{k = 1}^{n} 
 f(X_{t_{k-1}^{n}}^{\theta,\varepsilon},\theta)\tilde{\eta}_{k}^{n}(\mu) \bigg\rvert \bigg)^{p}\bigg] \nonumber\\
 \leq&  \mathbb{E}_{\theta_{0}}\bigg[\bigg(\sup_{\theta}\varepsilon^{-2}\bigg \lvert \sum_{k = 1}^{n} 
 f(X_{t_{k-1}^{n}}^{\theta,\varepsilon},\theta)\tilde{\eta}_{k}^{n}(\mu_{0}) \bigg\rvert \bigg)^{p}\bigg] \nonumber\\
 +&\mathbb{E}_{\theta_{0}}\bigg[\bigg(\sup_{\theta}\bigg \lvert \sum_{k = 1}^{n} 
 \frac{1}{n}f(X_{t_{k-1}^{n}}^{\theta,\varepsilon},\theta)b^{2}(X_{t_{k-1}^{n}}^{\theta,\varepsilon},\sigma_{0}) \bigg\rvert \bigg)^{p}\bigg] \nonumber\\
 +&\mathbb{E}_{\theta_{0}}\bigg[\bigg(\sup_{\theta}\bigg \lvert \sum_{k = 1}^{n}\frac{1}{n}\big(a(X_{t_{k-1}^{n}},\mu) - a(X_{t_{k-1}^{n}}^{\theta,\varepsilon},\mu_{0})\big)\bm{1}_{\tilde{J}_{k,0}^{n}}\rvert \bigg)^{p}\bigg] \nonumber
\end{align}
From Lemma \ref{approx_for_conditional_e} (1)-(3) and By \Aref{a9} and $f \in C_{\uparrow}^{1,1}(\mathbb{R} \times \overline{\Theta};\mathbb{R})$, we obtain Lemma \ref{moment_INEQ_diffusion} (4).

We show Lemma \ref{moment_INEQ_diffusion} (5). As in the proof of Lemma 3 in Kaino and Uchida \cite{4}, we have that
    \begin{align}
        &\tilde{\eta}_{k}^{n}(\mu_{0})^{2} - \frac{\varepsilon^{2}}{n}
        b^{2}(X_{t_{k-1}^{n}}^{\theta,\varepsilon},\sigma_{0})\bm{1}_{J_{k,0}^{n}} \nonumber\\
        =&\bigg( \Delta_{k}^{n}X^{\theta,\varepsilon} \bm{1}_{J_{k,0}^{n}} - \mathbb{E}_{\theta_{0}}\bigg[\Delta_{k}^{n}X^{\theta,\varepsilon}\bm{1}_{J_{k,0}^{n}} \mid \mathcal{F}_{t_{k-1}^{n}}\bigg]\bigg)^{2} - \mathrm{Var}\bigg(\Delta_{k}^{n}X^{\theta,\varepsilon}\bm{1}_{J_{k,0}^{n}} \mid \mathcal{F}_{t_{k-1}^{n}}\bigg)\nonumber\\
        +& \mathrm{Var}\bigg(\Delta_{k}^{n}X^{\theta,\varepsilon}\bm{1}_{J_{k,0}^{n}} \mid \mathcal{F}_{t_{k-1}^{n}}\bigg) - \frac{\varepsilon^{2}}{n}b^{2}(X_{t_{k-1}^{n}}^{\theta,\varepsilon},\sigma_{0})\bm{1}_{J_{k,0}^{n}} \nonumber\\
        +& 2 R(X_{t_{k-1}^{n}}^{\theta,\varepsilon},\frac{1}{n^{2}})\bigg(\Delta_{k}^{n}X^{\theta,\varepsilon}\bm{1}_{J_{k,0}^{n}} - \mathbb{E}_{\theta_{0}}\bigg[\Delta_{k}^{n}X^{\theta,\varepsilon}\bm{1}_{J_{k,0}^{n}} \mid \mathcal{F}_{t_{k-1}^{n}}\bigg] \bigg) \nonumber\\
        +& R(X_{t_{k-1}^{n}}^{\theta,\varepsilon},\frac{1}{n^{4}}), \nonumber
    \end{align}
    and we find that
    \begin{align}
    &\,\mathrm{Var}(\Delta_{k}^{n}X^{\theta,\varepsilon} \bm{1}_{J_{k,0}^{n}} \mid \mathcal{F}_{t_{k-1}^{n}})
    - \frac{\varepsilon^{2}}{n}b^{2}(X_{t_{k-1}^{n}}^{\theta,\varepsilon},\sigma_{0})\bm{1}_{J_{k,0}^{n}}\nonumber\\
    =&\,\mathbb{E}_{\theta_{0}}\big[\tilde{\eta}_{k}^{n}(\mu_{0})^{2} \mid \mathcal{F}_{t_{k-1}^{n}}\big] -\frac{\varepsilon^{2}}{n}b^{2}(X_{t_{k-1}^{n}}^{\theta,\varepsilon},\sigma_{0})\bm{1}_{J_{k,0}^{n}} -\mathbb{E}_{\theta_{0}}\big[\tilde{\eta}_{k}^{n}(\mu_{0})\mid \mathcal{F}_{t_{k-1}^{n}}\big]^2 \nonumber\\
        -&\, 2 \frac{1}{n}a(X_{t_{k-1}^{n}}^{\theta,\varepsilon},\mu_{0})\times \mathbb{E}_{\theta_{0}}
        \big[\big(\tilde{\eta}_{k}^{n}(\mu_{0}) - 
        \mathbb{E}_{\theta_{0}}\big[\tilde{\eta}_{k}^{n}(\mu_{0})
         \mid \mathcal{F}_{t_{k-1}^{n}}\big]\big)
         \big(\bm{1}_{J_{k,0}^{n}} - 
         \mathbb{P}(J_{k,0}^{n} \mid \mathcal{F}_{t_{k-1}^{n}})\big) 
         \mid \mathcal{F}_{t_{k-1}^{n}} \big] \nonumber\\
         +&\,\frac{1}{n^{2}}a^{2}(X_{t_{k-1}^{n}}^{\theta,\varepsilon},\mu_{0})\mathrm{Var}
         \big(\bm{1}_{J_{k,0}^{n}} \mid \mathcal{F}_{t_{k-1}^{n}}\big), \nonumber\\
         =&\, \mathbb{E}_{\theta_{0}}\bigg[\tilde{\eta}_{k}^{n}(\mu_{0})^{2} - 
         \frac{\varepsilon^{2}}{n}b^{2}(X_{t_{k-1}^{n}},\sigma_{0})\bm{1}_{J_{k,0}^{n}} \mid \mathcal{F}_{t_{k-1}^{n}}\bigg] 
          \nonumber\\
         +& \frac{\varepsilon^{2}}{n}b^{2}(X_{t_{k-1}^{n}}^{\theta,\varepsilon},\sigma_{0}) \bigg(\bm{1}_{J_{k,0}^{n}} - \mathbb{P}(J_{k,0}^{n} \mid \mathcal{F}_{t_{k-1}^{n}})\bigg) + R(X_{t_{k-1}^{n}}^{\theta,\varepsilon},\frac{\varepsilon^{2}}{n^{3/2}} + \frac{1}{n^{5/2}}). \nonumber
\end{align}
Therefore, we obtain 
\begin{align}
    &\,\tilde{\eta}_{k}^{n}(\mu_{0})^{2} - \varepsilon^{2}\frac{1}{n}b^{2}(X_{t_{k-1}^{n}}^{\theta,\varepsilon},\sigma_{0}) \bm{1}_{J_{k,0}^{n}} 
    = M_{k}^{(1)} + M_{k}^{(2)} + R(X_{t_{k-1}^{n}}^{\theta,\varepsilon},\frac{\varepsilon^{2}}{n^{3/2}} \lor \frac{1}{n^{5/2}} ),\nonumber 
  \end{align}
  where 
  \begin{align}
    M_{k}^{(1)}&\, = \bigg(\Delta_{k}^{n}X^{\theta,\varepsilon}\bm{1}_{J_{k,0}^{n}} - \mathbb{E}_{\theta_{0}}\bigg[\Delta_{k}^{n}X^{\theta,\varepsilon}\bm{1}_{J_{k,0}^{n}} \mid \mathcal{F}_{t_{k-1}^{n}}\bigg]\bigg)^{2} - \mathrm{Var}\bigg(\Delta_{k}^{n}X^{\theta,\varepsilon} \bm{1}_{J_{k,0}^{n}} \mid \mathcal{F}_{t_{k-1}^{n}}\bigg),\nonumber\\
    M_{k}^{(2)}&\, = 2R(X_{t_{k-1}^{n}}^{\theta,\varepsilon},\frac{1}{n^{2}})\bigg\{ \Delta_{k}^{n}X^{\theta,\varepsilon} \bm{1}_{J_{k,0}^{n}} -  \mathbb{E}_{\theta_{0}}\bigg[\Delta_{k}^{n}X^{\theta,\varepsilon} \bm{1}_{J_{k,0}^{n}} \mid \mathcal{F}_{t_{k-1}^{n}}\bigg] \bigg\}. \nonumber
  \end{align}
Imitating the proof of Lemma 3 (1) in Kaino and Uchida \cite{4}, we obtain for any $p \in \mathbb{N}$
\begin{align}
&\mathbb{E}_{\theta_{0}}\bigg[\bigg(\sup_{\theta}\sqrt{n}\bigg \lvert \varepsilon^{-2}\sum_{k = 1}^{n}f(X_{t_{k-1}^{n}}^{\theta,\varepsilon},\theta)M_{1,k}^{n}(\theta) \bigg\rvert\bigg)^{p} \bigg] < \infty, \nonumber\\
    &\mathbb{E}_{\theta_{0}}\bigg[\bigg(\sup_{\theta}\sqrt{n}\bigg \lvert \varepsilon^{-2}\sum_{k = 1}^{n}f(X_{t_{k-1}^{n}}^{\theta,\varepsilon},\theta)M_{2,k}^{n}(\theta) \bigg\rvert\bigg)^{p} \bigg] < \infty, \nonumber\\
    &\mathbb{E}_{\theta_{0}}\bigg[\bigg(\sup_{\theta}\sqrt{n}\bigg \lvert \varepsilon^{-2}\sum_{k = 1}^{n}f(X_{t_{k-1}^{n}}^{\theta,\varepsilon},\theta)R(X_{t_{k-1}^{n}}^{\theta,\varepsilon},\frac{\varepsilon^{2}}{n^{3/2}} \lor \frac{1}{n^{5/2}} ) \bigg\rvert\bigg)^{p} \bigg] < \infty, \nonumber
\end{align} 
which implies Lemma \ref{moment_INEQ_diffusion} (5).

Concerning Lemma \ref{moment_INEQ_diffusion} (6), we obtain it as in the proof of Lemma \ref{moment_INEQ_diffusion} (4).
\end{proof}
\begin{proof}[Proof of Lemma \ref{moment_ineq_jump}]
    Let $S =\{S_{k}\}_{1 \leq k \leq n}$ be a 
    Martingale and $S_{k}$ be $\mathcal{F}_{t_{k-1}^{n}}$ measurable for any $n \in \mathbb{N}$ and $k = 1,2,\dots,n$.
     Define the martingale \,$\{\psi_{k}^{1}(S)\}_{1 \leq k \leq n}$ as follows:
     \begin{align}
        \psi_{k}^{1}(S)&:= S_{k}. \nonumber
     \end{align}
    For $l \in \mathbb{N}$ with $l \geq 2$, suppose that $\{\psi_{k}^{l-1}(S)\}_{k = 1}^{n}$ is defined and set
    $\{\psi_{k}^{l}(S)\}_{k = 1}^{n}$ by
        \begin{align}
        \psi_{k}^{l}(S) &:= \bigg(\psi_{k}^{l}(S) - \mathbb{E}_{\theta_{0}}[\psi_{k}^{l}(S)\mid \mathcal{F}_{t_{k-1}^{n}}] \bigg)^{2}. \nonumber
    \end{align}
    To show Lemma \ref{moment_ineq_jump}, it suffices to show that for any integer $l$ with $l \geq 0$,
    \begin{align}
&\sup_{n,\varepsilon,\lambda}\mathbb{E}_{\theta_{0}}\bigg[\bigg\lvert \frac{1}{\lambda}\sum_{k = 1}^{n}g_{k-1}\bigg( c(X_{t_{k-1}^{n}}^{\theta,\varepsilon},\alpha_{0})V_{\tau_{k}},\alpha\bigg)\bm{1}_{J_{k,1}^{n}}\bigg \rvert^{2^{l}}\bigg] < \infty . \label{jump_moment_g}
    \end{align}
    \normalsize
     As in the proof of proposition 3 in Ogihara and Yoshida 
\cite{11}, we obtain for any $l \in \mathbb{N}$
    \begin{align}
     &\,\bigg\lvert \mathbb{E}_{\theta_{0}}\bigg[\bigg(\frac{1}{\lambda}\sum_{k = 1}^{n}g_{k-1}\bigg(c(X_{t_{k-1}^{n}}^{\theta,\varepsilon},\alpha_{0})
     V_{\tau_{k}},\alpha\bigg)\bm{1}_{J_{k,1}^{n}} \bigg)^{2^{l}}\bigg] \bigg \rvert
    \nonumber \\
    \lesssim &\,
    \mathbb{E}_{\theta_{0}}\bigg[\bigg \lvert 
    \sum_{k = 1}^{n} \psi_{k}^{l + 1}\bigg(\frac{1}{\lambda}g_{k-1}\big(c(X_{t_{k-1}^{n}}^{\theta,\varepsilon},\alpha_{0})V_{\tau_{k}},\alpha\big)\bigg)
    \bigg \rvert \bigg] + \sum_{i = 1}^{l}\mathbb{E}_{\theta_{0}}
    \bigg[ \bigg \lvert 
    \sum_{k = 1}^{n} \psi_{k}^{i}\bigg(\frac{1}{\lambda}g_{k-1}\big(c(X_{t_{k-1}^{n}}^{\theta,\varepsilon},\alpha_{0})V_{\tau_{k}},\alpha\big)\bigg)
    \bigg \rvert\bigg] \nonumber
    \end{align}
and as in the proof of Lemma 6 and 7 in Ogihara and Yoshida 
\cite{11}, we have \eqref{jump_moment_g}, which immediately implies Lemma \ref{moment_ineq_jump}.
\end{proof}

\end{document}